\documentclass[british, a4paper,11pt]{article}

\usepackage[utf8]{inputenc}
\usepackage[british]{babel}

\usepackage{amsthm,amssymb,bbm,stmaryrd}

\usepackage{libertinus} 
\usepackage[T1]{fontenc}
\usepackage{libertinust1math} 
\usepackage[cal=euler, scr=rsfso]{mathalpha} 
\useosf
\usepackage{cancel}

\usepackage{microtype,soul}

\usepackage{numprint}
\usepackage{framed}

\usepackage[explicit]{titlesec}

\usepackage{todonotes}

\titleformat{\section}[block]{\normalfont\large\bfseries\boldmath\centering}{\raggedright\makebox[1em][l]{\thesection.}}{.25em}{#1}
\titleformat{\subsection}[block]{\normalfont\bfseries\boldmath\centering}{\raggedright\makebox[1em][l]{\thesubsection.}}{1em}{#1}
\titleformat{\subsubsection}[runin]{\normalfont\bfseries\boldmath}{\raggedright\makebox[1em][l]{\thesubsubsection.}}{1.5em}{#1.~\hbox{---}}

\renewenvironment{abstract}{%
\begin{center}
\begin{minipage}{.9\textwidth}\linespread{1.05}\selectfont\small
\makebox[5em][l]{\bfseries\abstractname.~\hbox{---}}
\normalfont}
{\par\vspace{1em}
\end{minipage}
\end{center}
}

\usepackage[a4paper,margin=2.5cm]{geometry}
\selectfont

\usepackage{tikz}
	\usetikzlibrary{calc}
	\usetikzlibrary{decorations.markings}
	\usetikzlibrary{arrows,patterns}

\usepackage[font=sf]{caption}
\usepackage{subcaption}

\usepackage{hyperref}
\hypersetup{
	colorlinks=true,
		citecolor=blue!60!black,
		linkcolor=red!60!black,
		urlcolor=green!40!black,
		filecolor=yellow!50!black,
	breaklinks=true,
	pdfpagemode=UseNone,
	bookmarksopen=false,
}

\usepackage{enumitem}

\numberwithin{equation}{section}

\newtheorem{thm}{\bfseries \upshape Theorem}[section]
\newtheorem{lem}[thm]{Lemma}
\newtheorem{prop}[thm]{Proposition}
\newtheorem{cor}[thm]{Corollary}

\theoremstyle{definition}
\newtheorem{rem}[thm]{Remark}

\renewcommand{\P}{\mathop{}\!\mathbb{P}}
\newcommand{\E}{\mathop{}\!\mathbb{E}}

\renewcommand\Pr[1]{\mathbb{P}\left(#1\right)}

\newcommand{\Var}{\mathrm{Var}}
\newcommand{\R}{\mathbb{R}}
\newcommand{\Z}{\mathbb{Z}}
\newcommand{\N}{\mathbb{N}}

\newcommand{\Ver}{\mathcal{V}}

\newcommand{\Loop}{\mathscr{L}}
\newcommand{\Loopd}{\mathsf{L}}
\newcommand{\Map}{M}

\DeclareMathOperator{\diminf}{\underline{\dim}}
\DeclareMathOperator{\dimsup}{\overline{\dim}}

\newcommand{\e}{\operatorname{e}} 

\renewcommand{\d}{\mathop{}\!\mathrm{d}}
\renewcommand{\i}{\operatorname{i}}

\newcommand{\bddelta}{\boldsymbol{\delta}}
\newcommand{\bdgamma}{\boldsymbol{\gamma}}
\newcommand{\bdeta}{\boldsymbol{\eta}}

\newcommand{\ind}[1]{\mathbf{1}_{\{#1\}}}

\newcommand{\cv}[1][n]{\enskip\mathop{\longrightarrow}^{}_{#1 \to \infty}\enskip}
\newcommand{\cvloi}[1][n]{\enskip\mathop{\longrightarrow}^{(d)}_{#1 \to \infty}\enskip}

\newcommand{\cvproba}[1][n]{\enskip\mathop{\longrightarrow}^{\P}_{#1 \to \infty}\enskip}
\newcommand{\eqloi}[1][n]{\enskip\mathop{=}^{(d)}_{}\enskip}

\newcommand{\exc}{\mathrm{ex}}
\newcommand{\br}{\mathrm{br}}

\usepackage{mathtools}

\DeclarePairedDelimiter\floor{\lfloor}{\rfloor}

\let\originalleft\left
\let\originalright\right
\renewcommand{\left}{\mathopen{}\mathclose\bgroup\originalleft}
\renewcommand{\right}{\aftergroup\egroup\originalright}

\DeclareSymbolFont{extraup}{U}{zavm}{m}{n}
\DeclareMathSymbol{\vardspade}{\mathalpha}{extraup}{81}
\DeclareMathSymbol{\varheart}{\mathalpha}{extraup}{86}
\DeclareMathSymbol{\vardiamond}{\mathalpha}{extraup}{87}
\DeclareMathSymbol{\varclub}{\mathalpha}{extraup}{84}

\makeatletter
\renewcommand*{\@fnsymbol}[1]{\ensuremath{\ifcase#1\or \vardspade \or \varheart\or \vardiamond \or \varclub \or
  \mathsection\or \mathparagraph\or \|\or **\or \dagger\dagger
  \or \ddagger\ddagger \else\@ctrerr\fi}}
\makeatother

\author{
	Igor \textsc{Kortchemski}\thanks{Centre de Math\'ematiques Appliqu\'ees (CMAP), CNRS, \'Ecole polytechnique, Institut Polytechnique de Paris, 91120 Palaiseau, France.\hfill \href{mailto:igor.kortchemski@math.cnrs.fr}{\texttt{igor.kortchemski@math.cnrs.fr}}} 
\qquad\&\qquad
	Cyril \textsc{Marzouk}\thanks{Centre de Math\'ematiques Appliqu\'ees (CMAP), CNRS, \'Ecole polytechnique, Institut Polytechnique de Paris, 91120 Palaiseau, France.\hfill \href{mailto:cyril.marzouk@polytechnique.edu}{\texttt{cyril.marzouk@polytechnique.edu}}}
}

\title{Random L\'evy {L}ooptrees and L\'evy {M}aps}

\begin{document}

\maketitle

\begin{abstract}
What is the analogue of L\'evy processes for random surfaces? Motivated by scaling limits of random planar maps in random geometry, we introduce and study L\'evy looptrees and L\'evy maps. They are defined using excursions of general L\'evy processes with no negative jump and extend the known stable looptrees and stable maps, associated with stable processes. We compute in particular their fractal dimensions in terms of the upper and lower Blumenthal--Getoor exponents of the coding L\'evy process.
The case where the Lévy process is a stable process with a drift naturally appears in the context of stable-Boltzmann planar maps conditioned on having a fixed number of vertices and edges in a near-critical regime.
\end{abstract}

\begin{figure}[!ht]\centering
\begin{subfigure}[b]{0.49\textwidth}
\includegraphics[height=4.5cm]{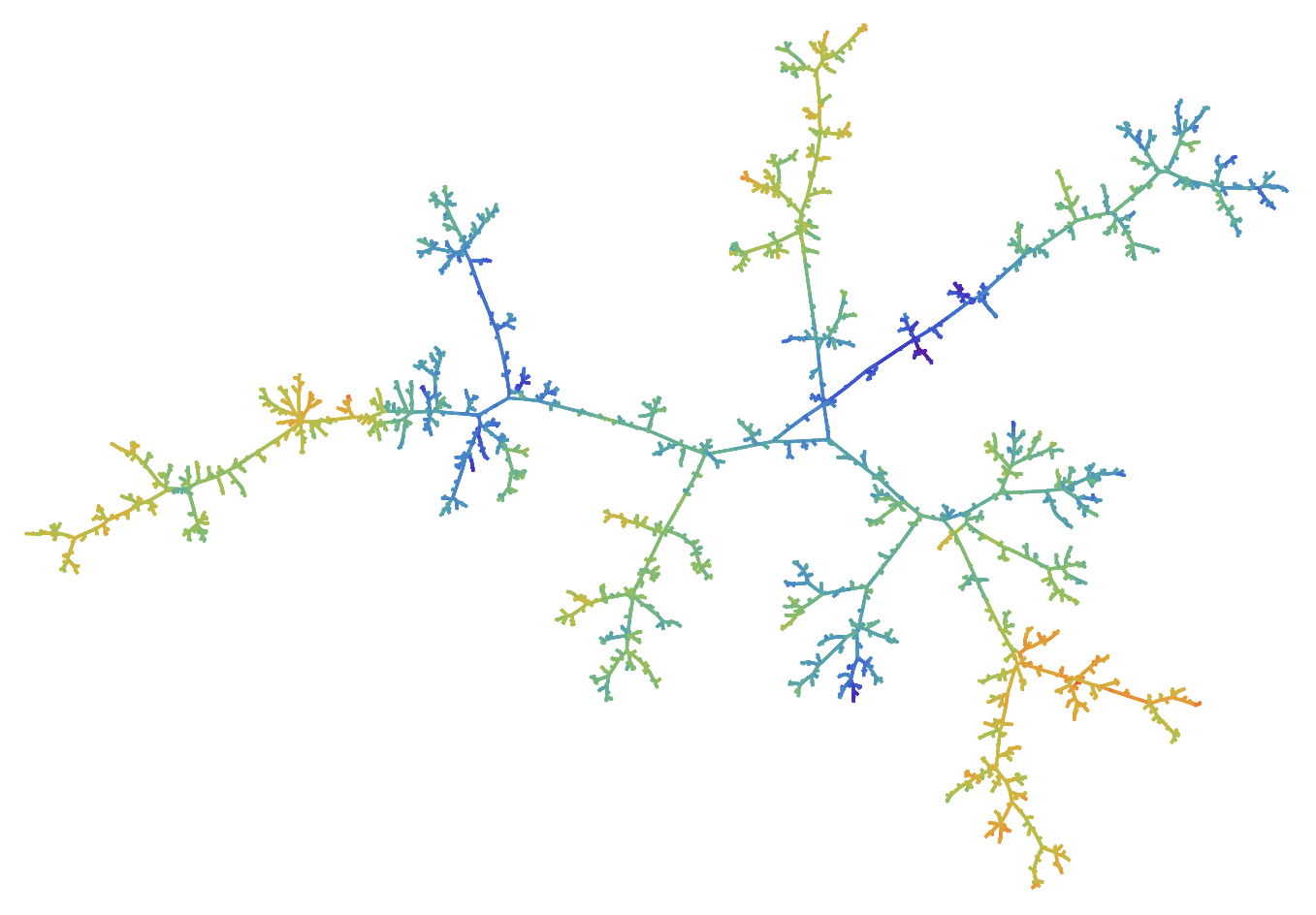}
\end{subfigure}
\begin{subfigure}[b]{0.49\textwidth}
\includegraphics[height=5cm, trim=1cm 2cm 3cm 1cm, clip]
{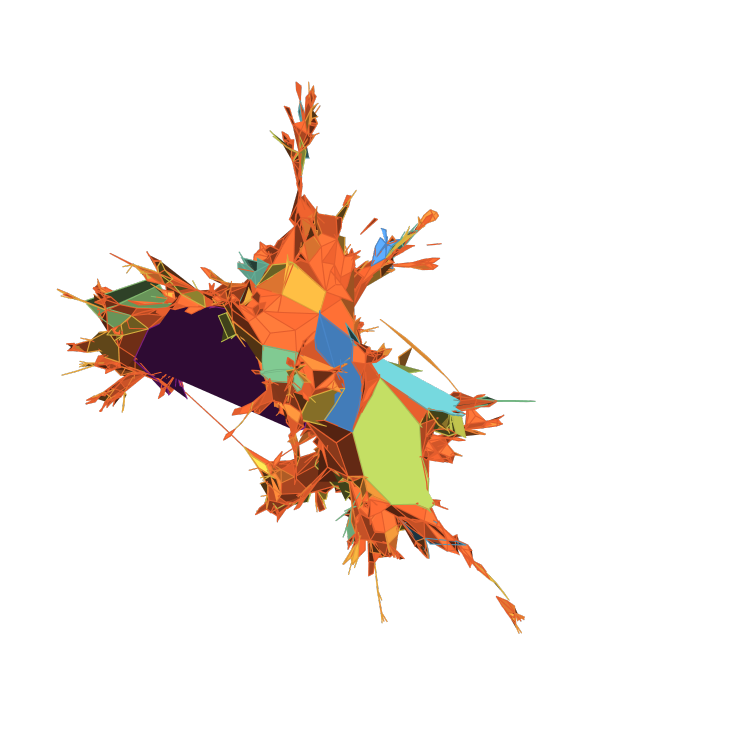}
\end{subfigure}
\begin{subfigure}[b]{0.49\textwidth}
\includegraphics[height=4.5cm]{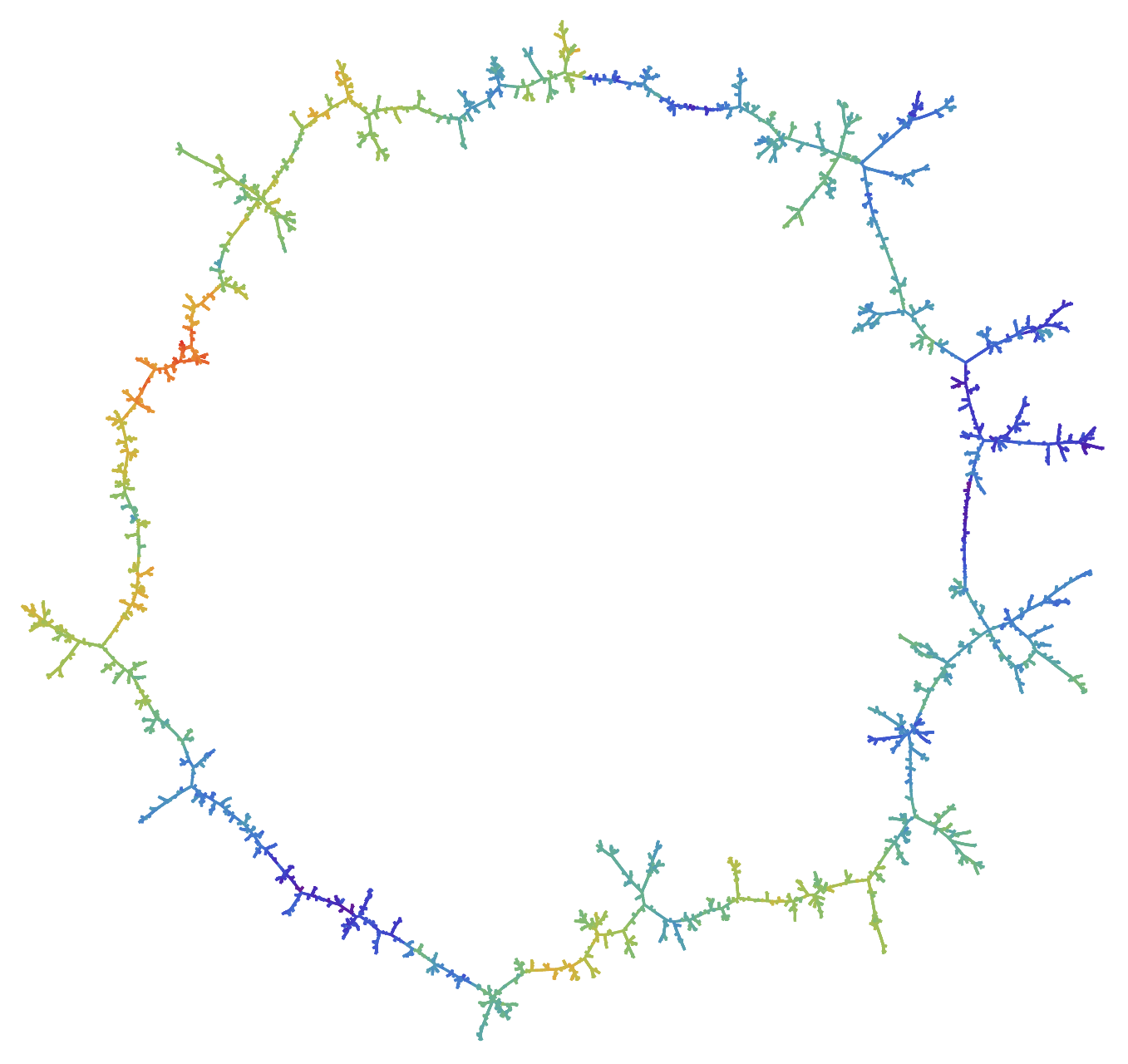}
\end{subfigure}
\begin{subfigure}[b]{0.49\textwidth}
\includegraphics[height=4.5cm, trim=6cm 6cm 0cm 2cm, clip]{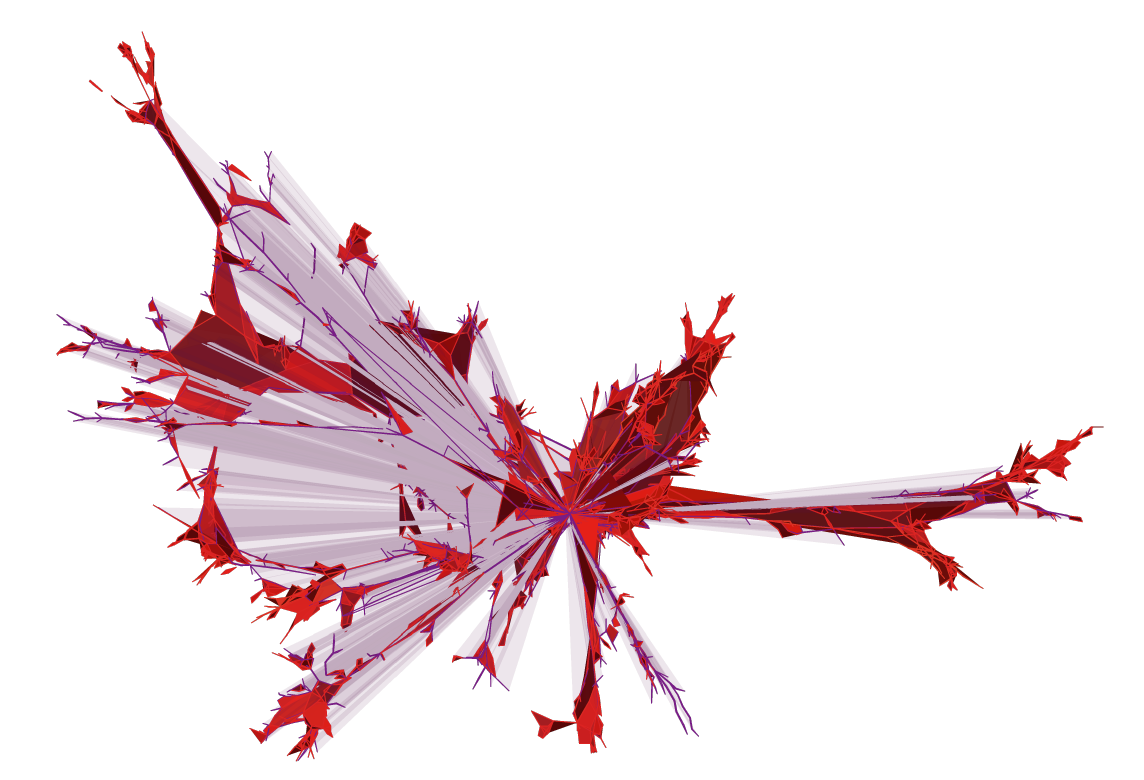}
\end{subfigure}
\caption{Left: two looptrees coded by (approximations of) an excursion of a stable process with index $3/2$ and with a drift $+7$ (top) and $-20$ (bottom); the colours represent Gaussian labels. Right: the associated plane maps embedded in three-dimensional space; the colours represent the degree of the faces.}
\label{fig:simus_looptrees_cartes}
\end{figure}

{\small
\tableofcontents
}

\section{Introduction}

In short, our main purpose is to extend the construction of stable looptrees and stable maps, which are random metric spaces built using stable L\'evy processes, to the setting of general L\'evy processes with no negative jumps.
We calculate their fractal dimensions and study their behaviour as the coding L\'evy processes vary.
Figure~\ref{fig:simus_looptrees_cartes} presents some simulations of such non stable objects.
In addition to the main results pertaining to random geometry, which are presented in this introduction, some of our tools involve new results for L\'evy processes, which are of independent interest.

\subsection{Motivation and literature}
\label{ssec:motiv}

\paragraph*{The Brownian sphere.}
In random geometry, the \emph{Brownian sphere} (also known as the Brownian map) is a random metric space almost surely homeomorphic to the two-dimensional sphere, whose study has attracted a lot of attention in the past twenty years. It has in particular been shown to be the universal scaling limit of many different models of random planar maps, see e.g.~\cite{LG13,Mie13,BLG13,BJM14,ABA17,ABA21,Abr16,CLG19,Mar19,KM23}. We recall that a planar map is the embedding without edge-crossing of a finite, connected multigraph on the sphere, viewed up to orientation-preserving homeomorphisms.
In two recent breakthroughs, it was also shown to appear at the scaling limit of models of random, non embedded, planar graphs~\cite{AFL23,Stu24}.
The models of maps in the previous references share the common feature that no face has a size which dominates the other ones, so they vanish in the limit after applying a suitable rescaling. In this sense, the Brownian sphere is the canonical model of spherical random surfaces.
It bears intimate connections with the Brownian excursion, Aldous' Continuum Random Tree (hereafter called the `Brownian tree'), and Le~Gall's Brownian snake, which are also important objects in probability.
It also relates to the theory of Liouville Quantum Gravity~\cite{MS20,MS21a,MS21b,MS21c}; let us refer to the 
expository papers~\cite{Mil18,Gwy20,GHS23,She23} for a gentle introduction to this topic.

\paragraph*{Stable maps.}
Parallel to this, some effort has been put into escaping the universality class of the Brownian sphere, motivated in part by statistical physics models on maps (see~\cite[Section~8]{LGM11}).
In this direction, Le~Gall \&~Miermont~\cite{LGM11} constructed a family of continuum models and proved that they are scaling limits, along subsequences, of Boltzmann random planar maps (defined in the next paragraph) in regimes where the degree of the faces have infinite variance, and whose behaviour is dictated by an exponent $\alpha \in (1,2)$. 
Very recently, Curien, Miermont \& Riera~\cite{CMR25} proved that the convergence in distribution holds without taking subsequences. They call these limits an $\alpha$-stable carpet when $\alpha \ge 3/2$ and an $\alpha$-stable gasket when $\alpha < 3/2$ because of topological differences between these two regimes. To simplify, we shall  call them \emph{$\alpha$-stable maps} in all regimes $\alpha \in (1,2)$.

Let us briefly recall the model of Boltzmann planar maps, which was introduced in~\cite{MM07}. As often, we restrict to bipartite maps, which are those whose faces all have even degree, and which turn out to be much simpler to study. We also consider pointed maps, meaning that one vertex is distinguished.
Given a sequence $(q_k)_{k \geq 1}$ of  nonnegative real numbers, one assigns to any finite bipartite pointed map a weight given by: 
\[\prod_{f \text{ face}} q_{\deg(f)/2},\]
where $\deg(f)$ is the degree of the face $f$.
One can then sample a pointed map with $n$ edges proportionally to its weight (assuming that the total weight of $n$-edged maps, which is a finite set, is non-zero).  This model is similar in spirit to the so-called model of \emph{simply generated trees}, which can be seen under mild assumptions as size-conditioned Bienaym\'e--Galton--Watson trees, see e.g.~the survey~\cite{Jan12} for details.

Roughly speaking, Le~Gall \&~Miermont~\cite{LGM11} identified a regime of weights for which  large degree faces are present and subsequential scaling limits  are built using decorated excursions of $\alpha$-stable L\'evy processes with no negative jumps, with $\alpha \in (1,2)$. An explicit sequence of such weights is for example given by~\cite[Section 6]{BC17}:
\begin{equation}
\label{eq:explicit}
q_{k}= c \kappa^{k-1} \frac{\Gamma(-1/2-\alpha+k)}{\Gamma(1/2+k)} \mathbf{1}_{k \geq 2}, 
\qquad \kappa= \frac{1}{4 \alpha+2}, 
\qquad c= \frac{-\sqrt{\pi}}{2\Gamma(1/2-\alpha)}.
\end{equation}

\paragraph*{Multiconditioned planar maps.}
In~\cite{KM23} we have recently considered planar maps conditioned both to have $n$ edges and $K_n$ vertices, and thus $n-K_n+2$ faces by Euler's formula. 
One motivation originated from predictions in~\cite{FG14} concerning the typical order of distances in such uniform random maps, which we have confirmed by showing a stronger scaling limit result, and also by connections with random hyperbolic geometry mentioned in the subsequent work~\cite{CKM22}. 
Another motivation was a large or moderate deviation question. Indeed, it can be shown that in the $\alpha$-stable map regime of~\cite{LGM11} as above, a map with $n$ edges has around $\theta n$ vertices for some constant $\theta \in (0,1)$ which depends on the weights $(q_k)_k$.
For example, $\theta=4\kappa$ for the particular weight sequence~\eqref{eq:explicit}.
It is then natural to study how the geometry changes when one forces the number of vertices to deviate from this typical behaviour.
One of the contributions of~\cite{KM23} was to show that in this $\alpha$-stable regime, multiconditioning enables new continuum maps to appear by finely tuning $K_{n}$.
To be concrete, in the case of~\eqref{eq:explicit}, this corresponds to the regime where $n^{-1/\alpha} (K_{n}- 4 \kappa n)$ has a finite limit as $ n \rightarrow \infty$.
In this regime, one obtains subsequential scaling limits which are coded by decorated excursions of $\alpha$-stable L\'evy processes with no negative jumps \emph{with a drift}, see  Corollary~\ref{cor:looptrees_maps_stable_drift} for a  precise statement and Figure~\ref{fig:simus_looptrees_cartes} for simulations.

\paragraph*{L\'evy maps.}
The objective of this present work is then twofold.
First we aim at defining more general continuum maps related to unit duration excursions of any L\'evy process with no negative jump. Indeed, as we have previously mentioned, the case of stable processes with a drift appears in a multiconditioned setting. More generally, we believe that the L\'evy maps that we introduce and construct here are the only possible scaling limits of Boltzmann maps, based on the fact that L\'evy processes are the only possible scaling limits of triangular arrays of random walks~\cite[Theorem 16.14]{Kal02}, that the continuous-state branching processes are the only possible scaling limits of discrete-time Bienaym\'e--Galton--Watson branching processes~\cite{Lam67}, and that the L\'evy trees are the only possible scaling limits of Bienaym\'e--Galton--Watson trees~\cite{DLG02}. It would also be interesting to investigate limits of gaskets of a large random $O(N)$ decorated triangulation in near critical regimes (see e.g.~the discussion in~\cite[Section 5.3]{CK15} which hints to the appearance of Lévy looptrees and maps).
We leave such questions for a future work, and here we focus  on the construction of these L\'evy maps, using ad hoc discrete models to pass to the limit.

\paragraph*{Fractal properties.}
We also express the fractal dimensions of L\'evy maps: Hausdorff, Minkowski, and packing dimensions, in terms of natural quantities in the underlying L\'evy process, namely the Blumenthal \&~Getoor exponents, extending results of Le~Gall \&~Miermont~\cite{LGM11} for $\alpha$-stable maps.
Intuitively, these dimensions quantify the `roughness' of these random metric spaces, and their identification is an important question in random geometry. Fractal dimensions of stable trees and  stable maps have been computed respectively in~\cite{HM04,LGM11}, and dimensions of general L\'evy trees have been computed in~\cite{DLG05}. The key difference with Le~Gall \&~Miermont's approach is that self-similarity is not available anymore, and the fact that the lower and upper Blumenthal-Getoor indices may differ significantly complicates the analysis.

\paragraph*{L\'evy looptrees.}
A useful tool to construct and study stable maps is the so-called \emph{stable looptrees}
introduced in~\cite{CK14}, also motivated by the study of percolation clusters on random maps~\cite{CK15}. Looptrees have been used in relation with maps~\cite{DMS21,MS21a,KR19,KR20,LG18,Ric18,SS19,BHS23}, but also for their own interest or in relation with other models~\cite{Arc20a,Arc21,AS23,BS15,CDKM15,CHK15,Kha22}.
Informally, stable looptrees are obtained by replacing branchpoints of the stable trees of~\cite{Duq03, DLG05} by `loops' and then gluing these loops along the tree structure. See Figure~\ref{fig:looptree_discret} for a representation of a discrete looptree and Figure~\ref{fig:simus_looptrees_cartes} for simulations of continuum looptrees.
Formally, stable looptrees are built from an excursion of a stable L\'evy process with no negative jumps. 
In this work, we also extend their construction to general L\'evy processes with no negative jumps and study their properties, such as their fractal dimensions.

\begin{figure}[!ht] \centering
\includegraphics[page=1, height=3.5cm]{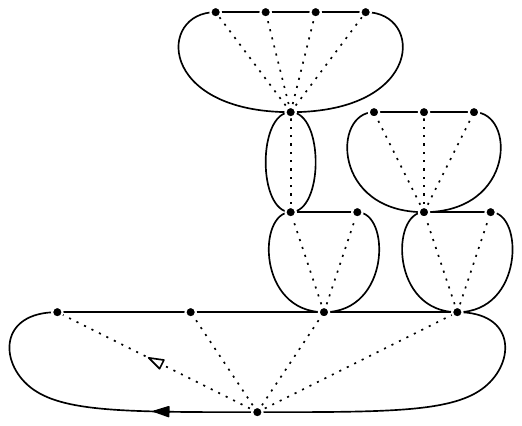}
\qquad
\includegraphics[page=2, height=3.5cm]{dessins}
\qquad
\includegraphics[page=3, height=3.5cm]{dessins}
\caption{Left: a plane tree in dotted lines and its looptree version. Right: a slight variation, used here, obtained by merging each internal vertex of the tree with its right-most offspring, as shown by the dashed lines in the middle.}
\label{fig:looptree_discret}
\end{figure}

\subsection{Main results}
\label{ssec:intro_resultats}

Let us state right away our main theorems on the dimensions of L\'evy looptrees and maps, while deferring the precise definition of these objects to the next section.

\subsubsection*{Background on L\'evy processes}

Throughout this work, we let $X = (X_t)_{t \ge 0}$ be a L\'evy process with no negative jump and with paths of infinite variation. We refer to~\cite{Ber96, DLG02} for details on L\'evy processes and further discussions related to this work.
Let us denote by $\psi \colon \lambda \in [0, \infty) \mapsto \log \E[\e^{-\lambda X_1}]$ its Laplace exponent; according to the L\'evy--Khintchine formula, the latter takes the form:
\[\psi(\lambda) = -d \lambda + \beta \lambda^{2}+ \int_0^\infty (\e^{-\lambda r} -1 + \lambda r) \pi(\d r),\]
where $d \in \R$ is the drift coefficient, $\beta \ge 0$ is the Gaussian parameter, and the L\'evy measure $\pi$ satisfies $\int_0^\infty (r \wedge r^2) \pi(\d r) < \infty$. From the condition of infinite variation paths, 
if $\beta=0$ then necessarily $\int_0^1 r \pi(\d r) = \infty$. 
The case of stable L\'evy processes corresponds to $\psi(\lambda) = \lambda^{\alpha}$ with $1 < \alpha < 2$, for which $d=\beta=0$ and $\pi(\d r) = \alpha(\alpha-1) \Gamma(2-\alpha)^{-1} r^{-\alpha-1} \d r$.

An idea to deal with general L\'evy processes is to compare them with stable processes. To this end, recall the lower and upper exponents of $\psi$ at infinity introduced by Blumenthal \&~Getoor~\cite{BG61}:
\begin{equation}\label{eq:exposants_BG}
\begin{aligned}
\bdgamma &\coloneqq \sup\Bigl\{c \geq 0 \colon \lim_{\lambda \to \infty} \lambda^{-c} \psi(\lambda) = \infty\Bigr\},
\\
\bdeta &\coloneqq \inf\Bigl\{c \geq 0 \colon \lim_{\lambda \to \infty} \lambda^{-c} \psi(\lambda) = 0\Bigr\}.
\end{aligned}
\end{equation}
We use in this work the notation from~\cite{DLG05} and the exponents $\bdgamma$ and $\bdeta$ actually correspond to $\beta''$ and $\beta$ respectively in~\cite{BG61}. 
Let us mention that $\bdeta$ also relates to the L\'evy measure $\pi$ by $\bdeta= \inf \{c \geq 0\colon \int_{0}^{1} x^{c} \pi(\d x) <\infty\}$.
We always have $\lim_{\lambda \to \infty} \lambda^{-2} \psi(\lambda) = \beta$, and since we assume that $X$ has paths with infinite variation, then $\lim_{\lambda \to \infty} \lambda^{-1} \psi(\lambda) = \infty$. 
In particular $1 \leq \bdgamma \leq \bdeta \leq 2$; the two exponents coincide (with $\alpha$) for $\alpha$-stable processes and more generally when $\psi$ varies regularly at infinity, but they differ in general 
and all pairs of values $1 \leq \bdgamma \leq \bdeta \leq 2$ are possible.

We shall assume throughout this paper the integrability condition:
for every $t > 0$,
\begin{equation}\label{eq:condition_integrale}
\int_{\R} |\E[\e^{i u X_t}]| \d u < \infty
.\end{equation}

As observed by Kallenberg~\cite[Section~5]{Kal81}, this condition holds as soon as either $\beta>0$ or $u^{-2} |\log u|^{-1} \int_0^u r^2 \pi(\d r) \to \infty$ as $u\to 0$. 
Observe that this last convergence requires only slightly 
more than $\pi$ to be infinite: $\pi$ is infinite as soon as $u^{-2}  \int_0^u r^2 \pi(\d r) \to \infty$, and is finite as soon as $u^{-2} |\log u|^{1+\varepsilon} \int_0^u r^2 \pi(\d r) \to 0$.

Condition~\eqref{eq:condition_integrale} also appears in~\cite{Kni96,UB14}, and enables us to define bridges and excursions of the process with a fixed duration. Precisely, by inverse Fourier transform, under~\eqref{eq:condition_integrale}, the random variable $X_t$ admits a continuous density for every $t>0$ (even jointly continuous in time and space).
These transition densities can then be used to define a regular version $X^{\br}$ of $(X_t)_{t \in [0,1]}$ conditioned on $X_0 = X_1 = 0$, which we call the bridge version of $X$.
Next, by exchanging the parts prior and after the first minimum of $X^{\br}$, the so-called Vervaat transform allows to define an excursion $X^{\exc}$, which is informally a version of $(X_t)_{t \in [0,1]}$ conditioned on $X_0 = X_1 = 0$ and $X_t > 0$ for every $t \in (0,1)$.
Let us refer to Section~\ref{ssec:Levy_looptrees_labels} and references therein for a few details.
Without the condition~\eqref{eq:condition_integrale}, the construction of the continuum looptrees and maps as well as the calculation of their dimensions can be carried out in the setting of unconditioned L\'evy processes, or under the It\=o excursion measure.

We can construct a looptree $\Loop(X^{\exc})$ from the excursion path $X^{\exc}$ in a way that somehow extends the construction of a tree from a continuous excursion via its contour exploration. 
Very informally, we see each positive jump, say $\Delta X^{\exc}_t>0$, as a vertical line segment with this length, turned into a cycle by identifying its two extremities. We then glue these cycles together by attaching the bottom $(t, X^{\exc}_{t-})$ of such a segment to the first point we meet on another vertical segment when going horizontally to the left, starting from $(t,X^{\exc}_{t-})$, at the last time $s<t$ such that $X^{\exc}_{s-} < X^{\exc}_{t-} < X^{\exc}_s$. In reality, due to the infinite variation paths, two macroscopic cycles never touch each other, and when the Gaussian parameter $\beta$ is nonzero, the looptree contains infinitesimal `tree parts'.
In the case when $X^{\exc}$ is the Brownian excursion, the looptree $\Loop(X^{\exc})$ actually has no cycle and reduces to a scaled version of the usual Brownian tree.
The formal definition of $\Loop(X^{\exc})$ is given in Section~\ref{ssec:Levy_looptrees_labels}, let us only mention that it is a metric measured space, which is the quotient of the interval $[0,1]$ by a continuous pseudo-distance defined from $X^{\exc}$.

\subsubsection*{Fractal dimensions}

Let us denote by $\dim_{H}$ the Hausdorff dimension, by $\dim_{p}$ the packing dimension, and by $\diminf$ and $\dimsup$ respectively the lower and upper Minkowski dimensions (sometimes also called box counting dimensions). We refer to~\cite[Chapter~4 and~5]{Mat95} for definitions and basic properties of these dimensions. Recall in particular that the Minkowski dimensions of a metric space $(E,d)$ are defined as:
\[\diminf E = \liminf_{\varepsilon \downarrow 0} \frac{\log N(\varepsilon)}{\log 1/\varepsilon}
\qquad\text{and}\qquad
\dimsup E = \limsup_{\varepsilon \downarrow 0} \frac{\log N(\varepsilon)}{\log 1/\varepsilon},\]
where $N(\varepsilon) \coloneqq \min\{k \geq 1 \colon \exists x_{1}, \dots, x_{k} \in E \text{ such that } E \subset \bigcap_{i=1}^{k} B(x_{i}, \varepsilon)\}$ is the minimal number of $\varepsilon$-balls required to cover the space.

\begin{thm}\label{thm:dimensions_fractales_Looptrees}
Almost surely, it holds
\[\dim_{H} \Loop(X^{\exc}) = \diminf \Loop(X^{\exc}) = \bdgamma
\qquad\text{and}\qquad
\dim_{p} \Loop(X^{\exc}) = \dimsup \Loop(X^{\exc}) = \bdeta,\]
where $1 \leq \bdgamma \leq \bdeta \leq 2$ are the exponents defined in~\eqref{eq:exposants_BG}.
\end{thm}

This theorem extends~\cite[Theorem~1.1]{CK14} in the case of $\alpha$-stable L\'evy processes for which $\bdgamma=\bdeta=\alpha$ for each $\alpha \in (1,2)$.
In the Brownian case $\alpha=2$, we also recover the dimension $2=\bdgamma=\bdeta$ of the Brownian tree.
It is also consistent with~\cite[Theorem~2.3]{BR22} who considered another model, related to the so-called inhomogeneous continuum random trees.

Turning to L\'evy maps, one could construct them directly from the process $X^{\exc}$ together with extra randomness given by a Gaussian field on the looptree $\Loop(X^{\exc})$, using the analogue of the formula of $D^\ast$ in~\cite{LG07, LG13, Mie13} to define the Brownian sphere.
However this is not what we will do and, as for the pioneer works in the Brownian and stable regimes~\cite{LG07, LGM11}, we will instead use approximations by rescaled discrete maps.
The reason is twofold: first our motivation to study these objects in the first place is to eventually prove scaling limit results, and second this will allow us to use the discrete objects to obtain some estimates that seem harder to obtain directly in the continuum world.

We shall give more details in Section~\ref{sec:cartes_looptrees} below, but for the benefit of the reader let us here give a first rather informal sense of these L\'evy maps. A celebrated bijection from~\cite{BDG04} shows that discrete planar maps with a distinguished vertex are encoded by certain trees whose vertices are equipped with suitable integer labels.
We can reformulate this bijection in terms of looptrees, again with labelled vertices, as shown in Figure~\ref{fig:BDG_looptree}. 
Then just as in the continuous setting, a discrete looptree can be coded by a discrete excursion $W^n$ of duration $n$ (the number of edges of the maps) with increments in $\Z_{\ge-1} = \{-1, 0, 1, 2, \dots\}$.
Under mild exchangeability conditions, we shall prove in Theorem~\ref{thm:convergence_looptrees_labels_Levy} that when such random excursions converge in distribution after suitable scaling to $X^{\exc}$, then the rescaled looptrees converge to the looptree $\Loop(X^{\exc})$ for the so-called Gromov--Hausdorff--Prokhorov topology.
In addition, the random labels on the vertices of the looptree converge in the scaling limit to a random Gaussian field on $\Loop(X^{\exc})$.

\begin{figure}[!ht] \centering
\includegraphics[page=4, height=4.5cm]{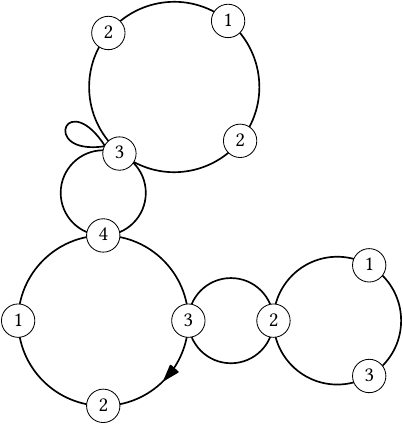}
\qquad
\includegraphics[page=5, height=5cm]{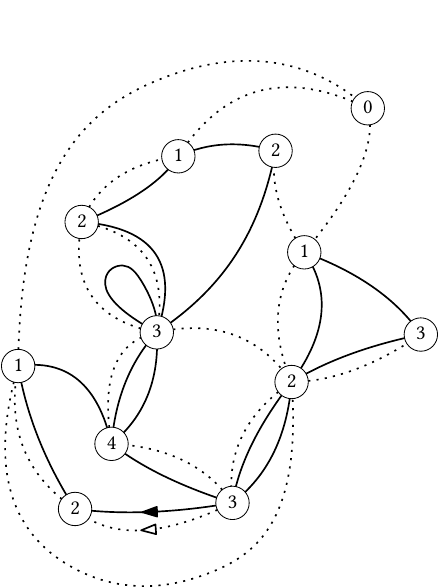}
\qquad
\includegraphics[page=6, height=5cm]{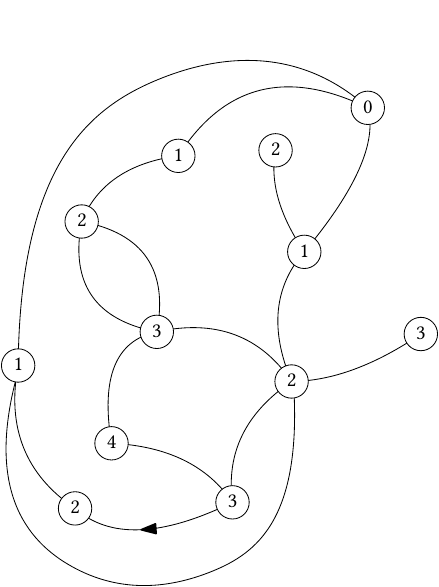}
\caption{Left: a well-labelled looptree. Right: a bipartite plane map with vertices labelled by their graph distance to a distinguished vertex (labelled $0$). Middle: the bijection linking them.
}
\label{fig:BDG_looptree}
\end{figure}
 
By now classical arguments originally due to Le~Gall~\cite{LG07}, this implies that the discrete random maps associated with these labelled looptrees are tight, and thus converge \emph{along subsequences} to some limit metric space.
Precisely, we shall prove in Theorem~\ref{thm:convergence_cartes_Boltzmann_Levy} below a formal version of the following statement.

\paragraph*{`Definition' of the L\'evy maps, to be made more precise in Theorem~\ref{thm:convergence_cartes_Boltzmann_Levy}.}
Let $M^n$ be random bipartite planar map with $n$ edges and a distinguished vertex and let $W^n$ denote the discrete excursion that codes its associated looptree. 
Assume an exchangeability property for the increments of $W^n$ (i.e.~for the faces of $M^n$), and that there exists a sequence $r_n \to \infty$ such that the rescaled path $(r_{n}^{-1} W^{n}_{\lfloor nt\rfloor})_{t \in [0,1]}$ converges in distribution to $X^{\exc}$ for the Skorokhod topology.
Then from every increasing sequence of integers, one can extract a subsequence along which the rescaled maps $(2 r_n)^{-1/2} M^n$ converge in distribution to some non degenerate measured metric space for the Gromov--Hausdorff--Prokhorov topology.
We shall denote by $\Map$ any such subsequential limit, which we call a \emph{L\'evy map} associated with $X^{\exc}$.

Theorem~\ref{thm:convergence_cartes_Boltzmann_Levy} actually provides more information and shows for example that the graph distances in $M^n$ to its distinguished vertex converge in distribution, without the need to extract a subsequence. However the convergence of all the pairwise distances remains open.
This result is the generalisation of~\cite{LG07, LGM11} in the Brownian and stable regimes, before one was able to prove that convergence of maps holds without the need to extract a subsequence.
These works also compute the Hausdorff dimensions of the subsequential limits and our next theorem extends this to the L\'evy regime.

\begin{thm}\label{thm:dimensions_cartes_Levy}
Almost surely, 
for any subsequential limit $M$,
it holds
\[\dim_{H} \Map = \diminf \Map = 2\bdgamma
\qquad\text{and}\qquad
\dim_{p} \Map = \dimsup \Map = 2\bdeta,\]
where $1 \leq \bdgamma \leq \bdeta \leq 2$ are the exponents defined in~\eqref{eq:exposants_BG}.
\end{thm}

\subsubsection*{Main techniques and difference with previous work on the stable case} The lower bounds on the dimensions follow from an upper bound on the rate of decay of the volume of balls centred at a uniform random point as the radius tends to $0$. The latter is established for both looptrees and maps in Section~\ref{sec:volume}. In the case of looptrees, it relies on a spinal decomposition, which rephrases the usual spinal decomposition of a tree and is formally described in terms of the coding L\'evy excursion. 
A geometric argument, summarised in Figure~\ref{fig:epine_looptrees}, allows to include the ball in a simpler set defined using hitting times of L\'evy processes. The core of the argument is the control of such quantities.
By local absolute continuity of the excursion near an independent uniform random time with a bi-infinite L\'evy process $(X_t)_{t \in \R}$, we are able to transfer estimates for that set obtained for the unconditioned process to its excursion.
The volume of balls in the map is controlled similarly, with a close geometric argument, now depicted in Figure~\ref{fig:boules_cartes}. The estimates however are much more challenging in the case of maps, since here one needs to control the Gaussian labels on the looptree.

These estimates are obtained in Section~\ref{sec:volume} and are much more involved compared to the corresponding ones in the stable case from~\cite{LGM11, CK14}.
Indeed, informally, in~\cite{LGM11} one starts from a uniform random point in the looptree, or actually in the associated tree, and follows its ancestral line towards the root and records the labels of the ancestors, until we first exceed a threshold.
These labels simply form a symmetric stable L\'evy process, so one can rely on explicit calculations; they are formally obtained by some time-change in the L\'evy process and one needs then to change back to the original time-scale.
This is no longer possible when the upper and lower exponents of the L\'evy process may differ; in addition, at the first time the label of an ancestor exceeds a threshold, this label may be much larger.
One thus has to be more careful and the looptree formalism helps a lot here to cut the trajectory differently.

Upper bounds on the dimensions rely on H\"older continuity estimates, obtained in Section~\ref{sec:dim_preuves}.
Indeed both the looptree and the map are a quotient of the interval $[0,1]$ by some pseudo-distance, and we prove there that the canonical projections are H\"older continuous.
For looptrees, this is obtained by cutting the excursion path into small pieces, whenever it makes a jump larger than a small threshold, and controlling the variation between two such jumps; again we rely on local absolute continuity with respect to an unconditioned L\'evy process. 
H\"older continuity estimates for the map are then easily derived from this, using the representation with labels on the looptree; by their Gaussian nature, the regularity of the labels is nearly half that of the looptree, which leads to the factor two in the dimensions.

This part is closer to~\cite{LGM11, CK14} in spirit but there self-similarity of the paths was a key ingredient in the calculations and arguing without it is more challenging. For example, in an unconditioned $\alpha$-stable L\'evy process, one can control the variation between two times simply by considering the value at a given time since the difference $X_{s+t}-X_{s}$ has the same law as $t^{1/\alpha} X_1$.
For this reason, our key estimates in Proposition~\ref{prop:looptree_Holder_borne_sup} and Corollary~\ref{cor:labels_Holder} require  more effort compared to the stable case.

\subsubsection*{Convergence of L\'evy-driven objects} In a general setting, the convergence of looptrees and maps when the driving L\'evy excursion varies is discussed in Theorem~\ref{thm:convergence_Levy_looptree_label} and Theorem~\ref{thm:convergence_Levy_cartes}. It is essentially established that if the transition densities of a sequence of L\'evy processes converge uniformly to those of another L\'evy process, then the associated labelled looptrees converge in law, and further the corresponding L\'evy maps converge in law after extracting a subsequence. A typical example of application we have in mind is the case where $X^{(\lambda)}$ is an $\alpha$-stable spectrally positive L\'evy process with drift $\lambda$, with varying $\alpha$ and $\lambda$. This will be studied in a companion paper.

\subsection{Open questions and plan}

This work leaves open several natural questions. First, here the L\'evy map $\Map$ is only defined as a subsequential limit of models of discrete maps, associated with discrete excursions that converge to $X^{\exc}$.
This was the case for the Brownian sphere in~\cite{LG07} until uniqueness was finally proved~\cite{LG13, Mie13}, and similarly for stable maps~\cite{LGM11}, for which uniqueness has been proved very recently~\cite{CMR25}.
We believe that several geometric ideas of~\cite{CMR25} can apply to prove uniqueness in the more general L\'evy setting when the Gaussian parameter $\beta$ vanishes, although the technical inputs are more involved without self-similarity. When the Gaussian parameter is nonzero, it is likely that other ideas are needed.

Following the first point, the convergence of looptrees and maps relies on that of the discrete excursions to $X^{\exc}$. The point of this paper was to construct the limiting objects related to $X^{\exc}$.
Given a L\'evy excursion, one can always make up discrete paths that converge to it and use them to build discrete maps, which we do here. The question of proving convergence of natural discrete models of maps and conditioned discrete excursions under optimal assumptions is left for future work.

In another direction, after computing the fractal dimensions of L\'evy trees~\cite{DLG05}, Duquesne and Le~Gall~\cite{DLG06} investigated the existence and the computation of a gauge function associated with a nontrivial Hausdorff or packing measure for the Brownian tree. Duquesne~\cite{Duq12} expressed more generally the exact packing measure of L\'evy trees, and proved on the contrary that there is no (regular) Hausdorff measure, even for the stable trees~\cite{Duq10}. More recently, Le~Gall~\cite{LG22b} expressed the Hausdorff measure of the Brownian sphere. 
We plan to investigate these questions on the L\'evy looptrees and maps.

The rest of this paper is organised as follow. In Section~\ref{sec:cartes_looptrees} we first define formally discrete and continuous looptrees and maps. We stress again that L\'evy looptrees are defined directly from the excursion $X^{\exc}$, whereas L\'evy maps are defined as subsequential limits of discrete maps, in Theorem~\ref{thm:convergence_cartes_Boltzmann_Levy}. Section~\ref{sec:biconditionnement} applies these results to the model of stable processes with a drift.
Then in Section~\ref{sec:volume} we state and prove technical bounds on the volume of small balls in looptrees and maps, which provide lower bounds for the fractal dimensions.
In Section~\ref{sec:dim_preuves} we first state and prove H\"older continuity estimates for the looptree and map distances, which provide upper bounds for the fractal dimensions. We then prove Theorem~\ref{thm:dimensions_fractales_Looptrees} and Theorem~\ref{thm:dimensions_cartes_Levy}.
Finally, Section~\ref{sec:limites_Levy} discusses the convergence of L\'evy looptrees and maps when the associated processes converge.

\subsubsection*{Acknowledgment}
The authors would like to thank the associated editor and the referees for their detailed comments and suggestions that greatly improved this paper.

\section{Planar maps and labelled looptrees}
\label{sec:cartes_looptrees}

We first define in Section~\ref{ssec:def_discret} the discrete labelled looptrees and plane maps as well as their coding by discrete paths. Then we construct analogously labelled looptrees associated with (excursions of) L\'evy processes in Section~\ref{ssec:Levy_looptrees_labels}. 
We shall stay brief and refer to~\cite{Mar24} and references therein for details. See also~\cite{Kha22} for more results on looptrees coded by a function.
Finally in Section~\ref{ssec:convergences} we state and prove invariance principles, showing that the discrete objects converge to the continuum ones. We define at this occasion the L\'evy maps as subsequential limits of discrete maps.

\subsection{The discrete setting}
\label{ssec:def_discret}

Recall that a plane map is the embedding of a planar graph on the sphere and viewed up to direct homeomorphisms. In order to break symmetries, we shall always distinguish an oriented edge, called the root edge of the map. The faces of the map are defined as the connected components of the complement of the edges. The face to the left of the root edge is called the boundary of the map and can be seen as an outer face. The degree of a face is defined as the number of edges surrounding it, counted with multiplicity; a plane map is bipartite if and only if all faces have even degree.

\subsubsection*{Looptrees}

A looptree, as represented on the right of Figure~\ref{fig:looptree_discret}, is by definition a plane map with the property that every edge has exactly one side incident to the outer face. 
This necessarily implies that all the other faces are simple cycles and are edge-disjoint; also no edge is pending inside the outer face.
One can naturally order the edges of a looptree, oriented to keep the boundary to their left, by following its contour: we start from the root edge $e_0$, then recursively, the edge $e_{k+1}$ is the leftmost edge originating from the tip of $e_k$. If the looptree has $n \geq 1$ edges, then the recursion ends at $e_n = e_0$, once the tour of the looptree is complete.
One can then define a path $W^{n} = (W^{n}_k ; 0 \leq k \leq n+1)$ by setting $W^n_0 = 0$ and letting for each $0 \leq k \leq n$ the increment $W^n_{k+1}-W^n_{k}$ be as follows:
\begin{enumerate}
\item Let $W^n_{k+1}-W^n_{k}$ be equal to the length minus $1$ of the cycle to the right of the edge $e_{k}$ \emph{if} no previously listed edge $e_0, \dots, e_{k-1}$ belongs to this cycle;
\item Otherwise let $W^n_{k+1}-W^n_{k} = -1$.
\end{enumerate}
One can check that $W^{n}$ is a \emph{{\L}ukasiewicz path} with length $n+1$, i.e.~a path whose increments all belong to $\Z_{\geq -1} = \{-1, 0, 1, 2, \dots\}$, started from $W^{n}_0 = 0$, with $W^{n}_{n+1} = -1$ and $W^{n}_{k} \geq 0$ for every $1 \leq k \leq n$. We extend it to a c\`adl\`ag path defined on the real interval $[0, n+1]$ by letting it drift at speed $-1$ between two integer times.
We refer to Figure~\ref{fig:looptree_Luka} for an example.
This construction is invertible, we shall describe below the converse construction directly in the continuum setting.
We denote by $\Loopd(W^n)$ the looptree coded by the path $W^n$. 
Let us mention that there is a well-known bijection between {\L}ukasiewicz paths and plane trees, in which the increments of $W^n$ code the offspring numbers minus $1$ when performing a depth first search exploration of the tree; the tree associated with $W^n$ is the dotted tree on the left of Figure~\ref{fig:looptree_discret}. 
It is important to note that $\Loopd(W^n)$ is \emph{not} the looptree naturally associated with the plane tree coded by $W^{n}$ as considered in~\cite{CK14,Kha22} (each internal vertex of the tree with its right-most offspring should be merged to obtain $\Loopd(W^n)$).
This version however is the one primarily studied in~\cite{Mar24} and that appeared previously in~\cite{CK15, KR20, Ric18} under the notation $\overline{\mathsf{Loop}}$ in view of applications to random planar maps.
This seemingly small difference actually affects the continuous part in the scaling limits, that is, the constant $a$ in~\eqref{eq:def_distance_looptree} below.

\begin{figure}[!ht] \centering
\includegraphics[page=7, height=5cm]{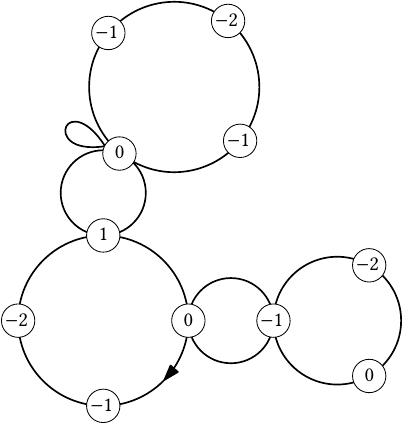}
\quad
\includegraphics[page=8, height=5cm]{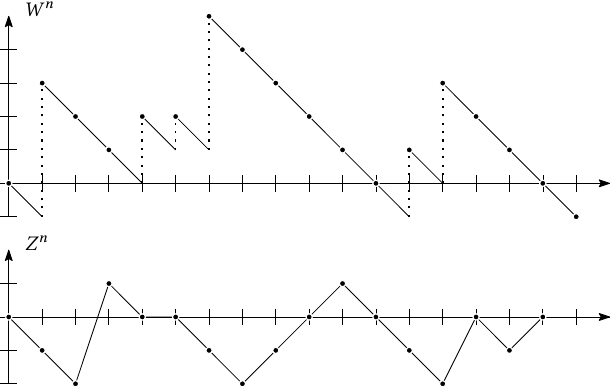}
\caption{\emph{Left:} A looptree equipped with a good labelling.
\emph{Right:} Its {\L}ukasiewicz path on top and its label process on the bottom.
} 
\label{fig:looptree_Luka}
\end{figure}

When considering scaling limits of looptrees, it will be useful to have a canonical order of their vertices.
First for each $0 \le k \le n$, let $u^n_k$ denote the origin of the oriented edge $e_k$ in the contour sequence. This list contains redundancies that we can remove by only keeping the last appearance of each vertex to extract a list of the vertices $(v^n_1, \dots, v^n_N)$ where $N$ is the total number of vertices; in particular, the root vertex is $u^n_0=u^n_n=v^n_N$ is the origin of the root edge.
The reason to keep the last appearance and not the first one come from the coding from the {\L}ukasiewicz path $W^n$: the vertices are in one-to-one correspondence with the negative increments of the path, so $N$ equals the number of such increments, and the time of a negative jump is precisely the last appearance of the corresponding vertex in the list $(u^n_0, \dots, u^n_n)$.

\subsubsection*{Labels and maps}

Let us further equip a looptree with a \emph{good labelling} on the vertices.
For every edge, oriented to keep the boundary on its left, we require that the difference of labels between the tip and the origin is an integer larger than or equal to $-1$. 
These differences, together with the label of the origin of the root edge fix uniquely the labels of all the vertices.
One can again encode the labels into a discrete path $Z^n = (Z^n_0, \dots, Z^n_n)$ by following the contour of the looptree. We further extend this path continuously by linear interpolation. See again Figure~\ref{fig:looptree_Luka} for an example.
One can notice that the labels of the vertices of any cycle of the looptree, when read clockwise, form a bridge with increments larger than or equal to $-1$; if the cycle has length $\ell \geq 1$, there are $\binom{2\ell-1}{\ell}$ such bridges.

The bijection from~\cite[Section~2]{BDG04} once reformulated in terms of looptrees shows that looptrees equipped with a good labelling are in one-to-two correspondence with bipartite plane maps equipped with a distinguished vertex, called \emph{pointed maps}. Moreover the bijection can be constructed by a simple algorithm in both directions (in addition to~\cite{BDG04}, let us refer to~\cite[Section~2.2]{Mar24}) and enjoys the following properties:
\begin{enumerate}
\item The map and the looptree have the same amount of edges;
\item The cycles of the looptree correspond to the faces of the map, and the length of a cycle is half the degree of the associated face;
\item The vertices of the looptree correspond to the non-distinguished vertices of the map, and the labels, 
shifted so the minimum is $1$,
equal the graph distances to the distinguished vertex in the map.
\item The factor $2$ in the correspondence only comes from the loss of the orientation of the root edge in the map.
\end{enumerate}
Let us refer to Figure~\ref{fig:BDG_looptree} for a graphical representation of this bijection.
The reader acquainted with the Schaeffer bijection between labelled trees and quadrangulations can observe that the latter is a particular case of the present bijection, when every cycle of the looptree has length $2$, so the looptree is really just a tree in which every edge is doubled.

In short, the construction of a map from a labelled looptree works as follows. First, shift all the labels so the minimum is $1$. Then assign to every corner of the outer face a successor, which is the first one which carries a smaller label when following the contour of the looptree, in the sense defined at the beginning of this section but extended by periodicity. The fact that the labelling is good implies that the difference of label between a corner and its successor is always exactly $-1$. The successor of the corners which carry the label $1$ is an extra vertex labelled $0$ in the outer face. Then the map is obtained by removing the edges of the looptree and instead linking every corner to its successor.
The extra vertex $0$ is the distinguished vertex of the map, and the labels of the vertices indeed correspond to their distance in the map to this vertex: the chain of successors forms a geodesic path to $0$.

This construction allows to canonically order the vertices of a map. Recall the order $(v^n_1, \dots, v^n_N)$ of the vertices of a looptree with $n$ edges and $N$ vertices. Then these are all the vertices of the associated map except the distinguished one which we set as $v^n_0$.
In addition, the vertex $v^n_N$ corresponds to the origin of the root edge in the looptree; in the map, it is an extremity of the root edge, but not necessarily its origin, but rather the one farther away from $v_0^n$.

\subsubsection*{Boltzmann distributions}

Let us briefly recall the model of Boltzmann maps introduced in Section~\ref{ssec:motiv}.
Let $\deg(f)$ is the degree of the face $f$, then given a sequence $(q_k)_{k \geq 1}$, one assigns to any finite bipartite pointed map a weight given by: 
\[\prod_{f \text{ face}} q_{\deg(f)/2}.\]
One can then sample a pointed map with $n$ edges proportionally to its weight (assuming that the total weight of the finite set of $n$-edged maps is non-zero).
By the previous bijection, this transfers to sampling a looptree with $n$ edges and equipped with a good labelling proportionally to the weight
$\prod_{c \text{ cycle}} q_{\text{length}(c)}$.
Recall that each cycle $c$ with length $\ell \geq 1$ in a looptree can be labelled in $\binom{2\ell-1}{\ell}$ different ways. Then the looptree, without the labels, is sampled proportionally to the weight:
\[\prod_{c \text{ cycle}} \binom{2\text{ length}(c)-1}{\text{length}(c)} q_{\text{length}(c)},\]
and then conditionally given this looptree, the labels on each cycle are uniform random bridges with increments larger than or equal to $-1$, independent of each other.
Recall that the length of the cycles of the looptree equal the size plus one of the nonnegative increments of the associated {\L}ukasiewicz path. 
This {\L}ukasiewicz path, with duration $n+1$, is thus sampled proportionally to the weight:
\[\prod_{k=1}^{n+1} \widetilde{q}_{x_k}
\qquad\text{where}\qquad
\widetilde{q}_{j} = \binom{2j+1}{j+1} q_{j+1} \enskip\text{for}\enskip j \geq 0 \enskip\text{and}\enskip \widetilde{q}_{-1} = 1,\]
and where $x_k$ is the $k$'th increment of the path.
By the so-called cyclic lemma~\cite[Lemma~6.1]{Pit06} one can realise such a random excursion by first sampling a \emph{bridge} with duration $n+1$, starting from $0$ and ending at $-1$, proportionally to the previous weight, and then cyclicly exchanging the increments at the first time this bridge achieves its overall minimum to turn it into an excursion.

Let us mention that this paper focuses on the continuum objects, so we consider pointed maps with $n$ edges which is the simplest model to study. However one may consider maps with $n$ vertices or $n$ faces instead, or even more general notions of size, see e.g.~\cite[Section~6.4]{Mar19} and references therein, or combine them and condition the map by its number of vertices, edges, and faces at the same time~\cite{KM23}. In another direction, one can consider maps without any distinguished vertex. Note that when we do not condition on the number of vertices, 
distinguishing a vertex in the map leads to a size-biasing of the latter. However standard arguments in the theory allow to show that the effect of this size-biasing disappears in the limit, provided some technical inputs, see e.g.~\cite[Proposition~6.4]{Mar19} and references therein.
Finally one can consider maps with a boundary~\cite{Bet15,BM17} by prescribing the degree of the face to the right of the root edge in the map, and hence the length of the cycle to the right of the root edge in the looptree, so finally the size of the first increment of the {\L}ukasiewicz path. By removing this first jump, one obtains a so-called first-passage bridge, and all this present work extends readily to this setting, see~\cite{Mar19, Mar24}.
We simply consider maps without boundary to ease the notation.

\subsection{L\'evy looptrees and Gaussian labels}
\label{ssec:Levy_looptrees_labels}

Recall from Section~\ref{ssec:intro_resultats} the model of L\'evy processes that we consider: $X$ has no negative jump, infinite variation paths, and its characteristic function is integrable~\eqref{eq:condition_integrale}. 
By inverse Fourier transform, the process $X$ then admits transition densities, say $(p_t(x))_{t > 0, x \in \R}$, which are jointly continuous in time and space as shown e.g.~in~\cite[Lemma~2.4]{Kni96}.
In addition, they never vanish, and precisely Sharpe~\cite{Sha69} shows that if one of them vanishes somewhere, then the process $X$ is, up to a drift, either a subordinator or the negative of one.

Following~\cite[Section~2]{Kni96}, thanks to the positivity and joint continuity of the transition densities, we may use them
to define a bridge with unit duration from $0$ to any $\vartheta \in \R$. Precisely, the conditional law $\P(\, \cdot \mid X_1=\vartheta)$ is characterised by the following absolute continuity relation with respect to the unconditioned path: for any $\varepsilon \in (0,1)$ and any continuous and bounded function $F$, it holds
\begin{equation}\label{eq:abs_cont}
\E\bigl[F\bigl((X_t)_{t \in [0,1-\varepsilon]}\bigr) \mid X_1=\vartheta\bigr] = \E\Bigl[F\bigl((X_t)_{t \in [0,1-\varepsilon]}\bigr)\, \frac{p_\varepsilon(\vartheta-X_{1-\varepsilon})}{p_1(\vartheta)}\Bigr].
\end{equation}
One can actually start from this identity and show that the family of processes it defines extends consistently to a process on the whole interval $[0,1]$, see again~\cite{Kni96}. 
We shall denote by $X^{\br}$ the bridge from $0$ to $0$, that is the path $(X_t)_{t \in [0,1]}$ under $\P(\, \cdot \mid X_1=0)$.
This bridge achieves its overall minimum at a unique time $U$ almost surely, which has the uniform distribution on $[0,1]$ by~\cite[Lemma~2.1 and Theorem~2.1]{Kni96}. We then define the Vervaat transformation $\Ver$, which exchanges the pre and post minimum, namely for every $t \in [0,1]$ we set:
\begin{equation}\label{eq:Vervaat}
X^{\exc}_t = (\Ver X^{\br})_t = X^{\br}_{U+t \bmod 1} - X^{\br}_U.
\end{equation}
Then $X^{\exc}$ is a unit duration excursion version of the process $X$, see also~\cite[Theorem~4]{UB14}.

\subsubsection*{L\'evy looptrees}

Let $Y$ be any c\`adl\`ag function on $[0, \infty)$ with no negative jump, which can be e.g.~the L\'evy process $X$ or its excursion $X^{\exc}$ extended to $0$ on $[1, \infty)$.
Let us recall from~\cite{Mar24} the construction of the looptree coded by $Y$ as well as random Gaussian labels on it. 
The reader may alternatively find this construction of looptrees in~\cite{Kha22} where they are called `vernation trees', together with additional topological details on these objects.
This originates from~\cite{LGM11,CK14} which applies in a pure jump case, where the continuous part $C$ below vanishes. 
First define from $Y$ a partial order $\preceq$ by setting for every $s, t \geq 0$:
\[s \prec t \qquad\text{when both}\qquad s < t \quad\text{and}\quad Y_{s-} \leq \inf_{[s,t]} Y.\]
We then set $s \preceq t$ when either $s \prec t$ or $s=t$. 
In analogy to the discrete setting, when $Y$ is the {\L}ukasiewicz path associated with a plane tree, we interpret $s = \sup\{r \geq 0 \colon r \prec t\}$ as the parent of $t$ and $\Delta Y_s$ as the degree of $s$.
Then for any $s, t \geq 0$, we let
\[s \wedge t \coloneqq \sup\{r \geq 0 \colon r \prec s \text{ and } r \prec t\}\]
denote their last common ancestor.
For every pair $s, t \geq 0$, let us set
\begin{equation}\label{eq:def_R}
R^t_s \coloneqq \inf_{r \in [s,t]} Y_r - Y_{s-} \qquad\text{if}\enskip s \preceq t,
\end{equation}
and let $R^t_s = 0$ otherwise, which is the case as soon as $Y$ does not jump at time $s$ for example.
Still in analogy to the discrete setting this corresponds to the number of siblings of $t$ that lie to its right.
Another way to visualise this quantity is to consider for $t$ fixed the dual path $Y^t = (Y_{t-} - Y_{(t-s)-})_{s \in [0,t]}$ and its running supremum $\overline{Y}\vphantom{Y}^t = (\sup_{r \in [0,s]} Y^t_r)_{s \in [0,t]}$, then $R^t_{t-s} = \Delta \overline{Y}\vphantom{Y}^t_s$ is the overshoot of $Y^t$ when it jumps to a new record at time $s$.
In general, the path $Y^t$ can also make records in a continuous way, thus we also define a continuous process by:
\begin{equation}\label{eq:def_J_et_C}
C_t = |\{\overline{Y}\vphantom{Y}^t_r; r \in [0,t]\}| = |\{\inf_{[s,t]} Y; s \in [0,t]\}|,
\end{equation}
where $|\cdot|$ stands for the Lebesgue measure.
For a c\`adl\`ag function $g$ on $[0,\infty)$, with no negative jump, let us define for every $0 \leq s \leq t$:
\begin{equation}\label{eq:def_distance_arbre}
d_g(s,t) = d_g(t,s) = g_s + g_{t-} - 2 \inf_{[s,t]} g
.\end{equation}
Note that we take the left-limit at time $t$.
When $g$ is continuous, it is well-known that $d_g$ is a pseudo-distance
and that 
the quotient space $T_g$ obtained by identifying all pairs of times at $d_g$-distance $0$ is a rooted real tree~\cite{LG05}.
In particular one may define a tree $T_C$ from the continuous part $C$ from~\eqref{eq:def_J_et_C}.
We shall also use $d_g$ for $g=Y$ to upper bound the looptree distance, see Equation~\eqref{eq:borne_dist_looptree_par_la_droite} below.

Let $Y=X$ be our L\'evy process. In this case the dual path $X^t$ has the same law as $X$, so $C_t$ has the same law as the Lebesgue measure of the range of the supremum process. The latter coincides with the measure of the range of the ladder process, obtained by time-changing this supremum process by the inverse local time at the supremum. In turn the ladder process is well-known to be a subordinator with Laplace exponent $\psi(\lambda)/\lambda$, see e.g.~\cite[Lemma~1.1.2]{DLG02}, whose drift coincides with the Gaussian parameter $\beta$ of $X$. Since the Lebesgue measure of the range of a subordinator equals its drift, then we conclude that for L\'evy processes, we have:
\[C \text{ vanishes} \quad\text{if and only if}\quad \beta = 0.\]
By the Vervaat transform and absolute continuity of the bridge, this holds also for $Y=X^{\exc}$.

If one wants again to draw an analogy with the discrete setting, viewing $X$ or $X^{\exc}$ as the {\L}ukasiewicz path of a tree, then when $\beta>0$, the nontrivial process $C$ coincides with $\beta H$ where $H$ is the so-called height process~\cite[Equation~14]{DLG02}, so $T_C = \beta T_H$ is the so-called L\'evy tree coded by $X$ or $X^{\exc}$, up to a multiplicative constant.
However when $\beta=0$, the process $C$ vanishes so the tree $T_C$ is reduced to a single point, whereas it is shown in~\cite{DLG02} that one can still define a nontrivial height process $H$, and thus a nontrivial L\'evy tree $T_H$.
We shall not need the height process in this work.

We may now define the looptree coded by $Y$.
We associate with any jump time $t \geq 0$ a cycle $[0, \Delta Y_t]$ with $\Delta Y_t$ identified with $0$, equipped with the metric $\delta_t(a,b) = \min\{|a-b|, \Delta Y_t - |a-b|\}$ for all $a,b \in [0, \Delta Y_t]$. For definiteness, if $\Delta Y_t = 0$, then we set $\delta_t(0,0) = 0$. 
Following~\cite{Mar24}, define the looptree distance with parameter $a\ge0$ between any $s, t \geq 0$ by:
\begin{equation}\label{eq:def_distance_looptree}
d^a_{\Loop(Y)}(s,t) \coloneqq \sum_{s \wedge t \prec r \prec s} \delta_r(0, R^s_r) + \sum_{s \wedge t \prec r \prec t} \delta_r(0, R^t_r) + \delta_{s \wedge t}(R^s_{s \wedge t}, R^t_{s \wedge t}) + a\, d_C(s,t),
\end{equation}
where here and below, a sum over an empty set is null and where $R^t_r$ is defined in~\eqref{eq:def_R} and $d_C$ is the tree distance defined in~\eqref{eq:def_distance_arbre}.
Let us refer to Figure~\ref{fig:distance_looptree} for a schematic representation of this distance.

\begin{figure}[!ht] \centering
\includegraphics[page=9, width=\linewidth]{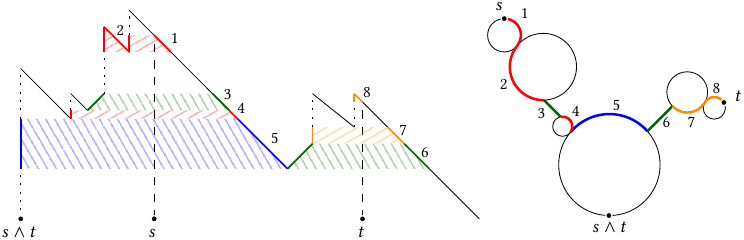}
\caption{A schematic representation of the looptree distance between two times $s$ and $t$ and the coding excursion. In different colours are highlighted the four components of the distance in~\eqref{eq:def_distance_looptree}: from left to right in the latter in red, orange, blue, and green. Each portion of the geodesic from $s$ to $t$ is numbered and associated with the corresponding part of the excursion.}
\label{fig:distance_looptree}
\end{figure}

It can be checked that $d^a_{\Loop(Y)}$ is indeed a continuous pseudo-distance on $[0,1]$ for every $a$, which moreover satisfies the following inequality:
let $d_Y$ be defined as $d_g$ in~\eqref{eq:def_distance_arbre} with $g=Y$, then
\begin{equation}\label{eq:borne_dist_looptree_par_la_droite}
d^1_{\Loop(Y)} \leq d_{Y},
\end{equation}
see~\cite[Proposition~3.2]{Mar24} which extends~\cite[Lemma~2.1]{CK14} when $C=0$.
Henceforth we denote by $\Loop^a(Y)$ the quotient space obtained by identifying all pairs of times at pseudo-distance $0$. We shall simply write $\Loop(Y)$ and $d_{\Loop(Y)}$ for $\Loop^1(Y)$ and $d^1_{\Loop(Y)}$.
The value of $a>0$ does not affect the fractal dimensions of the looptree, and Theorem~\ref{thm:dimensions_fractales_Looptrees} extends to $\Loop^{a}(X^{\exc})$. This constant only appears in the invariance principles in Section~\ref{ssec:convergences}.
In the case where $Y$ is a Brownian excursion, we have $C=Y$ so $\Loop^a(Y) = a\cdot T_C = a\cdot T_{Y}$ is a scaled Brownian tree coded by $Y$. 
The metric space $\Loop^a(Y)$ has been introduced in~\cite{{Mar24}}, and in the particular case $a=1/2$ in~\cite{Kha22} it is called the vernation tree coded by $Y$.

Observe that a discrete looptree as defined in Section~\ref{ssec:def_discret} is a map, while a continuous looptree $\Loop$ is a compact metric space. Let us note however that the two definitions are very close: specifically if $Y$ is a {\L}ukasiewicz path, extended to the real line by a drift $-1$ between the jumps, as depicted in Figure~\ref{fig:looptree_Luka}, then the continuum looptree $\Loop(Y)$ is obtained from the discrete one $\Loopd(Y)$ by replacing each edge by a segment of length $1$, and also attaching an extra such segment to the root.

\subsubsection*{Gaussian labels on looptrees}

Recall that $Y$ is any c\`adl\`ag function on $[0, \infty)$ with no negative jump and that $C$ is defined by~\eqref{eq:def_J_et_C}.
One may define a random Gaussian field on the tree $T_C$, known as the \emph{Brownian snake} driven by $C$, see~\cite[Chapter~4]{DLG02} for details on such processes in a broader setting. Formally, conditionally given $C$, there exists a centred Gaussian process $Z^C = (Z^C_t)_{t \in [0,1]}$ with covariance function
\[\E[Z^C_s Z^C_t] = \min_{r \in [s,t]} C_r,
\qquad s,t \in [0,1].\]
Equivalently, we have $\E[(Z^C_s-Z^C_t)^2] = d_C(s,t)$.
On an intuitive level, after associating a time $t$ with its projection in $T_C$, the values of $Z^C$ evolve along the branches of the tree like a Brownian motion, and these Brownian motions split at branchpoints and afterwards evolve independently.

We construct a Gaussian field on the looptree $\Loop^a(Y)$ by placing on each cycle an independent Brownian bridge, with duration given by the length of the cycle, which describes the increments of the field along the cycle. 
Formally, recall that the standard Brownian bridge $b = (b(t))_{t \in [0,1]}$ is a centred Gaussian process with covariance
\[\E[b(s) b(t)] = \min(s,t) - st,
\qquad s,t \in [0,1].\]
One can consider a bridge of any given duration using a diffusive scaling.
Now recall the partial order $\prec$ as well as the notation $R^t_s = \inf_{r \in [s,t]} Y_r - Y_{s-}$ from~\eqref{eq:def_R}. 
Let $C$ be the continuous part defined in~\eqref{eq:def_J_et_C} and let $Z^C$ be the associated Brownian snake. Independently let $(b_i)_{i \geq 1}$ denote a sequence of i.i.d. standard Brownian bridges
and define for every $a\ge0$ and $t \in [0,1]$:
\begin{equation}\label{eq:def_labels_browniens}
Z^a_t \coloneqq \sqrt{a}\, Z^C_t + \sum_{t_i \prec t} \sqrt{\Delta Y_{t_i}}\, b_i\bigl((\Delta Y_{t_i})^{-1} R^t_{t_i}\bigr),
\end{equation}
where the $t_i$'s are the jump times of $Y$ up to time $t$.
Given $Y$, the summands in~\eqref{eq:def_labels_browniens} are independent zero-mean Gaussian random variables with variance
\begin{align*}
\E\biggl[\Bigl(\sqrt{\Delta Y_{t_i}} b_i\bigl((\Delta Y_{t_i})^{-1} R^t_{t_i}\bigr)\Bigr)^2\biggr]
&= \frac{R^t_{t_i} (\Delta Y_{t_i} - R^t_{t_i})}{\Delta Y_{t_i}}
\\
&\leq \min\{R^t_{t_i}, \Delta Y_{t_i} - R^t_{t_i}\}
= \delta_{t_i}(0, R^t_{t_i}).
\end{align*}
Arguing as in the proof of Proposition~6 in~\cite{LGM11} (recast in the formalism of looptrees), one deduces that for any $q>0$, there exists $K_q > 0$ such that for every $s, t \in [0,1]$, it holds that
\begin{equation}\label{eq:label_moment_Kolmogorov_looptrees}
\E[|Z^a_t-Z^a_s|^q] \leq K_q \cdot d^a_{\Loop(Y)}(s,t)^{q/2}.
\end{equation}
Taking $s=0$, this shows that the random variable $Z^a_t$ in~\eqref{eq:def_labels_browniens} is well-defined for any fixed $t$, which in particular implies that the series in~\eqref{eq:def_labels_browniens} converges.
Further, given $Y$, we see that any pair $s,t \geq 0$ that has $d^a_{\Loop(Y)}(s,t) = 0$ also has $Z^a_s = Z^a_t$ almost surely so $Z^a$ can actually be seen as a random process indexed by the looptree. 
Note that by the scaling property of Gaussian variables, the process $c Z^a$ constructed from $Y$ has the same law as the process $Z^a$ but constructed from $c^2 Y$.
The value $a=1/3$ shall play a predominant role in relation with maps and we simply write $Z = Z^{1/3}$. We shall also denote by $Z^{\exc}$ this process $Z=Z^{1/3}$ constructed from 
$Y = X^{\exc}$.

\subsection{Invariance principles}
\label{ssec:convergences}

We have not defined L\'evy maps yet; they will shortly be defined as subsequential limits of rescaled finite random maps. 
Precisely, we shall consider rooted bipartite maps $M^n$ with $n$ edges that carry a distinguished vertex $v^n_0$ and we shall assume that their distribution for every $n$ only depends on their face degrees in the sense that if $(m,v)$ and $(m', v')$ are two (rooted bipartite and) pointed maps with $n$ edges that have the same amount of faces with degree $2k$ for every $k \ge 1$, then
\begin{equation}\label{eq:hyp_carte_loi_degres}
\P((M^n, v^n_0) = (m,v)) = \P((M^n, v^n_0) = (m',v')).
\end{equation}
Typical examples include the conditioned Boltzmann distributions recalled earlier, since $(m,v)$ and $(m', v')$ have the same weight. We may here generalise this distribution by letting the weight sequence $q$ of the Boltzmann law vary with $n$.

This assumption has two consequences.
First, conditionally on its number of faces of each degree $(M^n, v^n_0)$ has the uniform distribution on maps with these degree statistics, which will falls in the scope of~\cite{Mar19,Mar24}.
Second, conditionally on $M^n$, the distinguished vertex $v^n_0$ has the uniform distribution on the vertices of $M^n$.
This will only be used in this work in order to identify the limit as the Brownian sphere or a stable map in the appropriate regime.

By construction of the bijection presented earlier between pointed maps and well-labelled looptrees, and further with pairs of processes, if the random map $(M^n, v^n_0)$ satisfies~\eqref{eq:hyp_carte_loi_degres}, then its associated pair $(W^n, Z^n)$ gives the same probability to all pairs in which the {\L}ukasiewicz paths have the same amount of jumps of size $k$ for every $k \ge -1$.
In particular, conditionally on its increment sizes, the path $W^n$ is an excursion sampled uniformly at random among all possibilities, and $Z^n$ is further obtained conditionally given $W^n$ as a uniformly chosen random good labelling of the looptree, obtained by placing independent uniform random bridges around each cycles.
By the so-called cycle lemma, the law of $W^n$ is that of the path obtained from a random bridge $B^{n} = (B^{n}_{k})_{0\leq k \leq n+1}$, with $B^{n}_{n+1}=-1$, and with increments $\Delta B^{n}_i \in \Z_{\ge-1}$ that are exchangeable, by cyclicly shifting $B^{n}$ at its first minimum to turn it into an excursion.

Let us next say a few words on the notion of convergence that we consider.
The abstract framework to study scaling limit of (possibly pointed) finite graphs, such as (loop)trees and maps that has been developed is that of the Gromov--Hausdorff--Prokhorov (GHP) distance.
The latter is a separable and complete distance on the space of isometry classes of compact measured metric spaces, possibly with distinguished points. 
Then we say that rescaled graphs converge when their set of vertices equipped with the rescaled graph distance and the uniform probability measure do in the GHP sense.
In our case, we can make this convergence much more explicit.
Indeed, recall from Section~\ref{ssec:def_discret} that the vertices of looptrees and maps can ordered in some canonical way $(v^n_0, \dots, v^n_N)$ where $n$ is the number of edges and $N+1$ the number of vertices of the map, and where for simplicity we let $v^n_0=v^n_N$ in the case of the looptree, whereas in the map $v^n_0$ is the distinguished vertex.
One can then encode the graph distances in each model into a function $d^n$ on $\{0, \dots, N\}$ by letting $d^n(i,j)$ be the distance between $v^n_i$ and $v^n_j$ in the corresponding model, which we transform into a function on $[0,1]^2$ by a bilinear interpolation and scaling by $N$.
Then if for some scaling factor $r_n\to\infty$ we have
\[(r_n^{-1} d^n(s, t))_{s,t \in [0,1]} \cvloi (d(s,t))_{s,t \in [0,1]}\]
for the uniform topology, for some continuous pseudo-distance $d$, then 
the graph rescaled by the factor $r_n^{-1}$ indeed converges for the GHP distance to the quotient space $[0,1] / \{d=0\}$ obtained by identifying all pairs of times at distance $0$.
This convergence is actually what is proved for Brownian and stable looptrees and maps~\cite{LG13,Mie13,CK15,CMR25}.

Let us apply this to random looptrees first and then to maps.
We let $X$ be a L\'evy process with Laplace exponent $\psi$, with Gaussian parameter $\beta \geq 0$, and we let $X^{\br}$ and $X^{\exc}$ denote the unit duration bridge and excursion versions of $X$ respectively. 
For $a \ge 0$, we also let $\Loop^a(X^{\exc})$ denote the looptree defined as in~\eqref{eq:def_distance_looptree} and the lines below.
Finally, let $Z^{\exc}$ denote the label process $Z^a$ defined as in~\eqref{eq:def_labels_browniens} with $a=1/3$ and $Y=X^{\exc}$.
The next result is an easy consequence of~\cite[Theorems~7.4 and~7.9]{Mar24} by conditioning on the increment sizes. 
Below $d_{\Loopd(W^{n})}$ denotes the function on $[0,1]^2$ that encodes the graph distances in the looptree $\Loopd(W^{n})$ as introduced above.

\begin{thm}\label{thm:convergence_looptrees_labels_Levy}
For every $n \geq 1$, let $B^{n} = (B^{n}_{k})_{0\leq k \leq n+1}$ be a random bridge, with $B^{n}_{n+1}=-1$, with increments $\Delta B^{n}_i \in \Z_{\ge-1}$ and which are exchangeable. 
Let $W^{n} = (W^{n}_{k})_{0\leq k \leq n+1}$ denote the excursion obtained by cyclicly shifting $B^{n}$ at its first minimum. Let further $Z^n$ denote the label process associated with a uniform random good labelling of the looptree coded by $W^n$.
\begin{enumerate}
\item Suppose that there exists a sequence $r_{n} \to \infty$ such that
\[(r_{n}^{-1} B^{n}_{\lfloor nt\rfloor})_{t \in [0,1]} \cvloi X^{\br}.\]
for the Skorokhod $J_{1}$ topology.
Then the following convergences in distribution hold jointly:
\[
(r_{n}^{-1} W^{n}_{\lfloor nt\rfloor})_{t \in [0,1]} \cvloi X^{\exc}
\qquad\text{and}\qquad
((2 r_n)^{-1/2} Z^n_{n t})_{t \in [0,1]} \cvloi Z^{\exc}
\]
for the Skorokhod $J_{1}$ topology.

\item Assume in addition that there exists $a \ge 0$ such that:
\begin{equation}\label{eq:cvhyp2}
r_n^{-2} \#\{i \leq n+1 \colon \Delta B^n_i \in 2\Z-1\} \cvproba a \beta^2,
\end{equation}
where $a=0$ by convention when $\beta=0$.
Then, jointly with the first part, we have:
\[r_{n}^{-1}\, d_{\Loopd(W^{n})} \cvloi d^{(a+1)/2}_{\Loop(X^{\exc})},\]
in the space of real-valued continuous functions defined on $[0,1]^2$ equipped with the uniform topology.
This implies that $r_{n}^{-1}\,\Loopd(W^{n})$ converges in distribution to $\Loop^{(a+1)/2}(X^{\exc})$ for the Gromov--Hausdorff--Prokhorov topology.
\end{enumerate}
\end{thm}

Let us mention that~\cite[Section 6.3]{Kha22} obtains limit theorems for looptrees coded by excursions, but there are several major differences. First, as we have already mentioned in the beginning of Sec.~\ref{ssec:def_discret}, 
the discrete looptree coded $W^{n}$ considered here slightly differs from the version studied in~\cite{Kha22}.
Second, in our setting $X$ is not necessarily pure jump and we may have $a \neq 1/2$, which is outside of the scope of vernation trees considered in~\cite{Kha22}.

\begin{rem}\label{rem:geodesic_space}
Given a L\'evy bridge $X^{\br}$ and $a \ge 0$, one can always construct a sequence of discrete bridges $B^n$ of duration $n+1$ respectively which satisfies the assumptions of Theorem~\ref{thm:convergence_looptrees_labels_Levy}.
Therefore the L\'evy looptree $\Loop^{(a+1)/2}(X^{\exc})$ can always be realised as the Gromov--Hausdorff limit of rescaled finite looptrees. As such, it is a length space, and since it is compact, a geodesic space.
\end{rem}

Let us comment on Assumption~\eqref{eq:cvhyp2}.
When following a geodesic path in the looptree and traversing a cycle, one takes the shortest between the left and right length. Roughly speaking, typically half of the cycles are traversed on the right and half on the left, thus the microscopic cycles contributes to roughly half the continuous part of the coding path, leading to $\Loop^{1/2}(X^{\exc})$ in the limit.
However, at the discrete level, the fact that the cycle has even or odd length influences the law of the length of the smallest part between left and right, namely $\sum_{j=1}^k \min(j, k-j)$ depends on the parity of $k$. This contribution is precisely captured in~\eqref{eq:cvhyp2} and adds an extra factor to the continuous part in the looptree distance in the limit. Such a phenomenon was already observed in~\cite[Theorem~1.2]{KR20} (see however~\cite[Remark~7.7]{Mar24} on the extra constant).
We note that this 
has however no influence on the labels. The parameter $1/3$ appears in the limit due to our choice of labels in the discrete model, see~\cite[Theorem~7.9]{Mar24} for more general labels.

\begin{proof}
Using Skorokhod's representation theorem, let us work on a probability space where the assumptions hold almost surely. 
For every $i \geq 1$, let $\Delta X^{\shortdownarrow}_i \geq 0$ denote the $i$'th largest increment of $X^{\br}$ and then set

\[\theta_i = \Bigl(\beta^2 + \sum_{i \geq 1} (\Delta X^{\shortdownarrow}_i)^2\Bigr)^{-1/2} \Delta X^{\shortdownarrow}_i
\quad\text{and then}\quad 
\theta_0^2 = 1 - \sum_{i \geq 1} \theta_i^2 
= \frac{\beta^2}{\beta^2 + \sum_{i \geq 1} (\Delta X^{\shortdownarrow}_i)^2}
.\]
Finally let 
\[Y = \Bigl(\beta^2 + \sum_{i \geq 1} (\Delta X^{\shortdownarrow}_i)^2\Bigr)^{-1/2} X^{\br},\]
be the exchangeable increment process with Gaussian parameter $\theta_0$ and jump sizes $(\theta_i)_{i \geq 1}$.
Then the assumptions of~\cite[Theorem~7.4]{Mar24} are fulfilled, namely let $\sigma^2_n = \sum_{i=1}^n \Delta B^n_i (\Delta B^n_i+1) = \sum_{i=1}^n (\Delta B^n_i)^2 - 1$, then according to~\cite[Theorem~16.23]{Kal02} the convergence of $r_n^{-1} B^n$ to $X^{\br}$ is equivalent to the convergence for every $i \geq 1$ of the $i$'th largest increment of $r_n^{-1} B^n$ to $\Delta X^{\shortdownarrow}_i$ and the convergence $r_n^{-2} \sigma_n^2 \to \beta^2 + \sum_{i \geq 1} (\Delta X^{\shortdownarrow}_i)^2$.
Thus if we instead rescale $B^n$ by $\sigma_n^{-1}$, then for every $i \geq 1$, the $i$'th largest increment of $\sigma_n^{-1} B^n$ converges to $(\beta^2 + \sum_{i \geq 1} (\Delta X^{\shortdownarrow}_i)^2)^{-1/2} \Delta X^{\shortdownarrow}_i = \theta_i$, while
\[
\sigma_n^{-2} \#\{i \leq n \colon \Delta B^n_i \in 2\Z-1\} 
\cv
\frac{a \beta^2}{\beta^2 + \sum_{i \geq 1} (\Delta X^{\shortdownarrow}_i)^2} = a \theta_0^2
\]
by our assumption~\eqref{eq:cvhyp2}.
Denote by $\Ver Y$ the Vervaat transform of the bridge $Y$, as in~\eqref{eq:Vervaat}, so
\[\Ver Y=\Bigl(\beta^2 + \sum_{i \geq 1} (\Delta X^{\shortdownarrow}_i)^2\Bigr)^{-1/2} X^{\exc}.\]
It is plain from the construction of the looptree distance~\eqref{eq:def_distance_looptree} that for every $c>0$, we have $c\cdot \Loop^{(a+1)/2}(\Ver Y) = \Loop^{(a+1)/2}(c \Ver Y)$.
Then applying~\cite[Theorem~7.4]{Mar24}, or precisely Equation~38 there, we deduce that
\[r_{n}^{-1}\, d_{\Loopd(W^{n})}
\cvloi
\Bigl(\beta^2 + \sum_{i \geq 1} (\Delta X^{\shortdownarrow}_i)^2\Bigr)^{1/2} d^{(a+1)/2}_{\Loop(\Ver Y)}
= d^{(a+1)/2}_{\Loop(X^{\exc})}
,\]
for the uniform topology.

As for the labels, we apply~\cite[Theorem~7.9]{Mar24} in the particular case of the labels that code maps (see the beginning of~\cite[Section~7.4]{Mar24}) and deduce the convergence in distribution
\[((2r_n)^{-1/2} Z^n_{n t})_{t \in [0,1]} \cvloi \Bigl(\beta^2 + \sum_{i \geq 1} (\Delta X^{\shortdownarrow}_i)^2 \Bigr)^{1/4} \zeta^{1/3},\]
where the process $\zeta^{1/3}$ is defined as in~\eqref{eq:def_labels_browniens} using $\Ver Y$ instead of $X^{\exc}$.
As above, by the scaling property of Gaussian variables, multiplying the label process by some $c>0$ amounts to multiplying the underlying excursion with no negative jump by $c^2$, namely the limit above equals precisely $Z^{\exc}$.
\end{proof}

Theorem~\ref{thm:convergence_looptrees_labels_Levy} finds applications to scaling limits of random planar maps which is now quite classical, and we refer to e.g.~\cite[Section~2.4]{Mar24} for a recent account.
Let $d_{M^n}$ denote the function on $[0,1]^2$ that encodes the graph distances in the map $M^n$ induced by the order on the vertices inherited from that on the looptree.
Recall that in this order, the first vertex is the distinguished one $v_0^n$, while the last vertex is the extremity of the root edge that is the farthest away from $v_0^n$.

\begin{thm}\label{thm:convergence_cartes_Boltzmann_Levy}
Let $(M^n, v_0^n)$ be a random pointed map with $n$ edges whose distribution satisfies~\eqref{eq:hyp_carte_loi_degres}.
Let $W^n$ and $Z^n$ denote its associated {\L}ukasiewicz path and label process and assume that they satisfy the assumptions of Theorem~\ref{thm:convergence_looptrees_labels_Levy}.
Then the following holds jointly with this theorem.
\begin{enumerate}
\item The convergence in distribution
\[\Bigl(\frac{1}{\sqrt{2 r_n}}d_{M^n}(0, t)\Bigr)_{t \in [0,1]} \cvloi (Z^{\exc}_t - \min Z^{\exc})_{t \in [0,1]},\]
holds for the uniform topology. 
\item From every increasing sequence of integers, one can extract a subsequence along which 
the convergence in distribution
of continuous functions on $[0,1]^2$
\[\frac{1}{\sqrt{2 r_n}} d_{M^n} \cvloi D,\]
holds for the uniform topology, where $D$ is a random continuous pseudo-distance.
Consequently, along this subsequence, 
the rescaled map $(2 r_n)^{-1/2} M^n$ converges in distribution to the quotient space $[0,1]/\{D=0\}$ for the Gromov--Hausdorff--Prokhorov topology.

\item Define for every $0 \le s \le t \le 1$:
\begin{equation}\label{eq:non_distance_label}
D^\circ(s,t) = D^\circ(t,s) = Z^{\exc}_s + Z^{\exc}_t - 2 \max\Bigl(\inf_{[s,t]} Z^{\exc}, \inf_{[0,s] \cup [t,1]} Z^{\exc}\Bigr)
.\end{equation}
Then all the subsequential limits $D$ satisfy
\[D \leq D^\circ \qquad\text{almost surely}.\]

\item Sample $U$ uniformly at random on $[0,1]$ and independently of the rest. Then all the subsequential limits $D$ satisfy the identity in law:
\[(D(U,t))_{t \in [0,1]} \eqloi (Z^{\exc}_t - \min Z^{\exc})_{t \in [0,1]}.\]

\item If $X^{\exc}$ has no jump, so it reduces to $\beta$ times the standard Brownian excursion, then we have the convergence without the need to extract a subsequence:
\[\frac{1}{\sqrt{2 r_n}} d_{M^n} \cvloi \sqrt{\beta/3} \cdot D^\ast,\]
where $[0,1]/\{D^\ast=0\}$ is the standard Brownian sphere.
Similarly, if $\psi(\lambda)=c \lambda^\alpha$ with $c>$ and $1 < \alpha < 2$, 
then the maps converge in distribution without extraction to 
$c^{1/(2\alpha)}$ times the standard $\alpha$-stable map.
\end{enumerate}
\end{thm}

\begin{proof}
The first claim is a consequence of the convergence of $Z^n$ Theorem~\ref{thm:convergence_looptrees_labels_Levy} since the distances to $v^n_0$ in the map are given by the labels on the looptree, recorded in $Z^n$, shifted by their overall minimum plus $1$, see e.g.~the proof of Theorem~7.12 in~\cite{Mar19} for details.
The second and third claims follow from an argument from the original work of Le~Gall~\cite{LG07}, see also~\cite[Lemma~2.6]{Mar24} for a recent account of the argument in a similar context.
In a few words, in the discrete models, recall that the map is constructed from the labelled looptree by linking each corner to its successor, defined as the first one with a smaller label when following the contour of the looptree. Then if we start from two corners, the two chains of successors eventually meet and the sum of the two lengths until this meeting point (minus $2$) is precisely the discrete analogue of $D^\circ$, say $D^\circ_n$, defined using $Z^n$, and this quantity upper bounds the distances in the map between the two vertices where we started.
Since $D^\circ_n$ depends continuously on $Z^n$, then it converges in distribution after scaling to $D^\circ$, from which we easily derive both tightness of the distances in the map under the same rescaling due to the upper bound $d_{M^n} \le D^\circ_n+2$, and $D \le D^\circ$ after passing to the limit along any convergent subsequence.

The fourth claim is a re-rooting (or rather, re-pointing) invariance property: by the first two claims, the process on the right describes the distances to the distinguished vertex in the in the limit, whereas the one on the left describes the distances to an independent uniform random vertex. The identity thus follows from the fact that, in the discrete model, the distinguished vertex has the law of an independent uniform random vertex, which is due to the assumption~\eqref{eq:hyp_carte_loi_degres}.

The last claim in the Brownian case is a consequence of the identification of the Brownian sphere and the argument can be found in~\cite{LG13}, see precisely the proof of Equation~59 there.
It relies on the following ingredients: our claims 2, 3, and 4, the (nontrivial) fact that, if $U$ and $U'$ are two independent uniform random times on $[0,1]$, then $Z^{\exc}_U - \min Z^{\exc}$ has the same law as $\sqrt{\beta/3}$ times the distances between $U$ and $U'$ for the pseudo-distance $D^\ast$ describing the Brownian sphere, and finally, the fact that, a priori, for any subsequential limit $D$, we have $D(s,t)=0$ whenever $s$ and $t$ are identified in the Brownian tree coded by $X^{\exc}$.
In a few words, we obtain from these ingredients
the identity in law:
\[D(U,V) \eqloi \sqrt{\beta/3}\cdot D^\ast(U, V),\]
and in addition that 
$D \le D^\ast$ almost surely so the above identity in law in fact holds almost surely so finally $D=\sqrt{\beta/3} \cdot D^\ast$ almost surely by a density argument.
The same argument applies in the stable regime as discussed in~\cite[Section~7.4.2]{CMR25} thanks again to the identification there of $D^\ast(U,U')=Z^{\exc}_U - \min Z^{\exc}$ in law in this case, where in the last ingredient, we require that $D(s,t)=0$ whenever $s$ and $t$ are identified in the looptree tree coded by $X^{\exc}$.
We will prove this last fact later in Lemma~\ref{lem:distance_map_on_looptree} in our general L\'evy regime.
\end{proof}

The  invariance by re-rooting argument used in the last claim in the Brownian and stable regimes applies in the general L\'evy case and would similarly allow to obtain the convergence for all models whose {\L}ukasiewicz paths converge to the same L\'evy excursion, 
if one could prove 
the identity in law $D^\ast(U,U')=Z^{\exc}_U - \min Z^{\exc}$,
for example by proving the convergence of one specific discrete model of maps.

\begin{rem}
The first part of the statement, characterising the distances to the distinguished vertex $v^n_0$, is not enough to prove that all subsequential limits agree.
However, it does provide the convergence, without the need to extract a subsequence, of several quantities of interest.
Let $V(M^n)$ denote the set of vertices of $M^n$ and $\rho_n$ the origin of the root edge, then 
\[\frac{1}{\sqrt{2 r_n}} d_{M^n}(v^n_0, \rho_n) \cvloi - \min Z^{\exc}\]
and
\[
\frac{1}{\sqrt{2 r_n}} \max_{v \in V(M^n)} d_{M^n}(v^n_0, v) \cvloi \max Z^{\exc} - \min Z^{\exc},\]
and finally, for every continuous and bounded function $F$, 
\[\frac{1}{\#V(M^n)} \sum_{v \in V(M^n)} F\Bigl(\frac{1}{\sqrt{2 r_n}} d_{M^n}(v^n_0, v)\Bigr) \cvloi \int_0^1 F(Z^{\exc}_t - \min Z^{\exc}) \d t.\]
\end{rem}

In the rest of this paper, we call a \emph{L\'evy map} any subsequential limit $\Map = [0,1]/\{D=0\}$ in the previous theorem. 
We equip it with the metric $d_\Map$, the image of $D$ by the canonical projection, and the probability $p_\Map$, given by the push-forward of the Lebesgue measure on $[0,1]$. 
Similarly to looptrees in Remark~\ref{rem:geodesic_space}, L\'evy maps are geodesic spaces.

\section{Limits of biconditioned planar maps and  stable excursions with a drift}

\label{sec:biconditionnement}

Let us present in this section applications of our general results in the case where the coding process is a spectrally positive stable L\'evy process with a drift, defined just a few lines below, and recall how such objects appear as limits of multiconditioned discrete models introduced in~\cite{KM23}.

For every $\alpha \in (1,2)$, let us denote by $X^{\alpha}$ the stable L\'evy process characterised by the Laplace exponent $\psi(\lambda) = \lambda^\alpha$, or equivalently by the characteristic function:
\begin{equation}\label{eq:characteristic_stable}
\E[\exp(i u X^\alpha_t)] 
= \exp\Bigl(t |u|^\alpha \Bigl(\cos\Bigl(\frac{\pi \alpha}{2}\Bigr) - i \sigma(u) \sin\Bigl(\frac{\pi \alpha}{2}\Bigr)\Bigr)\Bigr)
,\end{equation}
for every $u \in \R$ and $t>0$, where we used the notation $\sigma(u) = \ind{u>0} - \ind{u<0}$ for the sign function to distinguish it from the sine function.
The right-hand side satisfies the integrability condition~\eqref{eq:condition_integrale}, and thus it admits positive, bi-continuous, transition densities $p^{\alpha}_t$ given by inverse Fourier transform, from which one can define the conditional law $\P(\,\cdot \mid X^\alpha_1 = \vartheta)$ for any $\vartheta \in \R$, characterised by~\eqref{eq:abs_cont}. For any $\vartheta \in \R$, let us also denote by $X^{\alpha, \vartheta}_t = X^\alpha_t - \vartheta t$ the stable process with  drift $-\vartheta$, which has transition densities $p^{\alpha}_t(\cdot + \vartheta t)$.
One can also define bridges of this process and we let $X^{\alpha, \vartheta, \br}$ denote the version of $X^{\alpha, \vartheta}$ under the conditional law $\P(\, \cdot \mid X^{\alpha,\vartheta}_1=0)$.
Finally let $X^{\alpha, \vartheta, \exc}$ denote the excursion obtained by applying the Vervaat transform~\eqref{eq:Vervaat} to $X^{\alpha, \vartheta, \br}$.

Recall from Section~\ref{ssec:def_discret}  the definition of Boltzmann random planar maps. From the bijection recalled there, Boltzmann maps with $n$ edges relate to certain (loop)trees and excursions with duration $n+1$ of downward skip-free random walks with step distribution $(\nu_q(k) ; k \ge -1)$ given by:
\[\nu_q(-1) = \frac{1}{Z_q}
\qquad\text{and}\qquad
\nu_q(k) = Z_q^k \binom{2k+1}{k+1} q_{k+1}
\qquad\text{for}\enskip k \ge 0
,\]
provided such a $Z_q > 1$ exists to ensure that $\nu_q$ is indeed a probability.
When $\nu_q$ is centred and has finite variance, the corresponding random walk excursion converges after scaling to the Brownian excursion, and accordingly, the Boltzmann maps conditioned to have many edges converge to the Brownian sphere, see e.g.~\cite[Section~6.4]{Mar19} for a more general statement.
However, when the probability $\nu_q$ is centred  and has the asymptotic behaviour $\nu_q(k) \sim c k^{-1-\alpha}$ as $k \to \infty$ for some unimportant constant $c>0$ and an (important) index $\alpha \in (1,2)$, such as in the explicit example~\eqref{eq:explicit} in the introduction, then the excursion converges to that of $X^\alpha$. In this regime, Le~Gall \&~Miermont~\cite{LGM11} proved Theorem~\ref{thm:convergence_cartes_Boltzmann_Levy} (i) and (ii) here, especially the convergence after extraction of a subsequence to a limiting space; and they also computed the Hausdorff dimension $2\alpha$ of these limits.
This convergence, after extraction of a subsequence, extends to the whole domain of attraction of an $\alpha$-stable law~\cite[Section~6.4]{Mar19} and it was very recently proved by Curien, Miermont, \& Riera~\cite{CMR25} that the convergence holds 
without the need of extracting subsequences.

In the aforementioned bijection, the vertices of a Boltzmann map relate to the negative jumps of the $\nu_q$-excursion. From this, one can prove that the number of vertices in such a map conditioned to have $n$ edges concentrates around $\nu_q(-1) n$, see once again~\cite[Section~6.4]{Mar19}.
We have considered in~\cite{KM23} maps $M^{n,K_n}$ with weights in the $\alpha$-stable regime as above and multiconditioned to have $n$ edges and $K_n$ vertices. 
The corresponding excursion $W^{n,K_n}$ has the law of a $\nu_q$-random excursion conditioned both to have duration $n+1$ and a total $K_n-1$ of negative increments.
Then Theorem~4.4(i) in~\cite{KM23} shows that if we set $r_n = (\alpha \Gamma(-\alpha) n)^{1/\alpha}$ and if there exists $\vartheta \in \R$ such that
\begin{equation}\label{eq:regime_K_n}
\frac{K_n - \nu_q(-1) n}{r_n (1-\nu_q(-1))} \cv \vartheta
,\end{equation}
then
\begin{equation}\label{eq:CVluka_distorted}
(r_n^{-1} W^{n,K_n}_{\floor{nt}})_{t \in [0,1]} \cvloi X^{\alpha,\vartheta,\exc}
.\end{equation}
Let us mention that the statement in~\cite{KM23} actually applies more generally to non-centred distributions, and the entire domain of attraction of a stable law.
The case $\alpha=2$ is also treated there, although in this regime the limit $X^{2,\vartheta,\exc}$ is simply the Brownian excursion independently of $\vartheta$, as it can be checked using Girsanov's formula.

Let $\Loopd(W^{n,K_n})$ be the discrete looptree coded by the {\L}ukasiewicz path $W^{n,K_n}$ in the sense of Section~\ref{ssec:def_discret}.
By a direct application of our previous results, namely theorems~\ref{thm:convergence_looptrees_labels_Levy} and~\ref{thm:dimensions_fractales_Looptrees} for looptrees and theorems~\ref{thm:convergence_cartes_Boltzmann_Levy} and~\ref{thm:dimensions_cartes_Levy} for maps we infer the following results.

\begin{cor}\label{cor:looptrees_maps_stable_drift}
Let $\alpha \in (1,2)$ and $\vartheta \in (-\infty, \infty)$ and let $K_n$ satisfy~\eqref{eq:regime_K_n}. Then
\[r_{n}^{-1}\, \Loopd(W^{n,K_n}) \cvloi \Loop(X^{\alpha, \vartheta, \exc}),\]
and the limit satisfies almost surely:
\[\dim_{H} \Loop(X^{\alpha, \vartheta, \exc}) = \dim_{p} \Loop(X^{\alpha, \vartheta, \exc}) = \diminf \Loop(X^{\alpha, \vartheta, \exc}) = \dimsup \Loop(X^{\alpha, \vartheta, \exc}) = \alpha.\]
In addition, the sequence $((2r_{n})^{-1/2} M^{n,K_n})_n$ is tight and every subsequential limit, say $\Map^{\alpha, \vartheta}$, has almost surely:
\[\dim_{H} \Map^{\alpha, \vartheta} = \dim_{p} \Map^{\alpha, \vartheta} = \diminf \Map^{\alpha, \vartheta} = \dimsup \Map^{\alpha, \vartheta} = 2\alpha.\]
\end{cor}

One can also derive the convergence without extraction of the profile and some other statistics on the distances in the map by application of Theorem~\ref{thm:convergence_cartes_Boltzmann_Levy}.

A motivation of this work, following~\cite{KM23}, was to make explicit this result, based on the convergence~\eqref{eq:CVluka_distorted}.
In a companion paper, we further study the behaviour of $\Loop(X^{\alpha, \vartheta, \exc})$ and $\Map^{\alpha, \vartheta}$ as $\alpha$ and $\vartheta$ vary, partly based  on the result in Section~\ref{sec:limites_Levy} below.
We prove that the corresponding looptrees interpolate as $\vartheta$ varies from $-\infty$ to $\infty$, or as $\alpha$ varies from $1$ to $2$, between a circle and the Brownian tree, whereas the corresponding maps interpolate between the Brownian tree and the Brownian sphere.

\section{A spinal decomposition and some volume bounds}
\label{sec:volume}

Throughout this section we let $X$ denote a L\'evy process with Laplace exponent $\psi$ and $X^{\exc}$ its excursion version.
The main results of this section are the following volume bounds. They will be the key to establish the lower bounds on the dimensions of looptrees and maps in the subsequent section. Recall that both the looptrees and maps are given by quotient of the interval $[0,1]$ by a pseudo-distance so we naturally identify times in this interval with their projection in the looptrees and maps.

\begin{prop}\label{prop:borne_sup_volume_boules_looptrees_psi}
Fix $\delta > 0$ and let $U$ be a random time with the uniform distribution on $[0,1]$ and independent of $X^{\exc}$.
Then almost surely, for every $n$ large enough, the ball of radius $2^{-n}$ centred at the image of $U$ in the looptree $\Loop(X^{\exc})$ has volume less than $\psi(2^{n(1-\delta)})^{-1}$.
\end{prop}

We then turn to maps. Recall that we let $(\Map,d_\Map,p_\Map)$ denote a subsequential limit of discrete random maps, related to a pair $(X^{\exc}, Z^{\exc})$, from Theorem~\ref{thm:convergence_looptrees_labels_Levy} and Theorem~\ref{thm:convergence_cartes_Boltzmann_Levy}. In addition, thanks to Skorokhod's representation theorem, we assume that all convergences in these theorems hold in the almost sure sense.
Notice the additional square inside $\psi$ in the next statement compared to the first one, which eventually leads to the factor $1/2$ in the dimensions.

\begin{prop}\label{prop:borne_sup_volume_boules_cartes_psi}
Fix $\delta > 0$ and let $U$ be a random time with the uniform distribution on $[0,1]$ and independent of $X^{\exc}$.
Then almost surely, for every $n$ large enough, the ball of radius $2^{-n}$ centred at the image of $U$ in the map $\Map$ has volume less than $\psi(2^{2n(1-\delta)})^{-1}$.
\end{prop}

In order to prove these results, we shall first describe in Section~\ref{ssec:epine} the path from the root to a uniformly random point in the L\'evy looptree, as well as the labels on this path.
Then Proposition~\ref{prop:borne_sup_volume_boules_looptrees_psi} is proved in Section~\ref{sec:volume_boules_looptrees}, whereas Proposition~\ref{prop:borne_sup_volume_boules_cartes_psi} is proved in Section~\ref{ssec:distances_carte}, relying on a technical estimates on Brownian paths stated in Section~\ref{sec:technique} and proved in Section~\ref{sec:appendice_volume}.

\subsection{A spinal decomposition}
\label{ssec:epine}

Let us first describe the equivalence classes in the looptree.

\begin{lem}\label{lem:identifications_looptree_Levy}
Fix $a \geq 0$. Almost surely, for every $0 \le s < t \le 1$, we have:
\[d^a_{\Loop(X^{\exc})}(s,t) = 0 
\iff 
\begin{cases}
\text{either}\enskip &\Delta X^{\exc}_s > 0 \enskip\text{and}\enskip t = \inf\{r>s \colon X^{\exc}_{r} = X^{\exc}_{s-}\},
\\
\enskip\text{or}\enskip &\Delta X^{\exc}_s = 0 \enskip\text{and}\enskip X^{\exc}_{s} = X^{\exc}_{t} = \inf_{[s,t]} X^{\exc}.
\end{cases}
\]
Moreover, almost surely for any $s \in [0,1]$, the time $t$ is unique in the first case, and there are either one or two such times $t$ in the second case.
\end{lem}

\begin{proof}
Let us first consider a general c\`adl\`ag function $Y$ and recall that the looptree distance $d^a_{\Loop(Y)}$ is defined in~\eqref{eq:def_distance_looptree}. Then $d^a_{\Loop(Y)}(s,t) = 0$ when all terms on the right-hand side of~\eqref{eq:def_distance_looptree} vanish, namely:
\begin{enumerate}
\item For every $r$ such that $s \wedge t \prec r \prec s$, either $\inf_{[r,s]} Y = Y_r$ or $\inf_{[r,s]} Y = Y_{r-}$;
\item For every $r$ such that $s \wedge t \prec r \prec t$, either $\inf_{[r,t]} Y = Y_r$ or $\inf_{[r,t]} Y = Y_{r-}$;
\item Either $\inf_{[s \wedge t, s]} Y = \inf_{[s \wedge t, t]} Y$ or both $Y_{s \wedge t} = \inf_{[s \wedge t, s]} Y$ and $Y_{(s \wedge t)-} = \inf_{[s \wedge t, t]} Y$;
\item Finally $C_s = C_t = \inf_{[s,t]} C$.
\end{enumerate}
Note that in the case $s\wedge t = s$, or equivalently when $s \prec t$, the first two items are empty and the third one reduces to requiring either $Y_s = \inf_{[s, t]} Y$ or $Y_{s-} = \inf_{[s, t]} Y$.
It is straightforward to check that these properties hold in each case on the right of the claim, let us prove the direct implication when $Y = X^{\exc}$ by relying on some properties of such a path.

Consider $X$ the unconditioned L\'evy process.
First almost surely if $\Delta X_r > 0$, then for every $\varepsilon>0$, we have both $\inf_{[r-\varepsilon, r]} X < X_{r-}$ and $\inf_{[r, r+\varepsilon]} X < X_{r}$. The latter follows from the Markov property and the fact that $0$ is regular for $(-\infty, 0)$ and the former by invariance under time reversal and the fact that $0$ is regular for $(0,\infty)$, see Theorem~VII.1 and Corollary~VII.5 in~\cite{Ber96}.
Consequently almost surely if $X$ achieves a local minimum, then it does not jump at this time.
Next, the strong Markov property also entails that the local minima are unique.
Finally, by~\cite[Proposition~III.2]{Ber96}, we know that almost surely every weak record of $X$ obtained by a jump is actually a strict record. Applied to the dual process at a time just after $t$, it shows that almost surely if $X$ realises a local minimum at time $t$ and if $r = \sup\{u < t \colon X_u < X_t\}$ has $\Delta X_r > 0$, then $X_{r-} < X_t < X_r$.
All these properties transfer to the bridge $X^{\br}$ by absolute continuity and then to the excursion $X^{\exc}$ by the Vervaat transform.
In the rest of this proof, we implicitly assume that they hold true.

Fix $s<t$ such that $d^a_{\Loop(X^{\exc})}(s,t) = 0$.
We first claim that necessarily $s \prec t$.
Indeed, if this fails, then there exists $u \in (s,t)$ such that $X^{\exc}_u = \inf_{[s,t]} X^{\exc} < X^{\exc}_{s-}$. Let $r = \sup\{v < u \colon X^{\exc}_v < X^{\exc}_{u}\}$.
The path $X^{\exc}$ makes a local minimum at time $u$ so $X^{\exc}_{r-} < X^{\exc}_u < X^{\exc}_r$. 
By construction we have
\[\inf_{[r,s]} X^{\exc} > X^{\exc}_u \ge \inf_{[r,t]} X^{\exc}\]
and by the properties of $X^{\exc}$, since $r$ is a jump time, then
\[X^{\exc}_r > \inf_{[r,s]} X^{\exc}.\]
Notice that $r = s \wedge t$ so this violates the third requirement of $d^a_{\Loop(X^{\exc})}(s,t) = 0$ in the list in the beginning of this proof.

Therefore, we must have $s \prec t$, namely $X^{\exc}_{s-} \leq \inf_{[s,t]} X^{\exc}$.
The condition $d^a_{\Loop(X^{\exc})}(s,t) = 0$ thus implies $\inf_{[s, t]} X^{\exc}$ be equal to either $X^{\exc}_s$ or to $X^{\exc}_{s-}$.
By the properties of $X^{\exc}$, if $\Delta X^{\exc}_s > 0$, then first $\inf_{[s, t]} X^{\exc} < X^{\exc}_s$, and second, at time $u = \inf\{r>s \colon X^{\exc}_r = X^{\exc}_{s-}\}$, the path $X^{\exc}$ cannot make a local minimum, so we must have $u=t$, which is the first alternative in our claim. In addition, this time $t$ is unique.

Suppose finally that $\Delta X^{\exc}_s = 0$.
Then there is no $r \in (s,t)$ such that $X^{\exc}$ achieves a local minimum at time $r$ with $X^{\exc}_r=X^{\exc}_s$, nor $X^{\exc}_s = X^{\exc}_{r-} < X^{\exc}_r$.
Therefore if $X^{\exc}_{t-} > X^{\exc}_s$, and since $d_C(s,t)=0$, then there exists $r \in (s,t)$ such that $\Delta X^{\exc}_r > 0$ and $X^{\exc}_s < X^{\exc}_{r-}$ and $X^{\exc}_{t-} < X^{\exc}_r$, which contributes to the second term in~\eqref{eq:def_distance_looptree}. We conclude that in this case we must have $X^{\exc}_{t-}=X^{\exc}_s$. Furthermore, the time $t$ cannot be a jump time, for otherwise $X^{\exc}$ would reach smaller values just before, hence $X^{\exc}_t=X^{\exc}_s = \inf_{[s,t]} X^{\exc}$. 
Finally there can only be at most two such times $t$, if $\inf\{t>s \colon X^{\exc}_t = X^{\exc}_s\}$ is a time of a local minimum since then it is unique at this height.
\end{proof}

Let us now describe the ``chain of loops'' from the root $0$ to a given point $u \in [0,1]$ in the looptree, see Figure~\ref{fig:epine_looptrees} for a graphical representation. 
First let $A_u$ be the set of pairs of times $(a,b) \in [0,u] \times [u,1]$ such that:
\begin{equation}\label{eq:def_ancetres}
\text{either}\quad 
X^{\exc}_a = X^{\exc}_b = \min_{[a,b]} X^{\exc}
\quad\text{or}\quad 
b = \inf\{t > a \colon X^{\exc}_t = X^{\exc}_{a-}\}.
\end{equation}
By Lemma~\ref{lem:identifications_looptree_Levy}, pairs $(a,b) \in A_u$ are identified in the looptree. These points intuitively correspond to the ancestors of $u$ in the tree. Define then the set of pairs of points that belong to the same cycle on the geodesic from $u$ to the root, with one point on the left and one on the right. 
Formally, let $B_u$ denote the set of pairs of times $(a,b) \in [0,u] \times [u,1]$ for which there exists $s<a$ such that:
\begin{equation}\label{eq:epine_gauche_droite}
X^{\exc}_{s} > X^{\exc}_{a-} = \min_{[s,a-]} X^{\exc} > X^{\exc}_{b-} = \min_{[s,b-]} X^{\exc} > X^{\exc}_{s-}
.\end{equation}
Pairs in $(a,b) \in B_u$ are not identified in the looptree (still by Lemma~\ref{lem:identifications_looptree_Levy}). 
Instead, if $s$ in the time in~\eqref{eq:epine_gauche_droite} and $t = \inf\{v > s \colon X^{\exc}_v = X^{\exc}_{s-}\} > u$, then $(s,t) \in A_u$, so they correspond to an ancestor of $u$.
This ancestor has a large offspring number $\Delta X^{\exc}_s$ which creates a cycle in the looptree. One point on this cycle corresponds to the next ancestor of $u$ and splits the cycle into a left part, to which $a$ belongs, and a right part, to which $b$ belongs.
One can easily prove that almost surely, for any $u \in [0,1]$ the set $A_u$ is closed and $A_u \cup B_u$ is the closure of $B_u$.

For any given $\varepsilon>0$ fixed, let us consider the set:
\[B_u^\varepsilon = \{(s,t) \in A_u \cup B_u \colon d_{\Loop(X^{\exc})}(s,u) \leq \varepsilon \enskip\text{and}\enskip d_{\Loop(X^{\exc})}(t,u) \leq \varepsilon\}.\]
Let us construct an extremal pair $(u^-_\varepsilon, u^+_\varepsilon) \in B_u^\varepsilon$, see again Figure~\ref{fig:epine_looptrees} for a graphical representation.
First let $s_\varepsilon = \sup\{s < u \colon X^{\exc}_{s-} \leq \min_{[s,u]} X^{\exc} \text{ and } d_{\Loop(X^{\exc})}(s,u) \geq \varepsilon\}$, which codes the last ancestor of $u$ (starting from the root) at $d_{\Loop(X^{\exc})}$-distance at least $\varepsilon$ from $u$.
\begin{enumerate}
\item If $d_{\Loop(X^{\exc})}(s_\varepsilon,u) = \varepsilon$, then define
\begin{equation}
\label{eq:uplus1}
\begin{aligned}
u^-_\varepsilon &= \inf\{s < u \colon X^{\exc}_{s-} = X^{\exc}_{s_\varepsilon-} = \min_{[s,u]} X^{\exc}\}
\\
u^+_\varepsilon &= \inf\{t > u \colon X^{\exc}_{t} < X^{\exc}_{s-}\}
\end{aligned}
\end{equation}
so $u^-_\varepsilon \leq s_\varepsilon$ with a strict inequality when $X^{\exc}$ realizes a local minimum at time $s_\varepsilon$ and that it previously crossed continuously this level.
Then $(u^-_\varepsilon, u^+_\varepsilon) \in A_u$ satisfies $d_{\Loop(X^{\exc})}(u, u^-_\varepsilon) = d_{\Loop(X^{\exc})}(u, u^+_\varepsilon) = \varepsilon$.

\item Otherwise if $d_{\Loop(X^{\exc})}(s_\varepsilon,u) > \varepsilon$, then necessarily $\Delta X^{\exc}_{s_\varepsilon} > 0$. Let $t_\varepsilon = \inf\{t > s_\varepsilon \colon X^{\exc}_{t-} \leq \min_{[t,u]} X^{\exc}\}$ be the next ancestor of $u$, then $d_{\Loop(X^{\exc})}(t_\varepsilon,u) < \varepsilon$ and every time $a \in [s_\varepsilon, t_\varepsilon]$ such that $X^{\exc}_a = \min_{[s_\varepsilon, a]} X^{\exc}$ has $d_{\Loop(X^{\exc})}(a,u) = d_{\Loop(X^{\exc})}(t_\varepsilon,u) + (X^{\exc}_a - X^{\exc}_{t_\varepsilon-})$. We then define:
\begin{align*}
u^-_\varepsilon 
&= \inf\{a > s_\varepsilon \colon X^{\exc}_{a} \leq X^{\exc}_{s_\varepsilon} - (d_{\Loop(X^{\exc})}(s_\varepsilon,u)-\varepsilon)\}
\\
&= \inf\{a > s_\varepsilon \colon X^{\exc}_{a} \leq X^{\exc}_{t_\varepsilon-} + (\varepsilon-d_{\Loop(X^{\exc})}(t_\varepsilon,u))\}
,\end{align*}
which is the first point in the looptree when following the left part of the cycle from $s_\varepsilon$ to $t_\varepsilon$ which lies at looptree distance $\varepsilon$ from $u$.
Finally we let
\begin{equation}
\label{eq:uplus2}
u^+_\varepsilon 
= \inf\{b > u \colon X^{\exc}_{b} < X^{\exc}_{t_\varepsilon-} + (\varepsilon-d_{\Loop(X^{\exc})}(t_\varepsilon,u))\}
,
\end{equation}
which is now the first point in the looptree when following the right part of the cycle from $s_\varepsilon$ to $t_\varepsilon$ which lies at looptree distance $\varepsilon$ from $u$ (equivalently, the last one when going from $t_\varepsilon$ to $s_\varepsilon$).
Then $(u^-_\varepsilon, u^+_\varepsilon) \in B_u$ satisfies $d_{\Loop(X^{\exc})}(u^-_\varepsilon) = d_{\Loop(X^{\exc})}(u^+_\varepsilon) = \varepsilon$.
\end{enumerate}

\begin{figure}[!ht] \centering
\includegraphics[page=10, height=7cm]{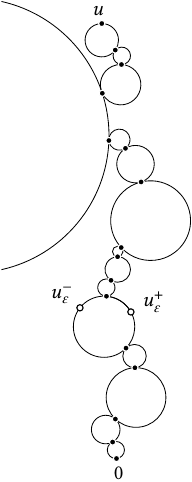}
\qquad\qquad\qquad
\includegraphics[page=11, height=7cm]{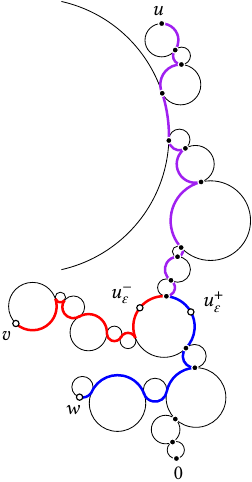}
\caption{Left: A schematic representation of 
the spine from a time $u$ to the root $0$ (a very big loop on the left has been cropped). The black dots correspond to ancestors of $u$, i.e.~points in $A_u$, whereas the 
cycles correspond to the points in $B_u$. 
The two white points $u^-_\varepsilon$ and $u^+_\varepsilon$ belong to the same cycle and lie at distance precisely $\varepsilon$ from $u$.
Right: The geodesic from $v$ to $u$ and from $w$ to $u$ are indicated in fat coloured paths (red and blue, as well as purple for their common part). Since $v$ and $w$ are visited before $u^-_\varepsilon$ then one must necessarily go through either $u^-_\varepsilon$ or $u^+_\varepsilon$ in order to reach $u$, so they lie outside the ball centred at $u$ and with radius $\varepsilon$.
}
\label{fig:epine_looptrees}
\end{figure}

The following geometric lemma should be clear on a picture, see e.g.~Figure~\ref{fig:epine_looptrees} right.

\begin{lem}\label{lem:geom_looptree}
Almost surely, the ball of radius $\varepsilon$ centred at $u$ in the looptree is contained in the interval $[u^-_\varepsilon, u^+_\varepsilon]$.
\end{lem}

It should also be clear from Figure~\ref{fig:epine_looptrees} that the inclusion is strict in general; but it turns out that this rough inclusion will give us the correct upper bound for computing the dimensions.

\begin{proof}
In the first case above, when $(u^-_\varepsilon, u^+_\varepsilon) \in A_u$, the path $X^{\exc}$ achieves values smaller than $X^{\exc}_{u^-_\varepsilon-} = X^{\exc}_{u^+_\varepsilon}$ immediately before $u^-_\varepsilon$ and immediately after $u^+_\varepsilon$. We infer from the definition of $d_{\Loop(X)}$ that any time $t < u^-_\varepsilon$ or $t > u^+_\varepsilon$ has $d_{\Loop(X)}(t,u) > d_{\Loop(X)}(u^-_\varepsilon, u) = d_{\Loop(X)}(u^+_\varepsilon, u) = \varepsilon$.
For the same reason, in the second case, any time $t > u^+_\varepsilon$ has $d_{\Loop(X)}(t,u) > d_{\Loop(X)}(u^+_\varepsilon, u) = \varepsilon$ and any time $t < s_\varepsilon$ has $d_{\Loop(X)}(t,u) > d_{\Loop(X)}(s_\varepsilon, u) > \varepsilon$. Finally, times $t \in (s_\varepsilon, u^-_\varepsilon)$ have $d_{\Loop(X)}(t,u) > \min\{d_{\Loop(X)}(s_\varepsilon, u), d_{\Loop(X)}(u^-_\varepsilon, u)\} = d_{\Loop(X)}(u^-_\varepsilon, u) = \varepsilon$.
This shows that the complement of the interval $[u^-_\varepsilon, u^+_\varepsilon]$ is contained in the complement of the ball of radius $\varepsilon$ centred at $u$ in the looptree.
\end{proof}

\subsection{A volume estimate in the looptree}
\label{sec:volume_boules_looptrees}

Recall the claim of Proposition~\ref{prop:borne_sup_volume_boules_looptrees_psi} on the volume of balls in the looptree. Given Lemma~\ref{lem:geom_looptree}, it suffices to upper bound the length of the interval $[U^-_\varepsilon, U^+_\varepsilon]$ where $U$ is a uniform random time independent of $X^{\exc}$.
In order to make calculations, we shall rely on the local absolutely continuity between the excursion and the unconditioned process, which simply comes from the construction of the excursion. 
Precisely, let $(\overleftarrow{X}_s)_{s \geq 0}$ and $(\overrightarrow{X}_t)_{t \geq 0}$ be two independent copies of the L\'evy process $X$, then by~\eqref{eq:abs_cont} and~\eqref{eq:Vervaat}, the pair 
$(X^{\exc}_U-X^{\exc}_{(U-s)-}, X^{\exc}_{t+U})_{s \in [0,U], t \in [0,1-U]}$ 
is locally absolutely continuous with respect to $(\overleftarrow{X}_s, \overrightarrow{X}_t)_{s,t \geq 0}$ with bounded density.

We shall keep the same notation as in the preceding subsection and consider the sets $A_0$, $B_0$, and $B^\varepsilon_0$ defined as above for $(-\overleftarrow{X}_{s-}, \overrightarrow{X}_t)_{s,t \geq 0}$.
In particular, the first elements of pairs in $A_0$ are given by the negative of the set of times at which $(\overleftarrow{X}_s)_{s \geq 0}$ realises a weak record. Let us thus consider the running supremum process $(S_t)_{t \geq 0}$ and the local time at the supremum $(L_t)_{t \geq 0}$ of the process $(\overleftarrow{X}_t)_{t \geq 0}$. We henceforth drop the arrow for better readability.
According to~\cite[Equation~96]{DLG02} the random point measure $\mathscr{N}(\d t \d x \d u) = \sum_{i \in I} \delta_{(t_{i}, x_{i}, u_{i})}$ defined by:
\begin{equation}\label{eq:mesure_poisson_epine}
\mathscr{N}(\d t \d x \d u) = \sum_{s: S_{s} > S_{s-}} \delta_{(L_{s}, S_{s} - X_{s-}, (S_{s-} - X_{s-})/(S_{s} - X_{s-}))}(\d t \d x \d u)
\end{equation}
is a Poisson random measure on $[0,\infty) \times [0,\infty) \times [0,1]$ with intensity $\d t \otimes x \pi(\d x) \otimes \d u$.

Let us next define four subordinators by $X'$, $X^R$, $\widetilde{X}$, and $\sigma$ by setting for every $t > 0$:
\[X'_{t} = \beta t + \sum_{t_{i} \leq t} x_{i}
\qquad\text{and}\qquad 
X^R_{t} = \beta t +\sum_{t_{i} \leq t} u_{i} x_{i},\]
as well as
\[\widetilde{X}_{t} = \beta t +\sum_{t_{i} \leq t} \min\{u_{i}, 1-u_{i}\} x_{i}
\qquad\text{and}\qquad 
\sigma_{t} = \beta t +\sum_{t_{i} \leq t} u_{i} (1-u_{i}) x_{i}.\]
Notice that we give them all the same drift $\beta$ for the continuous part.
Informally, in the infinite looptree coded by $(\overleftarrow{X}, \overrightarrow{X})$ with a distinguished path of loops coded by $A_0 \cup B_0$, when reading this path starting from $0$, the jumps of process $X'$ code the total length of the loops on this path, those of $X^R$ code their right length, those of $\widetilde{X}$ the shortest length among left and right, so related to the looptree distance, and finally those of $\sigma$ the product of the left and right lengths, which is related to a Brownian bridge on the loop.

\begin{lem}\label{lem:loi_epine_looptree}
The Laplace exponents of these processes are given by:
\[\psi'(\lambda) - \beta \lambda \,\,\, \text{for } X',
\,\,\,
\widetilde{\psi}(\lambda) \,\,\,\text{for } X^R,
\,\,\,
\widetilde{\psi}(\lambda/2) + \beta \lambda / 2 \,\,\,\text{for } \widetilde{X},
\,\,\,
\text{and}
\,\,\,
\phi(\lambda)+\beta \lambda \,\,\text{for } \sigma
,\]
where $\psi'$ simply denotes the derivative of $\psi$ and $\widetilde{\psi}(\lambda) = \psi(\lambda) / \lambda$ and finally $\phi(\lambda) = \int_{0}^{\infty} x \pi(\d x) \int_{0}^{1} \d u (1 - \exp(- \lambda u(1-u) x))$.
\end{lem}

\begin{proof}
Recall the Poisson random measure $\mathscr{N}$ from~\eqref{eq:mesure_poisson_epine}, then for every $t,\lambda > 0$, by exchanging the differentiation and integral, we get:
\[- \frac{1}{t} \log \E\Bigl[\e^{- \sum_{t_{i} \leq t} x_{i}}\Bigr]
= \int_{0}^{\infty} (1-\e^{-\lambda x}) x \pi(\d x)
= \frac{\d}{\d\lambda} \int_{0}^{\infty} (\e^{-\lambda x} - 1 + \lambda x) \pi(\d x),
\]
which is equal to $\psi'(\lambda) - 2\beta\lambda$.
For $X^R$, we have:
\[- \frac{1}{t} \log \E\Bigl[ \e^{- \sum_{t_{i} \leq t} u_{i} x_{i}}\Bigr]
= \int_{0}^{\infty} x \pi(\d x) \int_0^1 (1-\e^{-\lambda u x}) \d u
= \frac{1}{\lambda} \int_{0}^{\infty} \pi(\d x) (\lambda x - 1 + \e^{-\lambda x}),\]
which is equal to $\widetilde{\psi}(\lambda) - \beta\lambda$.

Noticing that the random variables $\min\{u_{i}, 1-u_{i}\}$ have density $2\,\mathbf{1}_{[0,1/2]}$, and that 
$\int_{0}^{1/2} 2 \d v (1-\e^{-\lambda v x})
= 1 + \frac{2}{\lambda x} (\e^{-\lambda x/2} - 1)
= \frac{2}{\lambda} (\e^{-\lambda x/2} - 1 + \frac{\lambda x}{2})$, 
we have similarly:
\[- \frac{1}{t} \log \E\Bigl[\exp\Bigl(- \sum_{t_{i} \leq t} x_{i} \min\{u_{i}, 1-u_{i}\}\Bigr)\Bigr]
= \int_{0}^{\infty} x \pi(\d x) \int_{0}^{1/2} 2 \d v (1-\e^{-\lambda v x}),
\]
which is equal to $\widetilde{\psi}(\frac{\lambda}{2}) - \frac{\beta \lambda}{2}$.

Finally
\[- \frac{1}{t} \log \E\Bigl[\exp\Bigl(- \sum_{t_{i} \leq t} x_{i} u_{i} (1-u_{i})\Bigr)\Bigr]
= \int_{0}^{\infty} x \pi(\d x) \int_0^1 (1-\e^{-\lambda u(1-u) x}) \d u
= \phi(\lambda).\]
This characterises the law of these subordinators.
\end{proof}

The following result claims that in the infinite model, on the distinguished chain of loops, the sum of the lengths of the loops we encounter is close to the sum of the shortest between their left and right length. 

\begin{lem}\label{lem:comparaison_distance_looptrees_tailles_boucles}
Let $\tau_\varepsilon = \inf\{t \geq 0 \colon \widetilde{X}_{t} > \varepsilon\}$ for every $\varepsilon>0$.
Fix $\delta \in (0, 1)$, then almost surely, for every $\varepsilon>0$ small enough, it holds $X^R_{\tau_{\varepsilon}-} \leq X'_{\tau_{\varepsilon}-} \leq \varepsilon^{1-\delta}$.
\end{lem}

\begin{proof}
Clearly $X^R \leq X'$ so we focus on the second bound. Also, it is enough to prove the claim assuming that $\beta=0$. Indeed, when $\beta>0$, since $\beta$ is the common drift coefficient to both $X'$ and $\widetilde{X}$, then by adding this contribution plus that of the jumps of $\widetilde{X}$, we infer from the $\beta=0$ case that $X'_{\tau_{\varepsilon}-} \leq \varepsilon^{1-\delta} + \widetilde{X}_{\tau_{\varepsilon}-}\leq \varepsilon^{1-\delta} + \varepsilon$.
We now assume that $\beta=0$. To simplify notation, we set $v_{i}=\min(u_{i},1-u_{i})$ and $y_{i}= v_{i} x_{i}$.
For $K>1$, let us decompose the process $X'$ by introducing a cutoff for small values of the $v_{i}$'s as follows: for every $t>0$, write
\[X'_{t} = \sum_{t_{i} \leq t} \frac{y_{i}}{v_{i}} \ind{v_{i} \geq 1/(2K)} + \sum_{t_{i} \leq t} \frac{y_{i}}{v_{i}} \ind{v_{i} < 1/(2K)}.\]
Also, since the $v_{i}$'s have density $2\,\mathbf{1}_{[0,1/2]}$, we have
\[\P\Bigl(v_{i} \leq  \frac{1}{2K}\Bigr) = \frac{1}{K}
\qquad\text{and}\qquad 
\E\Bigl[\frac{1}{v_{i}} \ind{v_{i} \geq 1/(2K)}\Bigr] = \int_{ 1/(2K)}^{ 1/2} \frac{2}{x} \d x = 2 \log K.\]
We claim that for $K$ well-chosen, the first sum in the first display will be close to the sum of the $y_{i}$'s, whereas in the second sum, there are too few terms and this sum will be small.

Fix $\delta > \delta' > 0$ and let $K_\varepsilon = \exp(\varepsilon^{-\delta'})$. Then the Markov inequality implies:
\begin{align*}
&\P\Bigl(\sum_{t_{i} \leq \tau_\varepsilon-} x_{i} \ind{v_{i} \geq \frac{1}{2K_\varepsilon}} \geq \varepsilon^{1-\delta}\bigm|(t_{i},y_{i})_{i \in I}\Bigr)\\
 & \qquad \qquad \leq \P\Bigl(\sum_{t_{i} \leq \tau_\varepsilon-} \frac{1}{v_{i}} y_{i} \ind{v_{i} \geq \frac{1}{2K_\varepsilon}} \geq \varepsilon^{-\delta} \sum_{t_{i} \leq \tau_\varepsilon-} y_{i} \bigm| (t_{i},y_{i})_{i \in I}\Bigr)
\leq 2 \varepsilon^{\delta} \log K_\varepsilon
\leq 2 \varepsilon^{\delta-\delta'}.
\end{align*}
Note that the sum of the last quantity over all values of $\varepsilon$ of the form $2^{-n}$ with $n \in \N$ is finite.

We next consider the other sum in the decomposition above, when the $v_{i}$'s are smaller than $1/(2K_\varepsilon)$.
Recall that we assume that $\beta=0$ and note in this case the identity in law (as processes):
\[\sum_{t_{i} \leq t} x_{i} \ind{v_{i} < \frac{1}{2 K_\varepsilon}} \eqloi X'_{t \P(v_{i} < \frac{1}{2K_\varepsilon})} = X'_{t / K_\varepsilon}.\]
Recall from Lemma~\ref{lem:loi_epine_looptree} that $X'$ has Laplace exponent $\psi'$, then for every $t,x> 0$,
we have by the Markov property and the upper bound $1-\e^{-y} \leq y$:
\[\P\Bigl(\sum_{t_{i} \leq t} x_{i} \ind{v_{i} < \frac{1}{2K_\varepsilon}} > x\Bigr)
= \P(1 - \e^{X'_{t / K_\varepsilon} / x} > 1-\e^{-1})
\leq \frac{\e}{\e-1} \bigl(1 - \E\bigl[\e^{-X'_{t / K_\varepsilon} / x}\bigr]\bigr),
\]
which is at most $\frac{\e}{\e-1} \frac{t}{K_\varepsilon} \psi'(1/ x)$.
On the other hand, since $\widetilde{X}$ has Laplace exponent $\widetilde{\psi}(\cdot/2)$,
\[\P(\tau_\varepsilon \geq t)
= \P(\widetilde{X}_{t} \leq \varepsilon)
= \P\bigl(\e^{-\widetilde{X}_{t} / \varepsilon} \geq \e^{-1}\bigr)
\leq \e \cdot \E\bigl[\e^{- \widetilde{X}_{t}/\varepsilon}\bigr]
= \e \cdot \e^{- t \widetilde{\psi}(1/(2\varepsilon))}.\]
We infer that for any $t>0$, we have:
\[\P\Bigl(\sum_{t_{i} \leq \tau_\varepsilon-} x_{i} \ind{v_{i} < \frac{1}{2K_\varepsilon}} > \varepsilon^{1-\delta}\Bigr)
\leq \frac{\e}{\e-1} t K_\varepsilon^{-1} \psi'(1/ \varepsilon^{1-\delta}) + \e \cdot \e^{- t \widetilde{\psi}(1/(2\varepsilon))}.\]
For $t = \varepsilon^{-\delta}\widetilde{\psi}(1/(2\varepsilon))^{-1}$ with $\delta > 0$, we read:
\[\P\Bigl(\sum_{t_{i} \leq \tau_\varepsilon-} x_{i} \ind{v_{i} < \frac{1}{2K_\varepsilon}} > \varepsilon^{1-\delta}\Bigr)
\leq \frac{\e}{\e-1} \frac{\psi'(1/ \varepsilon^{1-\delta})}{\varepsilon^{\delta} K_\varepsilon \widetilde{\psi}(1/(2\varepsilon))} + \e \cdot \e^{- \varepsilon^{-\delta}}.\]
Since $\widetilde{\psi}$ is increasing and $\psi'$ grows at most polynomially (in fact, its upper exponent at infinity is $\bdeta-1$), whereas $K_\varepsilon$ grows exponentially fast as $\varepsilon \downarrow 0$, then the sum of the upper bound when $\varepsilon$ ranges over the set $\{2^{-n}, n \in \N\}$ is finite.
We conclude from our two bounds and the Borel--Cantelli lemma that almost surely for every $n$ large enough, we have $X'_{\tau_{2^{-n}}-} \leq 2^{-n(1-\delta)}$.
We extend this to all values of $\varepsilon>0$ small enough by monotonicity.
\end{proof}

Recall that Proposition~\ref{prop:borne_sup_volume_boules_looptrees_psi} claims that almost surely, the volume of the ball in the looptree centred at an independent uniform random point $U$ and with radius $2^{-n}$ is less than $\psi(2^{n(1-\delta)})^{-1}$ for $n$ large enough.

\begin{proof}[Proof of Proposition~\ref{prop:borne_sup_volume_boules_looptrees_psi}]
As we have already observed, thanks to Lemma~\ref{lem:geom_looptree}, it suffices to prove that if $U$ is an independent random time with the uniform distribution on $[0,1]$, then almost surely, for every $n$ large enough, we have both
$U-U^-_{2^{-n}} \leq \psi(2^{n(1-\delta)})^{-1}$
and
$U^+_{2^{-n}} - U \leq \psi(2^{n(1-\delta)})^{-1}$.
Let us first focus on the second bound, we shall then deduce the first one by a symmetry argument.
By the local absolute continuity relation which follows from~\eqref{eq:abs_cont} and~\eqref{eq:Vervaat}, since the density is bounded and we aim for an almost sure result, it is sufficient to prove the claim after replacing $(X^{\exc}_U-X^{\exc}_{(U-s)-}, X^{\exc}_{t+U})_{s \in [0,U], t \in [0,1-U]}$ by $(\overleftarrow{X}_s, \overrightarrow{X}_t)_{s,t \geq 0}$.
In particular, we replace $U$ by $0$ and write $0^+_{\varepsilon}$ for the time $u^+_{\varepsilon}$ defined from $(\overleftarrow{X}_s, \overrightarrow{X}_t)_{s,t \geq 0}$ for $u=0$.
Relying on the Borel--Cantelli lemma, it suffices to prove that $\sum_n \P(0^+_{2^{-n}} \geq \psi(2^{n(1-\delta)})^{-1}) < \infty$ for this infinite model.

Let us rewrite the definition of $0^+_{\varepsilon}$ (recall~\eqref{eq:uplus1} and~\eqref{eq:uplus2}) in a more convenient way and let us refer to Figure~\ref{fig:epine_excursion} for a figurative representation.
Let $\widetilde{X}$ and $X^R$ be defined as previously from $\overleftarrow{X}$ instead of $X$, let $\tau_{\varepsilon} = \inf\{s > 0 \colon \widetilde{X}_s > \varepsilon\}$, and $T(x) = \inf\{t > 0 \colon \overrightarrow{X}_t \leq -x\}$. Then we have $0^+_{\varepsilon} = T(X^R_{\tau_{\varepsilon}-} + (\varepsilon-\widetilde{X}_{\tau_{\varepsilon}-}))$.

Fix $\delta>0$, then according to Lemma~\ref{lem:comparaison_distance_looptrees_tailles_boucles}, almost surely for every $\varepsilon$ small enough, it holds $X^R_{\tau_{\varepsilon}-} \leq \varepsilon^{1-\delta/4}$. We then implicitly work under the event that $X^R_{\tau_{\varepsilon}-}+(\varepsilon-\widetilde{X}_{\tau_{\varepsilon}-}) \leq \varepsilon^{1-\delta/2}$, so it remains to upper bound $T(\varepsilon^{1-\delta/2})$. First, for every $x>0$, we have by the Markov inequality:
\[\P(T(\varepsilon^{1-\delta/2}) \geq x)
= \P\bigl(1 - \e^{-T(\varepsilon^{1-\delta/2}) / x} \geq 1-\e^{-1}\bigr) 
\leq \frac{\e}{\e-1} \bigl(1 - \E\bigl[\e^{-x^{-1} T(\varepsilon^{1-\delta/2})}\bigr]\bigr)
.\]
It is well-known that $T$ is a subordinator with Laplace exponent $\Phi = \psi^{-1}$ the inverse of $\psi$, see e.g.~\cite[Theorem~VII.1]{Ber96}. Then, using also that $1-\e^{-y} \leq y$, we infer that
\[\P(T(\varepsilon^{1-\delta/2}) \geq x)
\leq \frac{\e}{\e-1} \bigl(1 - \e^{- \varepsilon^{1-\delta/2} \Phi(x^{-1})}\bigr)
\leq \frac{\e}{\e-1} \varepsilon^{1-\delta/2} \Phi(x^{-1})
.\]
Consequently, for $\varepsilon=2^{-n}$ and $x=\psi(\varepsilon^{-(1-\delta)})^{-1} = \psi(2^{n(1-\delta)})^{-1}$, we infer that
\[\P(T(2^{-n(1-\delta/2)}) \geq \psi(2^{n(1-\delta)})^{-1})
\leq \frac{\e}{\e-1} 2^{-n(1-\delta/2) + n(1-\delta)}
= \frac{\e}{\e-1} 2^{-n \delta/2}
,\]
which is a convergent series as we wanted.
This implies that almost surely, for every $n$ large enough, we have $0^+_{2^{-n}} \leq \psi(2^{n(1-\delta)})^{-1}$ in the infinite volume model. By local absolute continuity, we deduce that almost surely, for every $n$ large enough, we have $U^+_{2^{-n}} - U \leq \psi(2^{n(1-\delta)})^{-1}$ in the excursion model.

\begin{figure}[!ht] \centering
\includegraphics[page=12, width=\linewidth]{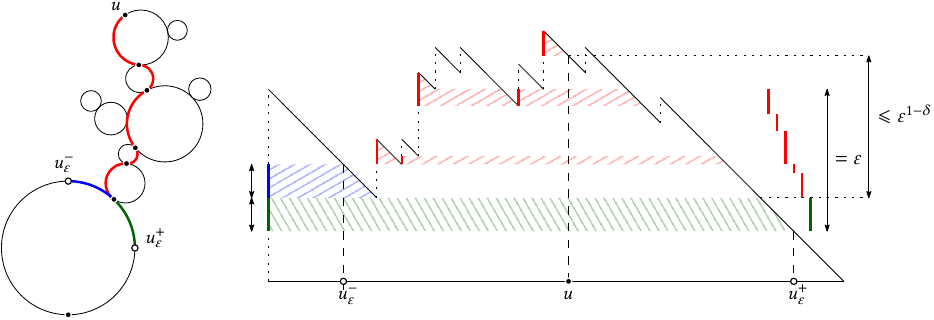}
\caption{Schematic representation of the times $u^-_\varepsilon$ and $u^+_\varepsilon$ in terms of the L\'evy excursion. On the excursion, the blue and green vertical lines on the very left have the same length. The sum of this length and that of all the red vertical lines (which correspond to the geodesic from $u$ to the root) equals precisely $\varepsilon$. This is less than $X^{\exc}_u - X^{\exc}_{u^+_\varepsilon}$, which codes the length of the path from $u$ to $u^+_\varepsilon$ in the looptree if one always follows the cycles on the spine on their right, and which is at most of order $\varepsilon^{1-\delta}$ by Lemma~\ref{lem:comparaison_distance_looptrees_tailles_boucles}.}
\label{fig:epine_excursion}
\end{figure}

Let $B(U, r)$ denote the closed ball centred at $U$ and with radius $r$ in $\Loop(X^{\exc})$ and let us denote by $|\cdot|$ the volume measure. We just proved that almost surely, for every $n$ large enough, we have $|B(U, 2^{-n}) \cap [U,1]| \leq \psi(2^{n(1-\delta)})^{-1}$.
Lemma~\ref{lem:symetrie_looptree} below, together with the symmetry of $U$, implies that the two sequences $(|B(U, 2^{-n}) \cap [0,U]|)_{n \ge 1}$ and $(|B(U, 2^{-n}) \cap [U,1]|)_{n \ge 1}$ have the same distribution, which completes the proof.
\end{proof}

We need the following lemma to conclude the proof of Proposition~\ref{prop:borne_sup_volume_boules_looptrees_psi}. It roughly says that the looptree $\Loop(X^{\exc})$ is invariant under mirror symmetry.
This property is not immediate from the construction from the excursion path $X^{\exc}$, although we believe that a suitable path construction could map $X^{\exc}$ to a process which has the same law as its time-reversal.
Also this mirror symmetry relates informally to that of the L\'evy trees from~\cite[Corollary~3.1.6]{DLG02}.
One could rely on this result (beware that it is written under the $\sigma$-finite It\=o excursion measure), and indeed the exploration process $\rho$ and its dual $\eta$, defined in the form of Equations~89 and~90 in~\cite{DLG02}, relate explicitly to the looptree distances.
However we prefer a self-contained soft argument using an approximation with discrete looptrees which are (almost) symmetric.

Recall from the previous proof that $B(u, r)$ denotes the closed ball centred at $U$ and with radius $r$ in $\Loop(X^{\exc})$, seen as a subset of $[0,1]$ and that $|\cdot|$ denotes the Lebesgue measure.

\begin{lem}\label{lem:symetrie_looptree}
For every $u \in [0,1]$, the processes $(|B(u, r) \cap [0,u]|)_{r \ge 0}$ and $(|B(1-u, r) \cap [1-u,1]|)_{r \ge 0}$ have the same distribution.
\end{lem}

\begin{proof}
Fix $u \in [0,1]$. 
We infer from Remark~\ref{rem:geodesic_space} that there exist discrete {\L}ukasiewicz paths, obtained by cyclicly shifting a bridge with exchangeable increments, which converge in distribution to $X^{\exc}$ once properly rescaled, such that if $d^n$ denotes the corresponding looptree distance, normalised both in time and space, then $(d^n(u, t))_{t \in [0,1]}$ converges to $(d_{\Loop(X^{\exc})}(u,t))_{t \in [0,1]}$ for the uniform topology. In particular, the processes $(|B(u, r) \cap [0,u]|)_{r \ge 0}$ and $(|B(1-u, r) \cap [1-u,1]|)_{r \ge 0}$ are the limits in law of the analogous processes for these rescaled discrete looptrees. We claim that these discrete volume processes for any given $n$ have the same distribution as in fact the discrete looptrees are almost invariant under mirror symmetry. 

\begin{figure}[!ht] \centering
\includegraphics[page=3, height=4cm]{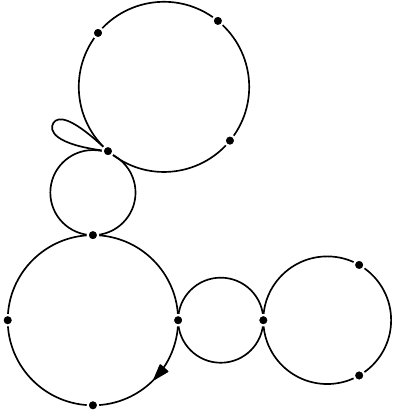}
\quad
\includegraphics[page=1, height=4cm]{dessins}
\quad
\includegraphics[page=13, height=4cm]{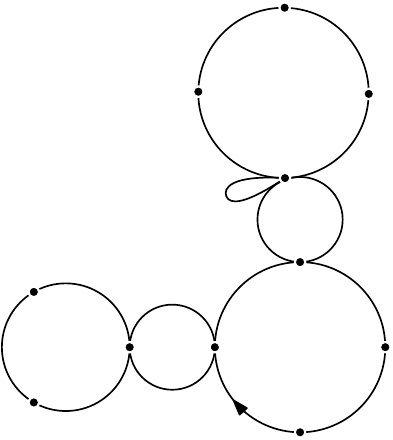}
\caption{Middle: a looptree obtained from a dotted plane tree by replacing the edges from each vertex to its parent by cycles.
Left: a looptree as we consider in this work, obtained from that in the middle by merging each internal vertex of the tree with its right-most offspring.
Right: a looptree obtained by first taking the mirror image of that in the middle and then again merging each internal vertex of the tree with its right-most offspring.
The looptrees on the left and on the right are the mirror image of each other, except for the loop of length $1$.
}
\label{fig:looptree_symetrie}
\end{figure}

Recall that the looptrees we consider here are those on the left of Figure~\ref{fig:looptree_symetrie}, where each internal node in the associated tree (in the middle of the figure) is merged with its right-most offspring. Let us call such a looptree $L_n$. By exchangeability of the coding bridge, the law of the looptree, say $\Lambda_n$, in the middle of Figure~\ref{fig:looptree_symetrie}, without this merging operation, is invariant under mirror symmetry. Hence, if one starts from this looptree $\Lambda_n$, takes its mirror image $\widehat{\Lambda}_n$, and constructs a further looptree $L_n'$ by merging each internal vertex with its right-most offspring, then $L_n'$ has the same law as $L_n$.
Beware that the looptree $L_n'$ is not exactly the mirror image $\widehat{L}_n$ of $L_n$. The only difference however is caused by loops of length $1$, which always pend to the left in $L_n$ and $L_n'$, but get moved instead to the right in $\widehat{L}_n$.
Recall that the random looptrees we are considering are obtained from a {\L}ukasiewicz path which is a discrete approximation of $X^{\exc}$; we can always assume that this path has no null increment, and thus the looptree $L_n$ has no such loop of length $1$, in which case $\widehat{L}_n = L_n'$ has indeed the same law as $L_n$.
We conclude by letting $n\to\infty$ as in the first part of the proof.
\end{proof}

\subsection{On distances in the map}
\label{ssec:distances_carte}

Let us now turn to the proof of Proposition~\ref{prop:borne_sup_volume_boules_cartes_psi} on the volume of balls in the map $(\Map, d_\Map, p_\Map)$.
Let us start by noticing that, similarly to the discrete setting, the distance in the map can be seen as a pseudo-distance on the looptree. This will be useful later.
Recall from Theorem~\ref{thm:convergence_looptrees_labels_Levy} and Theorem~\ref{thm:convergence_cartes_Boltzmann_Levy} that $\Map$ is defined as a subsequential limit of rescaled discrete random maps $M^n$, where $M^n$ is coded by a labelled looptree, itself coded by a pair of discrete paths $(W^n, Z^n)$. We shall denote by $X^n = r_{n}^{-1} W^{n}_{\floor n\cdot}$ the rescaled {\L}ukasiewicz path which converges to $X^{\exc}$ almost surely by combining Theorem~\ref{thm:convergence_looptrees_labels_Levy} with Skorokhod's representation theorem.

\begin{lem}\label{lem:distance_map_on_looptree}
Almost surely, for every $s,t \in [0,1]$, if $d_{\Loop(X^{\exc})}(s,t) = 0$, then $d_\Map(s,t) = 0$.
\end{lem}

\begin{proof}
Fix $s \leq t$ such that $d_{\Loop(X^{\exc})}(s,t) = 0$, which means by Lemma~\ref{lem:identifications_looptree_Levy} that
\begin{equation}\label{eq:identifications_looptree_Levy}
\begin{aligned}
&\text{(i)}\enskip \text{either}\enskip \Delta X^{\exc}_s = 0 \enskip\text{and}\enskip X^{\exc}_s = X^{\exc}_t = \inf_{[s,t]} X^{\exc}
\\
&\text{(ii)}\enskip \text{or}\enskip \Delta X^{\exc}_s > 0 \enskip\text{and}\enskip t = \inf\{r > s \colon X^{\exc}_r = X^{\exc}_{s-}\}
\end{aligned}
\end{equation}
Suppose first that we are in the first case and in addition that $X^{\exc}_r > X^{\exc}_s=X^{\exc}_t$ for all $r \in (s,t)$. 
Fix some small $\varepsilon>0$, then on the one hand we have $\inf_{[s+\varepsilon, t-\varepsilon]} X^{\exc} > X^{\exc}_t + 2\delta$ for some $\delta > 0$, and on the other hand by right-continuity there exists $s' \in (s, s+\varepsilon)$ such that $X^{\exc}_s < X^{\exc}_{s'} \leq X^{\exc}_s+\delta$; finally since $X^{\exc}$ has no negative jumps, then all values in $(X^{\exc}_t, X^{\exc}_t + 2\delta)$ are achieved by some times in $(t-\varepsilon, t)$.
Recall that $X^{\exc}$ is the almost sure limit in the Skorokhod $J_{1}$ topology of some rescaled discrete {\L}ukasiewicz path $X^n$. 
Then for every $n$ large enough, there exist $s_n \in (s, s+\varepsilon)$ and $t_n \in (t-\varepsilon, t)$ with $X^n_{s_n} = X^n_{t_n} = \inf_{[s_n,t_n]} X^n$, namely the two times correspond to the same vertex in the discrete looptree. This implies that they also correspond to the same vertex in the discrete map, i.e.~that $d_{M^n}(s_n,t_n)=0$. When passing to the limit along the corresponding subsequence and letting then $\varepsilon \downarrow 0$, we infer that indeed $d_\Map(s,t)=0$.

Suppose next that we are in the first case in~\eqref{eq:identifications_looptree_Levy}, but now that there exists $r \in (s,t)$ with $X^{\exc}_r = X^{\exc}_s = X^{\exc}_t$. Then by uniqueness of the local minima, neither $s$ nor $t$ are local minima. 
Since in addition $\Delta X^{\exc}_s=0$,  this then implies that for every $\varepsilon>0$, there exist $s' \in (s-\varepsilon,s)$ and $t' \in (t, t+\varepsilon)$ which satisfy $X^{\exc}_{s'} = X^{\exc}_{t'} < X^{\exc}_r$ for all $r \in (s',t')$ and we conclude from the preceding case.

Similarly, if we fall in the second case in~\eqref{eq:identifications_looptree_Levy}, then the previous argument shows that $t$ is not a time of local minimum for otherwise the reversed process started just after would reach this record again by a jump. Using the left-limit at $s$, we infer that for every $\varepsilon>0$, there exist times $s' \in (s-\varepsilon, s)$ and $t' \in (t, t+\varepsilon)$ with $X^{\exc}_{s'} = X^{\exc}_{t'} < X^{\exc}_r$ for every $r \in (s',t')$ and we conclude again from the first case.
\end{proof}

Let us next continue with a geometric lemma inspired by~\cite[Lemma~14]{LGM11} which will allow us to include balls in the map in some well-chosen intervals analogously to Lemma~\ref{lem:geom_looptree} for looptrees.
Recall the sets $A_u$ and $B_u$ defined by~\eqref{eq:def_ancetres} and~\eqref{eq:epine_gauche_droite} respectively.

\begin{lem}\label{lem:geometric_boules_cartes}
Almost surely for every $u \in (0,1)$ and for every pair $(a,b) \in A_u \cup B_u$, we have
\[d_\Map(u,t) \geq \min\{Z_{u}-Z_{a}, Z_{u}-Z_{b}\},\]
for every $t \in [0,a) \cup (b, 1]$.
\end{lem}

We shall assume throughout the proof that both $Z_{a} < Z_u$ and $Z_{b} < Z_{u}$ otherwise the claim is trivial.

\begin{proof}
Let us first consider the analogue in the discrete setting of a pointed map $(M^n, v^n_\star)$ coded by a pair $(W^n, Z^n)$ and the associate labelled looptree.
Here times are integers and we use the same notation for a vertex in the map different from the distinguished one, the corresponding vertex in the looptree, and a time of visit of one of its corners when following the contour of the looptree.
The assumption that $(a,b) \in A_u \cup B_u$ means that either $a$ and $b$ are the same vertex, which is an ancestor of $u$, or they both belong to the same cycle between two consecutive ancestors of $u$, with $a$ on the left and $b$ on the right.

Then take a vertex $v$ outside the interval $[a,b]$ and consider a geodesic \emph{in the map} from $v$ to $u$. This geodesic necessarily contains an edge of the map with one extremity $x \in [a,b]$ and the other one $y \notin [a,b]$. 
From the construction of the bijection briefly recalled in Section~\ref{ssec:def_discret}, this means that either $x$ is the successor of $y$ or vice versa.
Now if $x$ is the successor of $y$, this means that $x$ is the first vertex after $y$ when following the contour of the looptree (extended by periodicity) which has a label smaller than that of $y$.
Since one necessarily goes through $a$ when going from $y$ to $x$ by following the contour of the looptree, then 
$Z^n_a \ge Z^n_y \ge Z^n_x+1$,
see e.g.~Figure~\ref{fig:boules_cartes}.
Symmetrically, if $y$ is the successor of $x$, then $Z^n_x \le Z^n_b$.
Hence $Z^n_x \le \max\{Z^n_a, Z^n_b\}$ in any case. On the one hand this vertex $x$ lies on the geodesic in the map between $u$ and $v$, so $d_{M^n}(u,x) < d_{M^n}(u,v)$. On the other hand the labels encode the graph distance to the reference point $v^n_\star$ in the map, so by the triangle inequality, we have $d_{M^n}(u,x) \ge Z^n_u-Z^n_x$, which leads to our claim in the discrete setting.

\begin{figure}[!ht] \centering
\includegraphics[page=14, height=5.5cm]{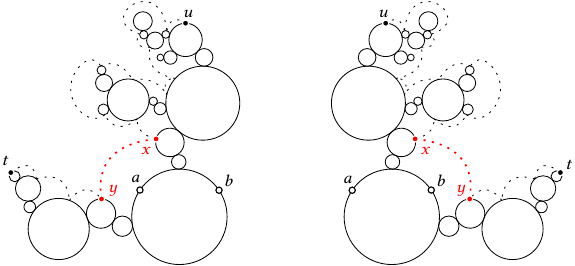}
\caption{A portion of a discrete labelled looptree in plain lines. In dashed lines is represented a geodesic path between $t$ and $u$ in the associated map. Each dashed edge thus links a time to its successor. The pictures display the two situations: when $x$ is the successor of $y$ on the left and vice versa on the right. Since the edge between $x$ and $y$ steps over $a$ or $b$, then the label of $x$ is necessarily smaller than or equal to that of $a$ or $b$ in each case.}
\label{fig:boules_cartes}
\end{figure}

In the continuum setting, we argue by approximation. Recall that we assume $(X^{\exc}, Z^{\exc})$ to arise as the almost sure limit of a pair $(X^n, Z^n)$, which is already rescaled in time and space to lighten the notation.
Let us first consider the case when $(a,b) \in B_u$, namely that there exists $s<a$ such that
\[X^{\exc}_{s} > X^{\exc}_a = \min_{[s,a]} X^{\exc} > X^{\exc}_b = \min_{[s,b]} X^{\exc} > X^{\exc}_{s-}.\]
We infer from the convergence $X^n \to X^{\exc}$ in the Skorokhod $J_{1}$ topology that there exist times $s_n < a_n < u_n < b_n$ in $\N/n$ which converge to $s$, $a$, $u$, $b$ respectively and such that 
$X^n_{a_n} = \min_{[s_n,a_n]} X^n > X^n_{b_n} = \min_{[s_n,b_n]} X^n > X^n_{s_n-1/n}$.
This implies that $a_n$ and $b_n$ both belong to the same cycle between two consecutive ancestors of $u_n$, with $a_n$ on the left and $b_n$ on the right and we infer from the first part of the proof that the graph distance in the rescaled map between $u_n$ and any vertex $t_n$ which lies outside $[a_n, b_n]$ is larger than or equal to $\min\{Z^n_{u_n}-Z^n_{a_n}, Z^n_{u_n}-Z^n_{b_n}\}$. We finally can pass to the limit along the corresponding subsequence, using that $Z^n \to Z$ for the uniform topology.

Suppose now that $(a,b) \in A_u$, which means that either
$X^{\exc}_a = X^{\exc}_b = \min_{[a,b]} X^{\exc}$
or
$b = \inf\{t > a \colon X^{\exc}_t = X^{\exc}_{a-}\}$.
As shown in the proof of Lemma~\ref{lem:distance_map_on_looptree}, in each case we can approximate this pair by pairs that satisfy
\[X^{\exc}_t > X^{\exc}_b=X^{\exc}_a \qquad\text{for every } t \in (a,b).\]
We thus henceforth assume that the latter holds.
In particular, for any $\varepsilon > 0$ small, the values of $X^{\exc}_t$ for $t \in[a+\varepsilon, b-\varepsilon]$ stay bounded away from $X^{\exc}_b=X^{\exc}_a$.
From the fact that $X^{\exc}$ is right-continuous at $s$ and makes no negative jump, we infer that every value larger than but close enough to $X^{\exc}_b=X^{\exc}_a$ is achieved both by times arbitrarily close to $a$ and by times arbitrarily close to $b$.
Then the convergence $X^n \to X^{\exc}$ in the Skorokhod $J_{1}$ topology yields the existence of times $a_n < u_n < b_n$ which converge to $a$, $u$, $b$ respectively and such that for $n$ large enough, we have $X^n_t \ge X^n_{a_n} = X^n_{b_n}$ for every $t \in (a_n,b_n) \cap \N/n$.
It follows that $a_n$ and $b_n$ correspond to the same vertex in the looptree, and it is an ancestor of $u_n$.
We conclude from the discrete case again.
\end{proof}

Recall that Proposition~\ref{prop:borne_sup_volume_boules_cartes_psi} claims that almost surely, the volume of the ball in the map centred at an independent uniform random point $U$ and with radius $2^{-n}$ is less than $\psi(2^{2n(1-\delta)})^{-1}$ for $n$ large enough.
This now easily follows from Lemma~\ref{lem:geometric_boules_cartes} and a technical result that we state and prove in the next subsection.

\begin{proof}[Proof of Proposition~\ref{prop:borne_sup_volume_boules_cartes_psi}]
According to Lemma~\ref{lem:geometric_boules_cartes}, almost surely the ball of radius $r$ centred at the image $U$ in the map is entirely contained in the image of any interval $[a_r, b_r]$ such that the pair $(a_r, b_r) \in A_U \cup B_U$ satisfies $\min\{Z_{U}-Z_{a_r}, Z_{U}-Z_{b_r}\} \geq r$.
Consider the infinite volume model, obtained by replacing the excursion $X^{\exc}$ by the two-sided unconditioned L\'evy process and the time $U$ by $0$. 
Proposition~\ref{prop:technique} stated below shows that almost surely, for every $r>0$ small enough, there exists such a pair $(a_r, b_r) \in A_0 \cup B_0$ such that $\min\{Z_{0}-Z_{a_r}, Z_{0}-Z_{b_r}\} \geq r$ and both $a_r$ and $b_r$ lie at looptree distance at most $r^{2(1-\delta)}$ from $0$.
Take $r=2^{-n}$. We infer from the proof of Proposition~\ref{prop:borne_sup_volume_boules_looptrees_psi}
that almost surely, for every $n$ large enough we have both 
$a_{2^{-n}} \ge U^-_{2^{-2n(1-\delta)}} \ge U-\psi(2^{2n(1-\delta)^2})^{-1}$
and
$b_{2^{-n}} \le U^+_{2^{-2n(1-\delta)}} \le U+\psi(2^{2n(1-\delta)^2})^{-1}$.
The claim follows by local absolute continuity.
\end{proof}

\subsection{A technical lemma with Brownian motions}
\label{sec:technique}

In order to complete the proof of Proposition~\ref{prop:borne_sup_volume_boules_cartes_psi}, it remains to establish the following.

\begin{prop}\label{prop:technique}
For every $\varepsilon > 0$, almost surely, for every $r>0$ small enough, in the infinite volume labelled looptree coded by the pair of paths $X = (\overleftarrow{X}_s, \overrightarrow{X}_t)_{s,t \geq 0}$, there exists a pair $(a_r, b_r) \in A_0 \cup B_0$ which satisfies both 
$-Z_{a_r} \geq r$ and $-Z_{b_r} \geq r$
as well as
$d_{\Loop(X)}(a_r, 0) \leq r^{2(1-\varepsilon)}$ and $d_{\Loop(X)}(b_r, 0) \leq r^{2(1-\varepsilon)}$.
\end{prop}

In the case of stable L\'evy processes, Le~Gall \& Miermont~\cite[Lemma~15]{LGM11} were able to find such a pair $(a_r, b_r) \in A_0$, namely one that codes an ancestor of the distinguished point $0$. Their argument can be extended when the upper and lower exponents $\bdgamma$ and $\bdeta$ coincide, or when the exponent $\bddelta \in [1, \bdgamma]$ introduced by Duquesne~\cite{Duq12} as $\bddelta \coloneqq \sup\{c > 0 \colon \exists K \in (0, \infty) \text{ such that } K\, \mu^{-c} \psi(\mu) \leq \lambda^{-c} \psi(\lambda), \text{ for all } 1 \leq \mu \leq \lambda\}$ is not equal to $1$.
However we were not able to extend it when $\bddelta=1$.
We therefore argue in a general setting, conditionally on the underlying looptree, and relying only on the properties of the Brownian labels.

Let us first introduce the setting in terms of Brownian motion and Brownian bridges. Let $\mathcal{C}$ be a random closed subset of $[0,\infty)$ containing $0$, its complement is a countable union of disjoint open intervals, say, $\mathcal{C}^c = \bigcup_{i \geq 1} (g_i, d_i)$ with $d_i > g_i$.
For every $i \geq 1$, let us denote by $\delta_i = d_i-g_i$ the length of these intervals and suppose we are given random lengths $\delta^L_i, \delta^R_i \geq \delta_i$.
Independently, let $W = (W_t , t \geq 0)$ be a standard Brownian motion started from $0$. Consider its trace on $\mathcal{C}$, then independently for every $i \geq 1$, conditionally on the values $W_{g_i}$ and $W_{d_i}$ and on the lengths $\delta^L_i$ and $\delta^R_i$, let us sample two independent Brownian bridges $W^{L,i} = (W^{L,i}_t ; t \in [0, \delta^L_i])$ and $W^{R,i} = (W^{R,i}_t ; t \in [0, \delta^R_i])$ with duration $\delta^L_i$ and $\delta^R_i$ respectively, both starting from $W^{L,i}_0 = W^{R,i}_0 = W_{g_i}$ and ending at $W^{L,i}_{\delta^L_i} = W^{R,i}_{\delta^R_i} = W_{d_i}$ and independent of the rest.

The connection with our problem is the following: the successive ancestors of $0$ in the infinite looptree, correspond to the times of weak records of $\overleftarrow{X}$, and their labels form a subordinate Brownian motion, precisely $W_{\sigma(\cdot)}$, where $W$ is a standard Brownian motion and $\sigma$ is the last subordinator in Lemma~\ref{lem:loi_epine_looptree}. We shall thus take $\mathcal{C}$ to be the range of $\sigma$, and the labels on each side of the successive cycles between these ancestors roughly correspond to the collection $\{W^{L,i},W^{R,i}\}$.

\begin{lem}\label{lem:technique}
For every $\varepsilon > 0$, almost surely, for every $r>0$ small enough, the following alternative holds: 
\begin{itemize}
\item either there exists $t \in \mathcal{C} \cap [0, r^{2(1-\varepsilon)}]$ such that $W_t \leq -r$,
\item or there exist $i \geq 1$ such that $g_i < r^{2(1-\varepsilon)}$, and $s \in [0, \min\{\delta^L_i, r^{2(1-\varepsilon)}\}]$, and also finally $t \in [0, \min\{\delta^R_i, r^{2(1-\varepsilon)}\}]$ such that both $W^{L,i}_s \leq -r$ and $W^{R,i}_t \leq -r$.
\end{itemize}
\end{lem}

The key point in this statement is that in the second case, both bridges reach $-r$ in the same interval of $\mathcal{C}^c \cap [0, r^{2(1-\varepsilon)}]$.
The idea of the proof is quite intuitive: first we know by the law of iterated logarithm that $W$ reaches values much less than $-r$ on $[0, r^{2(1-\varepsilon)}]$ so, if the first alternative fails, then there exists an interval $(g_i, d_i)$ on which $W$ goes much below $-r$, while starting and ending above this value. Then this interval cannot be too small for the Brownian motion to fluctuate enough. Therefore, if one resamples two Brownian bridges with the same duration $\delta_i$ and the same endpoints $W_{g_i}$ and $W_{d_i}$, then they have a reasonable chance to both go below $-r$ in this interval. Actually the Brownian bridges we sample have a longer duration, so intuitively, they fluctuate even more and have an even greater chance to reach $-r$.
However turning this discussion into a rigorous proof is quite challenging, and the proof of Lemma~\ref{lem:technique} is deferred to Section~\ref{sec:appendice_volume}.

Let us finish this section by relying on this result in order to establish Proposition~\ref{prop:technique}, thereby finishing the proof of Proposition~\ref{prop:borne_sup_volume_boules_cartes_psi}.

\begin{proof}[Proof of Proposition~\ref{prop:technique}]
Recall the Poisson random measure $\mathscr{N}$ in~\eqref{eq:mesure_poisson_epine} and the four subordinators defined from it.
By construction of the looptree, the successive ancestors of $0$ in the infinite looptree, started from $0$, are coded by the negative of the first element of pairs in $A_0$ (defined by~\eqref{eq:def_ancetres}) and correspond to the times of weak records of $\overleftarrow{X}$.
Further, by construction of the label process, the labels of these ancestors form a subordinate Brownian motion, precisely $W_{\sigma(\cdot)}$, where $W$ is a standard Brownian motion and $\sigma$ is the last subordinator in Lemma~\ref{lem:loi_epine_looptree}. 
We shall apply Lemma~\ref{lem:technique} to the closed set $\mathcal{C}$ given by the range of $\sigma$.

Every jump of this subordinator, and more precisely every atom of the Poisson random measure $\mathscr{N}$, codes two consecutive ancestors separated by a cycle in the looptree, whose left and right length is encoded in the atom of $\mathscr{N}$. In Lemma~\ref{lem:technique}, these two lengths correspond to $\delta^L_i$ and $\delta^R_i$, whereas the length $\delta_i = d_i-g_i$ equals $\delta_i = \delta^L_i \delta^R_i / (\delta^L_i+\delta^R_i) \leq \min(\delta^L_i, \delta^R_i)$ and corresponds indeed to the variance of a Brownian bridge of duration $\delta^L_i+\delta^R_i$ and evaluated at $\delta^L_i$ or $\delta^R_i$.
By standard properties of the Brownian bridge, conditionally on these lengths and on the labels of the two ancestors $W_{g_i}$ and $W_{d_i}$, the labels on each side of the cycle indeed evolve like two independent Brownian bridges, given precisely by $W^{L,i}$ and $W^{R,i}$.

Lemma~\ref{lem:technique} thus shows the following alternative: either there exists a value $t_r \leq r^{2(1-\varepsilon)}$ in the range of $\sigma$, say $t_r = \sigma(s_r)$, and such that $W_{\sigma(s_r)} \leq -r$.
In this case, the time $s_r$, once time-changed by the local time at the supremum of $\overleftarrow{X}$, corresponds to the first element of a pair $(-a_r, b_r) \in A_0$.
Otherwise there exists a jump of $\sigma$, say $g_r = \sigma(s_r-) < \sigma(s_r) = d_r$, with $g_r < r^{2(1-\varepsilon)}$, which corresponds to a cycle in the looptree, and then there exists one point on each side of this cycle, both at distance (on the cycle) at most $r^{2(1-\varepsilon)}$ from $g_r$ and with label less than $-r$.

Notice the easy bound $\min\{u, 1-u\} \leq 2 u(1-u)$ for any $0 \leq u \leq 1$ which shows that the subordinators $\sigma$ and $\widetilde{X}$ from Lemma~\ref{lem:loi_epine_looptree} are related by
$\widetilde{X} \leq 2 \sigma$.
Observe that $\widetilde{X}$ encodes the looptree distance to $0$ of the successive ancestors.
Therefore, in the first case of the alternative, the time $-a_r$ lies at looptree distance at most $2r^{2(1-\varepsilon)}$ from $0$. So does the time corresponding to $g_r$ in the second case, and thus the two points on each side of the cycle lie at looptree distance at most $3 r^{2(1-\varepsilon)}$ from $0$. One can then get rid of this factor $3$ by replacing $\varepsilon$ by a smaller value.
\end{proof}

\subsection{Proof of the technical lemma}
\label{sec:appendice_volume}

Let us end this section with the proof of Lemma~\ref{lem:technique}. 
The statement can be rephrased as follows: if the Brownian motion $W$ is such that $W_t > -r$ for all $t \in \mathcal{C} \cap [0, r^{2(1-\varepsilon)}]$, then necessarily the two Brownian bridges we resample both reach $-r$ in the same interval in $\mathcal{C}^c \cap [0, r^{2(1-\varepsilon)}]$.

We start with an easy intermediate result. We shall denote by $\tau_z = \inf\{t \geq 0 \colon W_t = z\}$ the first hitting time of $z \in \R$ by the Brownian motion $W$.
For every $r \in (0,1)$ small enough, we let 
$K_r$ be the largest integer such that
\begin{equation}\label{eq:K_r_lemme_technique}
K_r \leq \frac{\log(1/r)}{(\log\log(1/r))^3}
.\end{equation}

\begin{lem}\label{lem:ruine}
Almost surely, for every $r \in \{2^{-n}, n \in \N\}$ small enough, the following assertions hold:
\begin{enumerate}
\item The path $W$ reaches $-r \log^{10 K_{r}}(1/r)$ before time $r^{2(1-1/\log\log(1/r))}$, 
\item For every $k \leq K_r$, the path $W$ reaches $-r \log^{10 k}(1/r)$ before $r \log^{10 k + 3}(1/r)$,
\item 
There exists a universal constant $c \in (0,1)$ such that for at least $c K_r$ different indices $k \leq K_r$ the following holds: 
at the first time the path $W$ reaches $-r \log^{10 k}(1/r)$, it was below $-r$ for a duration at least $(r \log^{10 k}(1/r))^2$ and it previously never reached $r \log^{10 k}(1/r)$.
\end{enumerate}
\end{lem}

Let us note that by comparing any $r \in (0,1)$ to a negative power of $2$, the claims extend by monotonicity to any $r>0$ small enough, provided a minor change in the exponents.

\begin{proof}
Let us prove that for each event, the probability to fail is bounded above by a quantity which depends on $r$ and such that the sum over all values of $r \in \{2^{-n}, n \geq 1\}$ is finite. The claims then follow from the Borel--Cantelli lemma.
For the first assertion, the reflection principle and the scaling property show that:
\begin{align*}
&\P(\tau_{-r \log^{10 K_{r}}(1/r)} > r^{2(1-1/\log\log(1/r))})\\
& \qquad = \P(\inf\{W_t, t \in [0, r^{2(1-1/\log\log(1/r))}]\} \geq -r \log^{10 K_{r}}(1/r))
\\
& \qquad= \P(|W_{r^{2(1-1/\log\log(1/r))}}| \leq r \log^{10 K_{r}}(1/r)) 
\\
&\qquad= \P(|W_{1}| \leq r^{1/\log\log(1/r)} \log^{10 K_{r}}(1/r)) 
.\end{align*}
Since $K_r = o(\log(1/r)/(\log\log(1/r))^2)$, then for $r$ small enough, we have
\[r^{1/\log\log(1/r)} \log^{10 K_{r}}(1/r)
\leq \exp\Bigl(- \frac{\log(1/r)}{2 \log\log(1/r)}\Bigr)
,\]
whose sum of over all values of $r \in \{2^{-n}, n \geq 1\}$ is finite.
By a Taylor expansion of the last probability, written as the integral of the density, we obtain:
\[\P(\tau_{-r \log^{10 K_{r}}(1/r)} > r^{2(1-1/\log\log(1/r))})
= O\Bigl(\exp\Bigl(- \frac{\log(1/r)}{2 \log\log(1/r)}\Bigr)\Bigr)
,\]
hence the sum of the left-hand side over all values of $r \in \{2^{-n}, n \geq 1\}$ is finite as we wanted.

Next, by the classical ruin problem, we have:
\[\P(\tau_{r \log^{10k+3}(1/r)} < \tau_{-r \log^{10k}(1/r)})
= \frac{r \log^{10k}(1/r)}{r \log^{10k}(1/r) + r \log^{10k+3}(1/r)}
= \frac{1}{1 + \log^3(1/r)} 
.\]
Since $K_r = o(\log(1/r))$, then we infer by a union bound that 
\[\P(\exists k \leq K_r \colon \tau_{r \log^{10k+3}(1/r)} < \tau_{-r \log^{10k}(1/r)}) = o(\log^{-2}(1/r)),\]
and again the right-hand side forms a convergent series when summing over all values $r \in \{2^{-n}, n \in \N\}$.

For the last claim, to simplify notation set $r_{k}=r \log^{10k}(1/r)$. By symmetry, for any given $k$, the probability that $W$ reaches $-r_{k}$ before $r_{k}$ equals $1/2$ so if these events were independent, then this would follow e.g.~from Hoeffding's inequality. They are not independent, however by the previous bound, we may assume that at the first time $W$ reaches $-r_{k}$, it has not reached $r_{k+1}$ yet, and by the strong Markov property, the probability starting from this position to reach $-r_{k+1}$ before $r_{k+1}$ is (slightly) larger than $1/2$. Formally, one can e.g.~apply the Azuma inequality to the martingale whose increments are
\[
\Delta M_k = \ind{\tau_{-r_{k+1}} < \tau_{r_{k+1}}} - \P(\tau_{-r_{k+1}} < \tau_{r_{k+1}} \mid (W_t, t \leq \tau_{-r_{k}})),
\]
and since the conditional probability on the right is larger than $1/2$, then
\begin{align*}
\P\Bigl(\#\{k \leq K_r \colon \tau_{-r_{k+1}} < \tau_{r_{k+1}}\} < \frac{K_r}{4}\Bigr) & = \P\Bigl(\sum_{k=1}^{K_r} \Bigl(\ind{\tau_{-r_{k+1}} < \tau_{r_{k+1}}} - \frac{1}{2}\Bigr) < - \frac{K_r}{4}\Bigr)
\\
&  \leq \P\Bigl(\sum_{k=1}^{K_r} \Delta M_k < - \frac{K_r}{4}\Bigr) \leq \exp(- c K_r)
,\end{align*}
where $c>0$ is some universal constant.
The sum of these exponentials over all values $r \in \{2^{-n}, n \in \N\}$ is finite again.

Similarly, for any $z>0$, the law of the time-reversed path $(z+W_{(\tau_{-z}-t)^+})_{t \geq 0}$ is that of a three-dimensional Bessel process started from $0$, or equivalently the Euclidean norm of a three-dimensional Brownian motion, stopped at the last passage at $z$. The probability that a Bessel process does not reach $z/2$ before time $z^2$ is bounded below by some constant $c>0$ independent of $z$.
We conclude as above, that the proportion of the number of indices $k \leq K_r$ such that this occurs for $z = r_{k}$ is bounded away from $0$.
\end{proof}

Let us finally prove Lemma~\ref{lem:technique}.

\begin{proof}[Proof of Lemma~\ref{lem:technique}]
Let us work on the event that $W_t > -r$ for all $t \in \mathcal{C} \cap [0, r^{2(1-\varepsilon)}]$ and let us prove that necessarily the second case of the alternative occurs.
By monotonicity, it is sufficient to prove the claim when $r$ is restricted to the set $\{2^{-n}, n \in \N\}$. 
Let us implicitly assume also that the three assertions from Lemma~\ref{lem:ruine} hold, and let us prove that there exists $k \leq K_r$, the integer from~\eqref{eq:K_r_lemme_technique}, such that on the interval of $\mathcal{C}^c$ on which $W$ reaches $-r \log^{10k}(1/r)$ for the first time, the two bridges that we sample also reach $-r$.

Let us introduce some notation. For every $k \geq 1$, denote by $G_k$ and $D_k$ the boundaries $g_i$ and $d_i$ of the unique interval $(g_i, d_i) \subset \mathcal{C}^c$ such that 
$W$ reaches the level $-r \log^{10 k}(1/r)$ for the first time in the interval $(g_i, d_i)$. 
Note that these intervals are not necessarily distinct from each others and that when $r$ is small enough, by Lemma~\ref{lem:ruine}, they all intersect $(0, r^{2(1-\varepsilon)})$ for $k \leq K_{r}$, and only $D_{K_r}$ may exceed $r^{2(1-\varepsilon)}$. 
Let $\Delta_k = D_k-G_k$ and let $\Delta^L_k, \Delta^R_k \geq \Delta_k$ denote the corresponding extensions; let us change the notation and write $W^{L,k}$ and $W^{R,k}$ for the Brownian bridges with duration $\Delta^L_k$ and $\Delta^R_k$ respectively, with initial and terminal value $W_{G_k}$ and $W_{D_k}$ respectively.
We shall upper bound the probability of the event
\[A_r = \bigcap_{k \leq K_r} \Bigl(\Bigl\{\inf_{s \in [0,\Delta^L_k] \cap (0, r^{2(1-\varepsilon)})} W^{L,k}_s > -r\Bigr\} \cup \Bigl\{\inf_{t \in [0,\Delta^R_k] \cap (0, r^{2(1-\varepsilon)})} W^{R,k}_t > -r\Bigr\}\Bigr)\]
by a quantity that only depends on $r$ and whose sum over all values of $r \in \{2^{-n}, n \geq 1\}$ is finite, such as $\log^{-2}(1/r)$. 
We then conclude as in the previous proof from the Borel--Cantelli lemma. 
\medskip

\textsc{Step 1: Proof for Brownian motions.}
Let us first prove the claim when the Brownian bridges $W^L$ and $W^R$ are replaced by Brownian motions. Precisely, for every $k \leq K_r$, conditionally on $\Delta_k$, $\Delta_k^L$, $\Delta_k^R$, and $W_{G_k}$, let $X^{L,k} = (X^{L,k}_t ; t \in [0, \Delta_k^L])$ and $X^{R,k} = (X^{R,k}_t ; t \in [0, \Delta_k^L])$ be two independent Brownian motions started from $X^{L,k}_0=X^{R,k}_0=W_{G_k} > -r$.

We first claim that if the interval $[G_k,D_k]$ is too long, then the Brownian motions $X^{L,k}$ and $X^{R,k}$ on it cannot reasonably stay above $-r$. Specifically, if we denote by $X$ a generic Brownian motion independent of the rest, then by the same reasoning as in the previous proof, we have for every $k$:
\begin{align*}
&\P\Bigl(\Delta_k > 2(r+W_{G_k})^2 \log^6(1/r) \enskip\text{and}\enskip \max\Bigl\{\inf_{t \in [0,\Delta_k/2]} X^{L,k}_t, \inf_{t \in [0,\Delta_k/2]} X^{R,k}_t\Bigr\} > -r \Bigr)
\\
&\leq 2 \P\bigl(\inf\{X_s , 0 \leq s \leq (r+W_{G_k})^2 \log^6(1/r)\} > -(W_{G_k}+r) \bigr)
\\
&= 2 \P\bigl(|X_{(r+W_{G_k})^2 \log^6(1/r)}| < W_{G_k}+r \bigr)
\\
&= 2 \P(|X_1| < \log^{-3}(1/r) )
\\
&= O(\log^{-3}(1/r))
.\end{align*}
Since $K_r = o(\log(1/r))$, then a union bound combined with the Borel--Cantelli lemma imply that almost surely, all the events for $k \leq K_r$ fail for $r$ small enough (of the form $r=2^{-n}$).
Thus, in order to avoid that both $X^{L,k}$ and $X^{R,k}$ reach $-r$, we must have $\Delta_k \leq 2(r+W_{G_k})^2 \log^6(1/r)$ for every $k \leq K_r$. 
Since $G_k$ is smaller than the first hitting time of $-r \log^{10k}(1/r)$ by definition, then combined with Lemma~\ref{lem:ruine}, we have:
\[-r < W_{G_k} \leq r \log^{10k+3}(1/r)
\quad\text{and so}\quad
\Delta_k \leq 3(r \log^{10k+6}(1/r))^2 \quad\text{for each}\enskip k \leq K_r.\]

According to the strong Markov property applied at the first hitting time of $-r \log^{10k}(1/r)$, the path $W$ behaves after this time as an unconditioned Brownian motion started from this value during the remaining time of the interval $[G_k,D_k]$.
On the previous event, this remaining time is at most $3(r \log^{10k+6}(1/r))^2$.
This is not enough for $W$ to further reach $-r \log^{10(k+1)}(1/r)$. Indeed, if $X$ is a generic Brownian motion, then:
\begin{align*}
\P\Bigl(\inf_{[0,4(r \log^{10k+6}(1/r))^2]} X < -r \log^{10k+8}(1/r)\Bigr)
&= \P\Bigl(|X_1| > \frac{1}{2} \log^{2}(1/r)\Bigr)
\\
&\leq 2 \e^{- \frac{1}{4} \log^{2}(1/r)}
,\end{align*}
and the right-hand side, multiplied by $K_r$, is again a convergent series when summing over all values $r \in \{2^{-n}, n \in \N\}$.
Hence, we may also assume that the intervals for different values of $k \leq K_r$ are disjoint, and so the Brownian motions $X^{L,k}$ and $X^{R,k}$ that we sample are independent.

Finally, by the last claim of Lemma~\ref{lem:ruine}, for a positive proportion of the indices $k \leq K_r$ we have both $W_{G_k} \leq r \log^{10k}(1/r)$ and $\Delta_k > (r \log^{10k}(1/r))^2$, so there is a probability uniformly bounded away from $0$ that both $X^{L,k}$ and $X^{R,k}$ reach $-r$, in the first half of the interval.
We conclude by independence that this occurs for a positive proportion of the indices $k \leq K_r$.
\medskip

\textsc{Step 2: Proof for Brownian bridges.}
Let us next prove the actual statement by comparing the Brownian motions $X^{L,k}$ and $X^{R,k}$ with Brownian bridges $W^{L,k}$ and $W^{R,k}$. We aim at proving that the probability that one of these bridges stays above $-r$ in the first half of the interval when $\Delta_k > (r+W_{G_k})^2 \log^6(1/r)$ is summable when $r$ ranges over $\{2^{-n}, n \in\N\}$, and then that if both $W_{G_k} \leq r \log^{10k}(1/r)$ and $\Delta_k > (r \log^{10k}(1/r))^2$, then there is a probability bounded away from $0$ uniformly in $r$ and $k$ that the two bridges both reach $-r$.
Let us only consider $X^{L,k}$ and $W^{L,k}$ since the same argument applies to $X^{R,k}$ and $W^{R,k}$.
Since, conditionally on $W_{G_k}$, $W_{D_k}$, $\Delta_k$, and $\Delta^L_k$, the path $W^{L,k}$ is a Brownian bridge from $W_{G_k}$ to $W_{D_k}$ with duration $\Delta^L_k \geq \Delta_k$, then the following absolute continuity relation holds: for any measurable and nonnegative function $F$,
\begin{align*}
&\E\bigl[F(W^{L,k}_{t}, t \leq \Delta^L_k/2) \mid W_{G_k}, W_{D_k}, \Delta_k, \Delta^L_k\bigr] \\
& \qquad = \E\Bigl[F(X^{L,k}_t, t \leq \Delta^L_k/2) \, \frac{p_{\Delta^L_k/2}(W_{D_k}-X^L_{\Delta^L_k/2})}{p_{\Delta^L_k}(W_{D_k}-W_{G_k})} \bigm| W_{G_k}, W_{D_k}, \Delta_k, \Delta^L_k\Bigr]
,
\end{align*}
where $p_t(x) = (2\pi t)^{-1/2} \exp(-x^2/(2t))$.
We then integrate with respect to $W_{D_k}$. Precisely, we claim that there exists a constant $C>0$ such that
on the event $\{W_{G_k}>-r, W_{D_k}>-r, \Delta_k > 2(r+W_{G_k})^2 \log^6(1/r)\}$, we have
\begin{equation}\label{eq:borne_ratio_technique}
\E\Bigl[\frac{p_{\Delta^L_k/2}(W_{D_k}-X^L_{\Delta^L_k/2})}{p_{\Delta^L_k}(W_{D_k}-W_{G_k})} \bigm| X^L_{\Delta^L_k/2}, W_{G_k}, \Delta_k, \Delta^L_k\Bigr]
\leq C
.\end{equation}
This allows to infer from the first part of the proof that
\[\P\Bigl(\Delta_k > 2(r+W_{G_k})^2 \log^6(1/r) \enskip\text{and}\enskip \max\Bigl\{\inf_{t \in [0,\Delta_k/2]} W^{L,k}_t, \inf_{t \in [0,\Delta_k/2]} W^{R,k}_t\Bigr\} > -r \Bigr)
\]
is $ O(\log^{-3}(1/r))$. 
Then, as in the first part, we must impose $\Delta_k \leq 2(r+W_{G_k})^2 \log^6(1/r) \leq 3(r \log^{10k+6}(1/r))^2$ for every $k \leq K_r$ in order to prevent the bridges $W^{L,k}$ and $W^{R,k}$ to reach $-r$.

Let us prove the upper bound~\eqref{eq:borne_ratio_technique}. Recall the notation $r_{k}=r \log^{10k}(1/r)$.
Conditionally on $W_{G_k}$ and $\Delta_k$, the path $(W_t; t \in [G_k, D_k])$ has the law of a Brownian motion $Y$ with duration $\Delta_k$, started from $W_{G_k}$, and conditioned on the event $E_k = \{\inf\{Y_t ; t \in [G_k, D_k]\} \leq -r_{k}\}$. Then the law of $W_{D_k}$ can be obtained by splitting at the first hitting time of this value; by the Markov, property, the remaining path after is an unconditioned Brownian motion. Recall also that we work implicitly on the event where $W_{G_k}, W_{D_k} > -r$. 
For $z>0$, let us write $q_z$ for the density of the first hitting time of $-z$ by a standard Brownian motion, 
then:

\begin{align*}
&\E\Bigl[\frac{p_{\Delta^L_k/2}(W_{D_k}-X^L_{\Delta^L_k/2})}{p_{\Delta^L_k}(W_{D_k}-W_{G_k})} \ind{W_{G_k}, W_{D_k}>-r} \bigm| X^L_{\Delta^L_k/2}, W_{G_k}, \Delta_k, \Delta^L_k\Bigr]
\\
&= \frac{1}{\P(E_k)} \int_0^{\Delta_k} \d t \int_{-r}^\infty \d y\ q_{r_{k}+W_{G_k}}(t) p_{\Delta_k-t}(y+r_{k})  \frac{p_{\Delta^L_k/2}(y-X^L_{\Delta^L_k/2})}{p_{\Delta^L_k}(y-W_{G_k})} \ind{W_{G_k}>-r}
\\
&= \frac{1}{\P(E_k)} \int_0^{\Delta_k} \d t \frac{q_{r_{k}+W_{G_k}}(t)}{2\sqrt{\pi(\Delta_k-t)}}  \\
& \qquad  \qquad \qquad \cdot \int_{-r}^\infty \d y \exp\Bigl(\frac{(y-W_{G_k})^2}{2\Delta^L_k} - \frac{(y-X^L_{\Delta^L_k/2})^2}{\Delta^L_k} - \frac{(y+r_{k})^2}{2(\Delta_k-t)}\Bigr) \ind{W_{G_k}>-r}
.\end{align*}
We have $\P(E_k) = \int_0^{\Delta_k} \d t q_{r_{k}+W_{G_k}}(t)$, it thus remains to prove that the integral in $y$ of the rest of the expression is bounded, when in addition $\Delta_k > 2(r+W_{G_k})^2 \log^6(1/r)$.
Let us first consider the integral over $[W_{G_k}, \infty)$. Indeed, if $y > W_{G_k} > -r$, then $(y-W_{G_k})^2 \leq y^2 \leq (y+r_{k})^2$ and consequently
\begin{align*}
&\frac{1}{2\sqrt{\pi(\Delta_k-t)}} \exp\Bigl(\frac{(y-W_{G_k})^2}{2\Delta^L_k} - \frac{(y-X^L_{\Delta^L_k/2})^2}{\Delta^L_k} - \frac{(y+r_{k})^2}{2(\Delta_k-t)}\Bigr)
\\
&\leq \frac{1}{2\sqrt{\pi(\Delta_k-t)}} \exp\Bigl(- \frac{(y+r_{k})^2}{2} \Bigl(\frac{1}{\Delta_k-t}-\frac{1}{\Delta^L_k}\Bigr) - \frac{(y-X^L_{\Delta^L_k/2})^2}{\Delta^L_k}\Bigr)
.\end{align*}
Suppose first that $\Delta^L_k \geq 2 \Delta_k$ so $\Delta^L_k \geq 2 (\Delta_k-t)$ for any $t \in (0, \Delta_k)$ and thus $\frac{1}{\Delta_k-t}-\frac{1}{\Delta^L_k} \geq \frac{1}{2(\Delta_k-t)}$. In this case, by simply removing the last term in the exponential, we have:
\begin{align*}
&\frac{1}{2\sqrt{\pi(\Delta_k-t)}} \exp\Bigl(- \frac{(y+r_{k})^2}{2} \Bigl(\frac{1}{\Delta_k-t}-\frac{1}{\Delta^L_k}\Bigr) - \frac{(y-X^L_{\Delta^L_k/2})^2}{\Delta^L_k}\Bigr)
\\
&\leq \frac{1}{2\sqrt{\pi(\Delta_k-t)}} \exp\Bigl(- \frac{(y+r_{k})^2}{4(\Delta_k-t)}\Bigr)
,\end{align*}
which is the density of a Gaussian law, in particular its integral over the entire real line equals $1$. If $\Delta^L_k$ is close to $\Delta_k$ and $t$ is small, one has to be more careful.
Suppose now that $\Delta_k \leq \Delta^L_k \leq 2 \Delta_k$, then,
\begin{align*}
&\leq \frac{1}{2\sqrt{\pi(\Delta_k-t)}} \exp\Bigl(- \frac{(y+r_{k})^2}{2} \Bigl(\frac{1}{\Delta_k-t}-\frac{1}{\Delta^L_k}\Bigr) - \frac{(y-X^L_{\Delta^L_k/2})^2}{\Delta^L_k}\Bigr)
\\
&\leq \frac{1}{2\sqrt{\pi(\Delta_k-t)}} \exp\Bigl(- \frac{(y+r_{k})^2}{2} \frac{t}{\Delta_k(\Delta_k-t)} - \frac{(y-X^L_{\Delta^L_k/2})^2}{2\Delta_k}\Bigr)
\\
&\leq \frac{1}{2\sqrt{\pi \Delta_k / 2}} \exp\Bigl(- \frac{(y-X^L_{\Delta^L_k/2})^2}{2 \Delta_k}\Bigr) \ind{0 < t \leq \Delta_k/2}\\
& \qquad \qquad \qquad \qquad \qquad  \qquad + \frac{1}{2\sqrt{\pi(\Delta_k-t)}} \exp\Bigl(- \frac{(y+r_{k})^2}{4 (\Delta_k-t)}\Bigr) \ind{\Delta_k/2 < t < \Delta_k}
.\end{align*}
Again the integral over the entire real line of the upper bound is constant.
Finally, if $-r < y < W_{G_k}$ then $(y-W_{G_k})^2 \leq (r+W_{G_k})^2$. Here we use that we work on the event $\Delta_k > 2(r+W_{G_k})^2 \log^6(1/r) \geq 2 (y-W_{G_k})^2 \log^6(1/r)$. Indeed, this implies:
\begin{align*}
&\frac{1}{2\sqrt{\pi(\Delta_k-t)}} \exp\Bigl(\frac{(y-W_{G_k})^2}{2\Delta^L_k} - \frac{(y-X^L_{\Delta^L_k/2})^2}{\Delta^L_k} - \frac{(y+r_{k})^2}{2(\Delta_k-t)}\Bigr)
\\
&\leq \frac{1}{2\sqrt{\pi(\Delta_k-t)}} \exp\Bigl(\frac{1}{4\log^6(1/r)} - \frac{(y+r_{k})^2}{2(\Delta_k-t)}\Bigr)
,\end{align*}
and again the integral over $y$ of the last line is uniformly bounded above.
This completes the proof of the bound~\eqref{eq:borne_ratio_technique}.

Let us finally lower bound the probability that the Brownian bridge $W^{L,k}$ reaches $-r$ in the first half of the interval when both $-r < W_{G_k} \leq r_{k}$ and $\Delta_k > r_{k}^2$, where we keep the notation $r_{k}=r \log^{10k}(1/r)$.
We argue similarly: conditionally on $\Delta^L_k$, $W_{G_k}$, and $W_{D_k}$, this probability equals
\begin{align*}
&\int_0^{\Delta^L_k/2} \d t\ q_{-(r+W_{G_k})}(t) \frac{p_{\Delta^L_k-t}(W_{D_k}+r)}{p_{\Delta^L_k}(W_{D_k}-W_{G_k})}\\
&= \int_0^{\Delta^L_k/2} \d t\ q_{-(r+W_{G_k})}(t) \sqrt{\frac{\Delta^L_k}{\Delta^L_k-t}} \exp\Bigl(\frac{(W_{D_k}-W_{G_k})^2}{2 \Delta^L_k} - \frac{(W_{D_k}+r)^2}{2(\Delta^L_k-t)}\Bigr)
\\
&\geq \int_0^{\Delta_k/2} \d t\ q_{-(r+W_{G_k})}(t) \sqrt{2} \exp\Bigl(- \frac{(W_{D_k}+r)^2}{2 (\Delta_k-t)}\Bigr)
.\end{align*}
Without the exponential, the last integral is bounded away from $0$.
We then lower bound the expectation of the exponential with respect to $W_{D_k}$.
Recall that conditionally on $W_{G_k}$ and $\Delta_k$, the path $(W_t; t \in [G_k, D_k])$ has the law of a Brownian motion started from $W_{G_k}$ and conditioned on the event $E_k$ that it reaches $-r_{k}$ at some point in the time interval, so by splitting at the first hitting time of this value, we have since we restrict our attention to $t \leq \Delta_k/2$,
\begin{align*}
&\E\Bigl[\exp\Bigl(- \frac{(W_{D_k}+r)^2}{2(\Delta_k-t)}\Bigr) \ind{W_{G_k}, W_{D_k}>-r} \bigm| W_{G_k}, \Delta_k\Bigr]
\\
&= \frac{1}{\P(E_k)} \int_0^{\Delta_k} \d s \int_{-r}^\infty \d y\ \frac{q_{-(r_{k}+W_{G_k})}(s)}{\sqrt{2\pi(\Delta_k-s)}} \exp\Bigl(- \frac{(y+r_{k})^2}{2(\Delta_k-s)} - \frac{(y+r)^2}{2 (\Delta_k-t)}\Bigr) \ind{W_{G_k}>-r}
\\
&\geq \int_0^{\Delta_k} \d s \int_{-r}^\infty \d y\ \frac{q_{-(r_{k}+W_{G_k})}(s)}{\sqrt{2 \pi(\Delta_k-s)}} \exp\Bigl(- \frac{(y+r_{k})^2}{2} \Bigl(\frac{1}{\Delta_k-s}+\frac{1}{\Delta_k-t}\Bigr)\Bigr) \ind{W_{G_k}>-r}
\\
&\geq \int_0^{\Delta_k/2} \d s \int_{-r}^\infty \d y\ \frac{q_{-(r_{k}+W_{G_k})}(s)}{\sqrt{2 \pi \Delta_k}} \exp\Bigl(- 2 \frac{(y+r_{k})^2}{\Delta_k}\Bigr) \ind{W_{G_k}>-r}
.\end{align*}
Consider the integral over $y$ only in the last line, and without de density $q$. It equals half the probability that a Gaussian random variable with mean $-r_{k}$ and variance $\Delta_k/4$ is larger than $-r$, or equivalently that a Gaussian random variable with mean $0$ and variance $1$ is larger than $2(r_{k}-r)/\sqrt{\Delta_k}$. Since $\Delta_k > r_{k}^2$, this probability is indeed bounded away from $0$ uniformly in $k$ and $r$.
Then the integral of $q$ over $s$ equals the probability that a standard Brownian motion reaches $-(W_{G_k}+r_{k})$ in less than $\Delta_k/2$ when $W_{G_k} \leq r_{k} \leq \sqrt{\Delta_k}$, which is again bounded away from $0$.

Combined with the results on $X^{L,k}$, we infer that there exists some constant $c > 0$ such that for all $r<1$, on each interval $[G_k,D_k]$ with $W_{G_k} \leq r_{k}$ and $\Delta_k \geq (r_{k})^2$, there is a probability at least $c$ that $W^{L,k}$ reaches the value $-r$ in the first half of the interval. The same holds for $X^{L,k}$ and since they are independent, they both reach the value $-r$ in the first half with probability at least $c^2$. We can then conclude from Lemma~\ref{lem:ruine} that this occurs for a positive proportion of indices $k \leq K_r$.
\end{proof}

\section{Computation of the dimensions}
\label{sec:dim_preuves}

Throughout this section, we let 
$\Loop(X^{\exc})$ denote the looptree constructed from the excursion $X^{\exc}$ and as in the previous section, we let $(\Map,d_\Map,p_\Map)$ denote a subsequential limit of discrete random maps, related to a pair $(X^{\exc}, Z^{\exc})$, from Theorem~\ref{thm:convergence_looptrees_labels_Levy} and Theorem~\ref{thm:convergence_cartes_Boltzmann_Levy}. In addition, thanks to Skorokhod's representation theorem, we assume that all convergences in these theorems hold in the almost sure sense.
Our main goal is to compute the dimensions of these looptrees and maps stated in Theorem~\ref{thm:dimensions_fractales_Looptrees} and Theorem~\ref{thm:dimensions_cartes_Levy} respectively.
The upper bound on the dimensions will follow from H\"older continuity estimates established in 
Section~\ref{ssec:Holder}
while the lower bounds will follow in Section~\ref{ssec:preuves_dimensions} from the estimates on the volume of balls with small radius established in the previous section.

\subsection{H\"older continuity of the distances}
\label{ssec:Holder}

The main goal of this subsection is to prove that the canonical projections $[0,1] \to \Loop(X^{\exc})$ and $[0,1] \to \Map$ are almost surely H\"older continuous.
We will easily derive from this the upper bound for the Hausdorff and upper Minkowski dimensions.
Let us start with the looptree.
Recall from~\eqref{eq:exposants_BG} the upper and lower exponents $\bdeta$ and $\bdgamma$ of $\psi$ at infinity.

\begin{prop}\label{prop:looptree_Holder_borne_sup}
For any $a \in(0, 1/\bdeta)$, 
there exists $C_{a} \in (0, \infty)$ such that:
\begin{equation}\label{eq:looptree_Holder}
d_{\Loop(X^{\exc})}(s,t) \leq C_{a} \cdot |t-s|^{a}
\qquad\text{for all }s,t \in [0,1].
\end{equation}
\end{prop}

We shall see that the exponent $1/\bdeta$ is optimal in the sense that almost surely, no such $C_a<\infty$ exist for $a > 1/\bdeta$, see Remark~\ref{rem:looptree_Holder_borne_inf}.
As the proof will show, the claim also holds for the unconditioned L\'evy process $X$ and its bridge version.

We shall need the following tail bound that we state as a separate lemma.
The `universal constant' does not depend on the L\'evy process.

\begin{lem}\label{lem:ineq_max_Levy}
There exists a universal constant $C$ such that the following holds.
For every L\'evy process with no negative jump, for any $t,x > 0$, we have
\[\P\Bigl(\sup_{s \in [0,t]} |X_{s}| > x\Bigr) + t\, \overline{\pi}(x) \leq C\, t\, (2d^+ x^{-1} + \psi(x^{-1})),\]
where $d^+ = \max(d,0)$ is the positive part of the drift and $\overline{\pi}(x) = \pi((x,\infty))$ is the tail of the L\'evy measure.
\end{lem}

\begin{proof}
Let us focus on the probability on the left.
Such a maximal inequality has been obtained by Pruitt~\cite{Pruitt81} for L\'evy processes with no Gaussian component (but allowing for both positive and negative jumps).
Precisely, Equation~3.2 there reads, when $\beta=0$:
\[\P\Bigl(\sup_{s \in [0,t]} |X_{s}| > x\Bigr) \leq C\, t\, h(x),\]
where, taking note of the difference in the convention for the L\'evy--Khintchine formula,
\begin{align*}
h(x)
&= \overline{\pi}(x) + x^{-2} \int_{(0,x]} r^{2} \pi(\d r) + x^{-1} \Bigl|d - \int_{(x,\infty)} r \pi(\d r)\Bigr|
\\
&\le \overline{\pi}(x) + x^{-2} \int_{(0,x]} r^{2} \pi(\d r) + x^{-1} \int_{(x,\infty)} r \pi(\d r) + |d| x^{-1}
\\
&\le 2 \Bigl(x^{-2} \int_{(0,x]} r^{2} \pi(\d r) + x^{-1} \int_{(x,\infty)} r \pi(\d r) + |d| x^{-1}\Bigr)
,\end{align*}
after noticing $\overline{\pi}(x) \leq x^{-1} \int_{(x, \infty)} r \pi(\d r)$.

We may add the Gaussian component by recalling the identity $X=\beta W + Y$ in distribution, where $W$ is a standard Brownian motion, and where $Y$ is independent of $W$ and has no Gaussian component.
By the reflection principle, a well-known tail bound for Gaussian random variables, and finally the easy bound $\e^{u} \ge \mathrm{e} \times u$ for every $u \ge 0$,
we have:
\[\P\Bigl(\sup_{s \in [0,t]} |\beta W_{s}| > x\Bigr)
\leq 4 \P(W_{t} > x/\beta)
\leq 4 \exp\Bigl(-\frac{x^{2}}{2 \beta t}\Bigr)
\leq \frac{4}{\mathrm{e}} \frac{2 t \beta}{x^{2}}
.\]
A union bound then yields:
\begin{align*}
\P\Bigl(\sup_{s \in [0,t]} |X_{s}| > x\Bigr) + t\, \overline{\pi}(x) 
&\le \P\Bigl(\sup_{s \in [0,t]} |\beta W_{s}| > x\Bigr) + \P\Bigl(\sup_{s \in [0,t]} |Y_{s}| > x\Bigr) + t\, \overline{\pi}(x) 
\\
&\leq \frac{8}{\mathrm{e}} t \beta x^{-2} + C\, t\, h(x) + t\, \overline{\pi}(x) 
,\end{align*}
which is bounded by some universal constant times $t$ times
\[x^{-2} \int_{(0,x]} r^{2} \pi(\d r) + x^{-1} \int_{(x,\infty)} r \pi(\d r) + |d| x^{-1} + \beta x^{-2}.\]

Let $\psi_0(\lambda) = \psi(\lambda)+d\lambda-\beta \lambda^2$ denote the Laplace exponent of the process with its drift and Gaussian component removed and let $\phi=\psi_0'$ denote its derivative. Then $\phi$ is the Laplace exponent of a subordinator without drift and L\'evy measure $\nu(\d r) = r \pi(\d r)$. In the last display, the sum of the first two terms, with the integrals, corresponds to $x^{-2} J_\phi(x)$ in the notation of~\cite[Equation~47]{Duq12}. The next equation there, taken from~\cite[Proposition~III.1]{Ber96}, shows that it lies between two universal constant times $x^{-1} \phi(x^{-1}) = x^{-1} \psi_0'(x^{-1})$, which itself, by~\cite[Equation~50]{Duq12} lies between two universal constant times $\psi_0(x^{-1})=\psi(x^{-1})+dx^{-1}-\beta x^{-2}$.
Therefore, our last display lies between two universal constant times
$\psi(x^{-1})+(d+|d|) x^{-1} = \psi(x^{-1})+2d^+ x^{-1}$.
\end{proof}

The exact result that we shall need in order to prove Proposition~\ref{prop:looptree_Holder_borne_sup}, based on this lemma, is the following.
Recall that $\bdeta$ is the exponent defined in~\eqref{eq:exposants_BG}.
For any $\kappa, t >0$, we let $N_\kappa(t) = \#\{s \in [0,t] \colon \Delta X_s > \kappa^{-1}\}$ and
$X^\kappa = (X^\kappa_t)_{t \ge 0}$ the process obtained from $X$ by removing all these jumps larger than $\kappa^{-1}$.

\begin{cor}\label{cor:ineq_max_Levy_tronque}
Fix $N \ge 1$. There exists a constant $C_N$ such that for any $t > 0$ and any $\delta>0$, we have
\[\P\Bigl(\sup_{s \in [0,t]} |X^\kappa_{s}| > 2 N \kappa^{-1}\Bigr) + \P(N_\kappa(t) \geq N) 
\leq C_N t^N \kappa^{N(\bdeta+\delta)}
,\]
for all $\kappa > 0$ large enough, depending in $\delta$ and $\psi$.
\end{cor}

\begin{proof}
Note that $X^\kappa$ is a L\'evy process with the same drift and Gaussian component as $X$ and with L\'evy measure $\pi^\kappa(\cdot) = \pi(\cdot \cap [0,\kappa^{-1}])$. 
Since its Laplace exponent is smaller than or equal to $\psi$, then we may apply Lemma~\ref{lem:ineq_max_Levy} to $X^\kappa$ with $\psi$ in the upper bound.
Because $X^\kappa$ makes only jumps smaller than or equal to $\kappa^{-1}$, at the first time it exceeds $x$ in absolute value, it cannot exceed $x+\kappa^{-1}$.
With the help of the strong Markov property, we obtain by induction:
\[\P\Bigl(\sup_{s \in [0,t]} |X^\kappa_{s}| > 2 N \kappa^{-1}\Bigr) 
\leq \bigl(C\, t (2 d^+ \kappa + \psi(\kappa))\bigr)^N.\]
In addition, the random variable $N_\kappa(t)$ has the Poisson distribution with rate $t\, \overline{\pi}(\kappa^{-1})$.
The Chernoff bound shows that the tail probability of a Poisson random variable $Y$ with rate $\lambda$ satisfies $\P(Y \geq k) \leq \e^{-\lambda} (\e \lambda/k)^{k}$ for every $k>\lambda$.
Thus, for $c_N = (\e/N)^{N}$, we have
\[\P(N_\kappa(t) \geq N) 
\leq c_N \bigl(t\, \overline{\pi}(\kappa^{-1})\bigr)^N
\leq c_N \bigl(C\, t (2 d^+ \kappa + \psi(\kappa))\bigr)^N
\]
by Lemma~\ref{lem:ineq_max_Levy}.
By the very definition of $\bdeta$ we have $\psi(\kappa) =o(\kappa^{\bdeta+\delta})$ as $\kappa \to \infty$ as well as $\kappa = o(\kappa^{\bdeta+\delta})$ since $\bdeta+\delta>1$, which concludes the proof.
\end{proof}

We may now prove our H\"older continuity result.

\begin{proof}[Proof of Proposition~\ref{prop:looptree_Holder_borne_sup}]
We claim that given $a < 1/\bdeta$, there exists a deterministic constant $K_a \in (0,\infty)$ which only depends on $a$, not even on the law of the L\'evy process, such that almost surely, there exists $n_0 \geq 1$ (random, and which may depend on the law of the L\'evy process) such that for every $n \geq n_0$, we have:
\begin{equation}\label{eq:looptree_Holder_petits_dyadiques}
\sup_{0 \leq i \leq 2^{n}-1} d_{\Loop(X^{\exc})}(i 2^{-n}, (i+1) 2^{-n}) \leq K_a 2^{-an}.
\end{equation}
For $1 \le n < n_0$, there are only finitely many quantities involved, so we infer that almost surely, there exists a now random $C_a \in (0,\infty)$ such that~\eqref{eq:looptree_Holder} holds for every pair $(s,t)$ of the form $(i 2^{-n}, (i+1)2^{-n})$ with $n \geq 1$ and $i \in \{0, \dots, 2^{n}-1\}$. It then extends to all dyadic numbers by a routine argument, after multiplying $C_a$ by $2/(1-2^{-a})$, and then to the whole interval $[0,1]$ by density, using that $d_{\Loop(X^{\exc})}$ is continuous almost surely.

It remains to prove~\eqref{eq:looptree_Holder_petits_dyadiques}. 
Let us first replace $X^{\exc}$ by the unconditioned process $X$ and then argue by absolute continuity. 
According to~\eqref{eq:borne_dist_looptree_par_la_droite}, almost surely, for $s<t$, we have:
\[d_{\Loop(X)}(s,t) 
\leq X_{s} + X_{t-} - 2 \inf_{r \in [s,t)} X_{r} 
\leq 4 \sup_{r \in [s,t)} |X_{r}-X_{s}|.\]
Recall indeed that, informally the term in the middle is the length of the path if one goes from $s$ to $t$ in the looptree by always following a cycle on its right, instead of optimising between left and right.
Note that the value of the jump of $X$ at time $t$, if any, does not affect the looptree distance, and not even the upper bound. 
This bound is too crude to be used directly because of the jumps of $X$, but we shall use it between consecutive large jumps since the possible jumps at time $s$ and $t$ do not appear in the upper bound.

Precisely, fix $\delta>0$, let $\theta = (1+\delta) (\bdeta+\delta)$, and let us prove~\eqref{eq:looptree_Holder_petits_dyadiques} with $a = 1/\theta$ by relying on Lemma~\ref{lem:ineq_max_Levy} and Corollary~\ref{cor:ineq_max_Levy_tronque} with $\kappa = 2^{n/\theta}$.
Indeed, fix $n \geq 1$ and $i \in \{0, \dots, 2^{n}-1\}$. 
Let $t_{n,i}(1) < \dots < t_{n,i}(N_{n,i})$ denote the increasing enumeration of the set of times $\{t \in (i 2^{-n}, (i+1) 2^{-n}) \colon \Delta X_{t} > 2^{-n/\theta}\}$. Set also $t_{n,i}(0) = i 2^{-n}$ and 
$t_{n,i}(N_{n,i}+1) = (i+1) 2^{-n}$.
Then combining the last display with the triangle inequality, we have:
\begin{equation}\label{eq:borne_looptre_par_la_droite}
d_{\Loop(X)}(i 2^{-n}, (i+1) 2^{-n}) 
\leq 4 \sum_{j=0}^{N_{n,i}} \sup_{r \in [t_{n,i}(j), t_{n,i}(j+1))} |X_r - X_{t_{n,i}(j)}|
.\end{equation}
Note that conditionally on the $t_{n,i}(j)$'s, the paths between two consecutive such times behave like the process $X^\kappa$ with $\kappa = 2^{n/\theta}$ from Corollary~\ref{cor:ineq_max_Levy_tronque} whereas $N_{n,i}$ has the same law as $N_{2^{n/\theta}}(2^{-n})$.
Some union bounds, invariance by time translation of the L\'evy process, and the upper bound $t_{n,i}(j+1)-t_{n,i}(j) \le 2^{-n}$, lead to the following bound for every $N\ge1$:
\begin{align*}
&\P\Bigl(\sum_{j=0}^{N_{n,i}} \sup_{r \in [t_{n,i}(j), t_{n,i}(j+1)]} |X_{r}-X_{t_{n,i}(j)}| > 2 N^2 2^{-n/\theta}\Bigr)
\\
&\le \P(N_{n,i} \geq N) + \P\Bigl(N_{n,i} < N \enskip\text{and}\enskip \sup_{0 \le j < N} \sup_{r \in [t_{n,i}(j), t_{n,i}(j+1)]} |X_{r}-X_{t_{n,i}(j)}| > 2 N 2^{-n/\theta}\Bigr)
\\
&\le \P(N_{2^{n/\theta}}(2^{-n}) \geq N) + N \P\Bigl(\sup_{r \in [0, 2^{-n}]} |X^{2^{n/\theta}}_{r}| > 2 N 2^{-n/\theta}\Bigr)
.\end{align*}
According to Corollary~\ref{cor:ineq_max_Levy_tronque}, for every $n$ large enough, the last line is bounded above by some constant that only depends on $N$ times
\[(2^{-n} 2^{n (\bdeta+\delta)/\theta})^N
= 2^{-nN \delta / (1+\delta)}
.\]
Choose any $N > (1+\delta)/\delta$ and let $q = \delta N / (1+\delta) - 1 > 0$.
Then a last union bound implies that, for some constant $C_N$, for every $n$ large enough,
\[\P\Bigl(\exists 0 \le i < 2^n \colon \sum_{j=0}^{N_{n,i}} \sup_{r \in [t_{n,i}(j), t_{n,i}(j+1)]} |X_{r}-X_{t_{n,i}(j)}| > 2 N^2 2^{-n/\theta}\Bigr)
\leq C_N 2^{-n q}
,\]
The Borel--Cantelli lemma shows that almost surely these events eventually fail for $n$ large enough. 

For the conditional law, recall that the first half the bridge $(X^{\br}_{t} ; 0 \leq t \leq 1/2)$ is absolutely continuous with respect to the unconditioned process $(X_{t} ; 0 \leq t \leq 1/2)$.
Hence almost surely, these events eventually fail for $n$ large enough in $X^{\br}$, when restricting to $i \in \{0, \dots, 2^{n-1}\}$.
In addition, the time- and space-reversal of the second half of the bridge has the same law as the first half, so this also holds for the second half and thus for the entire bridge.
Finally we transfer it to the excursion by applying the Vervaat transform. Notice that the latter changes the starting point and thus shifts the intervals, but this only mildly affects the constant $2 N^2$.

In each case, we obtain the desired H\"older property~\eqref{eq:looptree_Holder_petits_dyadiques} thanks to the deterministic bound~\eqref{eq:borne_looptre_par_la_droite}.
\end{proof}

\begin{rem}\label{rem:looptree_Holder_borne_sup}
If one wanted to replace the exponent $1/\bdeta$ by the smaller one $1/\bdgamma$, then instead of the upper bound 
$\psi(x) \leq x^{\bdeta+\delta}$ for every $x$ large enough, one would only have the existence of a sequence $x_{n} \to \infty$ along which $\psi(x_{n}) \leq x_{n}^{\bdgamma+\delta}$. Taking a subsequence if necessary, we may assume that $\sum_{n} x_{n}^{-1} < \infty$ and then the preceding argument shows that instead of~\eqref{eq:looptree_Holder_petits_dyadiques} for any $a \in (0, 1/\bdgamma)$, almost surely there exist $K_{a} > 0$ and a sequence $\varepsilon_{n} \downarrow 0$ such that, for every $n$ large enough,
\[d_{\Loop(X^{\exc})}(i \varepsilon_{n}, (i+1) \varepsilon_{n}) \leq K_{a} \varepsilon_{n}^{a}
\qquad\text{for all}\enskip i \in \{0, \dots, \lfloor \varepsilon_{n}^{-1}\rfloor-1\}.\]
This will be used to upper bound the packing and lower Minkowski dimensions.
\end{rem}

We then turn to the map $M$ and the label process $Z^{\exc}$.
Informally, we simply lose a factor $1/2$ due to the regularity of the Brownian labels.
Again, we state the result under the excursion law but it holds similarly under the unconditioned or bridge law.

\begin{cor}\label{cor:labels_Holder}
For any $a \in(0, 1/(2\bdeta))$, the label process $(Z^{\exc}_t)_{t \in [0,1]}$ has a modification that is H\"older continuous with exponent $a$.
\end{cor}

\begin{proof}
Recall the moment bound on the Brownian labels~\eqref{eq:label_moment_Kolmogorov_looptrees}, then if the H\"older continuity~\eqref{eq:looptree_Holder} on the looptree distances held true with a deterministic constant $C_a$, then we could apply the classical Kolmogorov continuity criterion, showing that for any $a < 1/\bdeta$ and any $q>0$, the process $(Z^{\exc}_t)_t$ has a modification that is H\"older continuous with any exponent $b < q^{-1} (aq/2-1)$, which is arbitrarily close to $1/(2\bdeta)$ when $q$ is large and $a$ is close to $1/\bdeta$.

In order to do so, let $A_n = \{n-1 < C_a \le n\}$ for each $n \ge 1$,  so these events are disjoint and their union has probability $1$.
Let us set $\N' = \{n \ge 1 \colon \P(A_n) \ne 0\}$ and $A_0 = (\bigcup_{n \in \N'} A_n)^c$.
Then for every $n \in \N'$ we can apply Kolmogorov's criterion under $\P(\,\cdot\mid A_n)$ and deduce the existence of a modification $Z^n$ that is H\"older continuous with any exponent $b<1/(2\bdeta)$, under the conditional law.
Let $Z'$ coincide with $Z^n$ on $A_n$ for every $n \in \N'$ and let $Z'$ be a fixed arbitrary value on $A_0$. 
Then for every $t \in [0,1]$, since $\P(A_0)=0$, then we have 
$\P(Z^{\exc}_t = Z'_t) = \sum_{n \in \N'} \P(A_n) \P(Z^{\exc}_t = Z^n \mid A_n) = 1$ so $Z'$ is a modification of $Z^{\exc}$. By the same argument $Z'$ is almost surely H\"older continuous with any exponent $b<1/(2\bdeta)$.
\end{proof}

From now on, we always work with these H\"older continuous modifications.
This further implies that the distances in the map are H\"older continuous.

\begin{prop}\label{prop:carte_Holder_borne_sup}
For any $a \in(0, 1/(2\bdeta))$, 
there exists $C_{a} \in (0, \infty)$ such that:
\begin{equation}\label{eq:carte_Holder}
d_\Map(s,t) \leq C_{a} \cdot |t-s|^{a}
\qquad\text{for all }s,t \in [0,1].
\end{equation}
\end{prop}

Similarly to Proposition~\ref{prop:looptree_Holder_borne_sup}, we shall see that the exponent $1/(2\bdeta)$ is optimal in the sense that almost surely, no such $C_a<\infty$ exist for $a > 1/(2\bdeta)$, see again Remark~\ref{rem:looptree_Holder_borne_inf}.

\begin{proof}
The claim follows from the upper bound $d_\Map \leq D^\circ$ from Theorem~\ref{thm:convergence_cartes_Boltzmann_Levy}, where we recall that $D^\circ$ is defined by
\[D^\circ(s,t) = D^\circ(t,s) = Z^{\exc}_s + Z^{\exc}_t - 2 \inf_{[s,t]} Z^{\exc}.\]
Then the H\"older continuity of $Z^{\exc}$ from Corollary~\ref{cor:labels_Holder} transfers to $D^\circ$
which thus satisfies the upper bound~\eqref{eq:carte_Holder}, and thus so does $d_M$.
\end{proof}

\subsection{Fractal dimensions}
\label{ssec:preuves_dimensions}

We are now ready to compute the dimensions of our spaces.
Recall that Theorem~\ref{thm:dimensions_fractales_Looptrees} about the looptrees precisely states that almost surely, we have:
\[\dim_{H} \Loop(X^{\exc}) = \diminf \Loop(X^{\exc}) = \bdgamma
\qquad\text{and}\qquad
\dim_{p} \Loop(X^{\exc}) = \dimsup \Loop(X^{\exc}) = \bdeta.\]
For the maps, Theorem~\ref{thm:dimensions_cartes_Levy} states that the dimensions are twice these quantities.

\begin{proof}[Proof of Theorem~\ref{thm:dimensions_fractales_Looptrees} and Theorem~\ref{thm:dimensions_cartes_Levy}]
Let us refer to~\cite{Mat95} for the definition and basic properties of these dimensions. Let us only recall here 
the following relations between these notions of dimension:
\[\dim_{H} \leq \diminf
\qquad\text{and}\qquad
\dim_{H} \leq \dim_{p} \leq \dimsup.\]
It thus suffices to upper bound $\diminf$ and $\dimsup$ and to lower bound $\dim_{H}$ and $\dim_{p}$.

The upper bound for $\dimsup$ is an immediate consequence of the H\"older continuity provided by Proposition~\ref{prop:looptree_Holder_borne_sup} and Proposition~\ref{prop:carte_Holder_borne_sup}. Indeed, fix $\delta>0$ and let $N_{n} = \lfloor 2^{n(\bdeta+\delta)}\rfloor$; then by the former result, almost surely, for every $1 \leq i \leq N_{n}$, the diameter of the image of $[i N_{n}^{-1}, (i-1) N_{n}^{-1}]$ in the looptree is small compared to $2^{-n}$. Hence, almost surely, for every $n$ large enough, the space $\Loop(X^{\exc})$ can be covered by $N_{n}$ balls with radius $2^{-n}$ and this implies that $\dimsup \Loop(X^{\exc}) \leq \bdeta+\delta$ almost surely with $\delta$ arbitrarily small.
The upper bound $\diminf \Loop(X^{\exc}) \leq \bdgamma$ is similar, using Remark~\ref{rem:looptree_Holder_borne_sup}: with the sequence $(\varepsilon_{n})_{n \geq 1}$ from this remark, we may apply the previous reasoning with $N_{n} = \lfloor \varepsilon_{n}^{-\bdgamma-\delta}\rfloor$.
The upper bounds for the maps are proved similarly using Proposition~\ref{prop:carte_Holder_borne_sup} and the analogue of Remark~\ref{rem:looptree_Holder_borne_sup}, which is derived from the latter.

The lower bounds on the Hausdorff and packing dimensions follow from our estimate on the volume of small balls. 
Recall that we associate times on $[0,1]$ by their projection in the looptree; let then $B(u,r)$ denote the ball centred at $u \in [0,1]$ and with radius $r \ge 0$ for the looptree distance. Let also $|\cdot|$ denote the Lebesgue measure on $[0,1]$.
Let $U$ be an independent uniform random time on $[0,1]$ and $\delta > 0$.
We read from Proposition~\ref{prop:borne_sup_volume_boules_looptrees_psi} that almost surely, for every $n$ large enough, we have:
\[|B(U,2^{-n})| \leq \psi(2^{n(1-\delta)})^{-1}.\]
By definition of $\bdgamma$, we know that $2^{-n(1-\delta)(\bdgamma-\delta)} \psi(2^{n(1-\delta)}) \to \infty$. Therefore, almost surely, for every $r>0$ small enough, if $2^{-(n+1)} \leq r \leq 2^{-n}$, then
\[|B(U, r)|
\leq |B(U,2^{-n})|
= o(2^{-n(1-\delta)(\bdgamma-\delta)})
= o(r^{(1-\delta)(\bdgamma-\delta)})
.\]
On the other hand, by definition of $\bdeta$, we know that $2^{-n(1-\delta)(\bdeta-\delta/2)} \psi(2^{n(1-\delta)})$ does not tend to $0$, so there is exists a subsequence which is bounded away from $0$ and we infer similarly that for some $r_n \to 0$, it holds:
\[|B(U,\varepsilon_n)|
= o(r_n^{(1-\delta)(\bdeta-\delta)})
.\]
This shows that almost surely, for Lebesgue almost every $u \in [0,1]$, we have both:
\[\limsup_{r \downarrow 0} \frac{|B(u, r)|}{r^{(1-\delta)(\bdgamma-\delta)}} = 0
\qquad\text{and}\qquad
\liminf_{r \downarrow 0} \frac{|B(u, r)|}{r^{(1-\delta)(\bdeta-\delta)}} = 0.\]
We conclude from density theorems that $\dim_{H} \Loop(X^{\exc}) \ge \bdgamma-\delta$ and $\dim_{p} \Loop(X^{\exc}) \ge \bdeta-\delta$, see e.g.~\cite[Theorem~6.9 and~6.11]{Mat95}.
Again, this is adapted to the maps by simply replacing the estimate from Proposition~\ref{prop:borne_sup_volume_boules_looptrees_psi} by that from Proposition~\ref{prop:borne_sup_volume_boules_cartes_psi}.
\end{proof}

\begin{rem}\label{rem:looptree_Holder_borne_inf}
Notice that the upper bounds on the Hausdorff dimensions only rely on the H\"older continuity results and extend as far as the latter do. Therefore if the canonical projection on the looptree (resp.~on the map) was H\"older continuous with exponent $a > 1/\bdeta$ (resp.~$a > 1/(2\bdeta))$, then we would upper bound the dimension by $1/a < \bdeta$ (resp.~$1/a < 2\bdeta)$, and this would contradict the lower bounds.
In other words, if the bound~\eqref{eq:looptree_Holder} held with $a>1/\bdeta$, then by taking $\delta>0$ small enough, this would contradict the liminf in the last display of the previous proof. 
Therefore the exponents in Proposition~\ref{prop:looptree_Holder_borne_sup} and Proposition~\ref{prop:carte_Holder_borne_sup} are optimal.
\end{rem}

\section{Approximations in the continuum world}
\label{sec:limites_Levy}

In this section we ask ourselves the following question. Suppose that for all $\lambda>0$ we are given a L\'evy process $X^{(\lambda)}$ with no negative jump and infinite variation paths, with exponent $\psi^{(\lambda)}$, and another one $X$ with exponent $\psi$.
Suppose that $X^{(\lambda)}$ and $X$ satisfy the integrability condition~\eqref{eq:condition_integrale}, so they all have continuous transition densities $p^{(\lambda)}_t(x)$ and $p_t(x)$ respectively, and we can define bridges and excursions as before. If these transition densities converge, then it is well known that the L\'evy processes converge in law; but do the corresponding looptrees, Gaussian labels, and continuum maps converge as well? We prove the convergence of labelled looptrees in Section~\ref{ssec:convergence_looptree_Levy} and that of maps (along a subsequence) in Section~\ref{ssec:convergence_cartes_Levy}.

One could for instance consider the case where $X^{(\lambda)}$ is an $\alpha$-stable spectrally positive L\'evy process with drift $\vartheta$, as mentioned in Section~\ref{sec:biconditionnement}, with $\alpha \in (1,2)$ and $\vartheta \in \R$ that may both vary. We will study this in a companion paper.

Let us say a few words on the techniques used to derive these results. If the limit $X$ has no Gaussian part, then it has the pure jump property from~\cite{Mar24} and Section~4 there answers positively our question, see alternatively~\cite[Theorem~3~(iii)]{Kha22} for the convergence of looptrees in the pure jump case. 
However, here we consider the general setting and we need additional arguments.
First, tightness of the looptree distances is easy but that of the labels is not: we argue by adapting the proof of the usual Kolmogorov criterion and using some general concentration bounds for Gaussian processes.
We then turn to the 
finite-dimensional marginals, that is, the pairwise distances and joint labels of finitely many points in the looptree.
The contribution of the large cycles can easily be controlled, and the contribution of the microscopic parts is proved by a sort of weak law of large numbers in the spirit of~\cite{CK14}.
Indeed, informally, the L\'evy excursion codes the length of paths in the looptree when one always follows the cycles on their right, whereas the true distance minimises left and right lengths. 
We show that for microscopic cycles, this basically divides their lengths by two.
We actually argue on the unconditioned process first and then use a local absolute continuity argument for the excursion; the convergence of the densities is used for this transfer.
Once the convergence of labelled looptrees is proved, the convergence of maps along a subsequence, and the identification of the Brownian sphere and tree in the extreme regimes, is routine as the usual arguments to control discrete maps can be adapted.

Let us start as a warmup with an easy lemma.

\begin{lem}\label{lem:cv_densite_excursion}
Suppose that for every $t \geq 0$, the transition density $p^{(\lambda)}_t$ converges uniformly to $p_t$ as $\lambda \to \infty$.
Then $X^{(\lambda)} \to X$ in distribution for the Skorokhod $J_{1}$ topology under the unconditional law as well as under $\P(\,\cdot\mid X^{(\lambda)}_1=\vartheta)$ for any $\vartheta \in \R$, and finally under the excursion law.
\end{lem}

\begin{proof}
The convergence of the densities implies the convergence in distribution $X^{(\lambda)}_t \to X_t$ for any $t \geq 0$ fixed under the unconditional law. Since they are L\'evy processes, this implies the convergence in distribution for the Skorokhod $J_{1}$ topology, see e.g.~\cite[Chapter~VII]{JS03} for details. Next, fix $\vartheta \in \R$, then for every $\varepsilon \in (0,1)$, recall that the bridges satisfy the absolute continuity relation:
\[\E[F(X^{(\lambda)}_t ; t \in [0, 1-\varepsilon]) \mid X^{(\lambda)}_1=\vartheta]
= \E\Bigl[F(X^{(\lambda)}_t ; t \in [0, 1-\varepsilon]) \cdot \frac{p^{(\lambda)}_{\varepsilon}(\vartheta-X^{(\lambda)}_{1-\varepsilon})}{p^{(\lambda)}_1(\vartheta)}\Bigr]
,\]
for any continuous and bounded function $F$. We infer from the convergence in distribution of the unconditioned processes and the uniform convergence $p^{(\lambda)}_{\varepsilon} \to p_{\varepsilon}$ that the right-hand side converges to the same quantity for $X$, hence the bridges converge on any interval $[0, 1-\varepsilon]$. Notice then that the bridges are invariant under space- and time-reversal (since L\'evy processes are), so the part on the interval $[1-\varepsilon, 1]$ converges as well, 
and in particular it is tight and so the entire bridge on $[0,1]$ is tight.
The convergence on any interval $[0, 1-\varepsilon]$ then characterises uniquely the subsequential limits.
Finally, the convergence of bridges for $\vartheta=0$ is transferred to the excursions by the Vervaat transform.
\end{proof}

In practice, the uniform convergence of transition densities can  often be checked by using the inverse Fourier transform. Precisely, we may write:
\[
\sup_{x \in \R} |p^{(\lambda)}_t(x) - p_t(x)| 
\leq \frac{1}{2\pi} \int_{-\infty}^\infty |\E[\e^{i u X^{(\lambda)}_t}] - \E[\e^{i u X_t}]| \d u
.\]
The right-hand side tends to $0$ provided the integrand does and one can e.g.~apply the dominated convergence theorem.
This will be the case in our application in Section~\ref{sec:biconditionnement}.

We assume for the rest of this section that the assumption of Lemma~\ref{lem:cv_densite_excursion} is in force.

\subsection{Convergence of looptrees and labels}
\label{ssec:convergence_looptree_Levy}

Let us first study the convergence of the looptrees and Gaussian labels associated with $X^{(\lambda)}$ and $X$.
Recall from~\eqref{eq:def_distance_looptree} and~\eqref{eq:def_labels_browniens} respectively that they are defined using a parameter that
acts as a multiplicative constant on the continuous part.
In the next theorem, we fix the value of this constant: $1/2$ for looptrees and $1/3$ for the labels.
We discuss this restriction after the statement.

\begin{thm}\label{thm:convergence_Levy_looptree_label}
Suppose that for every $t \geq 0$, the transition density $p^{(\lambda)}_t$ converges uniformly as $\lambda \to \infty$ to $p_t$.
Under the unconditional law or the bridge or excursion law, we have jointly with the convergence in distribution $X^{(\lambda)} \to X$:
\[\Loop^{1/2}(X^{(\lambda)}) \cvloi[\lambda] \Loop^{1/2}(X)
\qquad\text{and}\qquad
Z^{(\lambda,1/3)} \cvloi[\lambda] Z^{1/3}\]
for the Gromov--Hausdorff--Prokhorov topology and
for the uniform topology respectively.
\end{thm}

As was previously mentioned, in the particular case where $X$ has no Gaussian part, then it has the pure jump property and the convergence of looptrees follows from~\cite[Section 4]{Mar24}, or alternatively~\cite[Theorem~1.4]{Kha22}. 
In the general case, Theorem~1.6 in~\cite{Kha22} proves the convergence of a `shuffled' version of the looptrees, provided that the excursions converge and under some technical assumption on this shuffling. This technical assumption is proved to hold for stable looptrees, and also that their shuffled version has the same law as the original looptree~\cite[Theorem~6.6]{Kha22}.
A possible alternative proof of the convergence in distribution of looptrees in Theorem \ref{thm:convergence_Levy_looptree_label} could thus be to extend this approach to L\'evy looptrees, as noted in~\cite[Remark 6.8]{Kha22}. This approach, however, does not seem to
directly give joint convergence of the looptrees and the L\'evy processes.

Let us discuss the choice of constants in this theorem.
First notice that in absence of jumps in the limit, namely if $X$ is a scaled Brownian 
path, then for any $c > 0$, the looptree $\Loop^{c}(X) = c\, \Loop^{1}(X)$ is a scaled Brownian tree and $Z^c = \sqrt{c}\, Z$ is a scaled Brownian snake so the choice of $c$ amounts to fix a normalisation. 
In the opposite case, if this continuous part in the limit vanishes, then the choice of this constant is irrelevant. Similarly, 
if the continuous part in $X^{(\lambda)}$ vanishes, then this constant is irrelevant in the sequence.
The reason for this restriction is the following. The continuous part for $X$ appears as the limit of the continuous part of $X^{(\lambda)}$ together with sum of quantities of the form $R^t_s$ as defined in~\eqref{eq:def_R}, corresponding to small jumps of $X^{(\lambda)}$. Now typically, the contribution to the looptree distances of these small jumps will be half of the sum of these $R^t_s$'s by symmetry, since the latter code the right-length of the cycles. Hence in order to control the looptrees with a different multiplicative constant than $1/2$, we must understand the precise contribution of the continuous part of $X^{(\lambda)}$ and its small jumps to the continuous part of $X$, whereas with the constant $1/2$, in the looptree, they both contribute equally.
The same applies for the constant $1/3$ for the Gaussian labels.

Let us first provide a tightness argument. We note that the latter does not require to restrict to the parameters $1/2$ and $1/3$.

\begin{proof}[Proof of tightness]
Thanks to the upper bound~\eqref{eq:borne_dist_looptree_par_la_droite} on looptree distances, the convergence of the excursions implies tightness of the associated looptrees. 
Indeed, for any $s, s', t, t' \in [0,1]$ we have by the triangle inequality:
\begin{align*}
|d^{1/2}_{\Loop(X^{(\lambda)})}(s,t) - d^{1/2}_{\Loop(X^{(\lambda)})}(s',t')|
&\leq d^{1/2}_{\Loop(X^{(\lambda)})}(s,s') + d^{1/2}_{\Loop(X^{(\lambda)})}(t,t')
\\
&\leq d_{X^{(\lambda)}}(s,s') + d_{X^{(\lambda)}}(t,t')
.\end{align*}
The last line converges in distribution to $d_{X}(s,s') + d_{X}(t,t')$ under the unconditional law or the bridge or excursion law by Lemma~\ref{lem:cv_densite_excursion}. This limit in turn is arbitrarily small when the pairs $(s,t)$ and $(s',t')$ are arbitrarily close to each other (for the Euclidean distance). Tightness of the looptree distances then follows easily from this control on the modulus of continuity.

Tightness of the label process is less immediate. Let us prove it provided that $d^{1/2}_{\Loop(X^{(\lambda)})} \to d^{1/2}_{\Loop(X)}$ in distribution, which we shall prove below.
Our aim is again to control the modulus of continuity. Precisely, let us fix $\varepsilon, \theta > 0$ and let us prove that there exists $\delta>0$ such that for every $\lambda$ large enough, we have:
\[\P(\exists (s,t) \in [0,1]^2 \colon |t-s| \leq \delta \quad\text{and}\quad |Z^{(\lambda, 1/3)}_s - Z^{(\lambda, 1/3)}_t| > \theta)
\leq 10 \varepsilon.\]
Let $C_a$ be the random variable as in~\eqref{eq:looptree_Holder} for the limit process $X$, for some fixed admissible $a>0$.
Since $C_a<\infty$ almost surely then there exists a deterministic constant $M \in (0,\infty)$ such that $C_a \leq M$ with probability at least $1-\varepsilon$.
Let $C>0$ which will be chosen large enough, depending only on $\theta$ and $M$ (and thus $\varepsilon$), and let us set $\kappa = (\theta / C)^2$ and then 
$N = \lceil (M/\kappa)^{1/a} \rceil = \lceil (MC^2 / \theta^2)^{1/a} \rceil$; we shall take $\delta = 1/N$.
Consider the
good event:
\[G(N,\kappa,\lambda) = \{\forall (s,t) \in [0,1]^2 \enskip\text{such that}\enskip |t-s| \leq 1/N \enskip\text{it holds}\enskip d^{1/2}_{\Loop(X^{(\lambda)})}(s,t) \leq \kappa\}
,\]
and let us write $G(N,\kappa)$ for the same event for the process $X$.
It follows from the convergence of looptree distances, the bound~\eqref{eq:looptree_Holder}, and our choice of constants that
\[\P(G(N,\kappa,\lambda)) \cv[\lambda] \P(G(N,\kappa)) \geq 1 - \varepsilon.\]
We shall work on the event $G(N,\kappa,\lambda)$ and upper bound
\[\P(G(N,\kappa,\lambda) \cap \{\exists (s,t) \in [0,1]^2 \colon |t-s| \leq 1/N \enskip\text{and}\enskip |Z^{(\lambda, 1/3)}_s - Z^{(\lambda, 1/3)}_t| > C \sqrt{\kappa}\})
.\]
By the triangle inequality and a union bound, this probability is less than or equal to
\[N \max_{0 \leq i < N} \P(G(N,\kappa,\lambda) \cap \{\exists s \in [0, 1/N] \colon |Z^{(\lambda, 1/3)}_{s+i/N} - Z^{(\lambda, 1/3)}_{i/N}| > C \sqrt{\kappa} / 4\})
.\]
Recall that conditional on $X^{(\lambda)}$, the label process $Z^{(\lambda, 1/3)}$ is a centred Gaussian process with a so-called canonical or intrinsic pseudo-distance given by:
\[\rho^{(\lambda)}(s,t) = \sqrt{\E[(Z^{(\lambda, 1/3)}_t - Z^{(\lambda, 1/3)}_s)^2 \mid X^{(\lambda)}]} \leq \sqrt{d^{1/2}_{\Loop(X^{(\lambda)})}(s,t)},\]
where the upper bound follows by replacing the product of left and right length along cycles by the minimum.
Let then $N^{(\lambda)}(r,N,i)$ denote the number of $\rho^{(\lambda)}$-balls of radius $r$ needed to cover the interval $[i/N, (i+1)/N]$. Then on the event $G(N,\kappa,\lambda)$ we have $N^{(\lambda)}(r,N,i) = 1$ for $r \geq \sqrt{\kappa}$ and $N^{(\lambda)}(r,N,i) \leq \lceil \sqrt{\kappa} / r \rceil$ for $r < \sqrt{\kappa}$.
Then Dudley's entropy bound~\cite[Corollary~13.2]{BLM13} shows that there exists a universal constant $K$ such that
\begin{align*}
&\E\Bigl[\sup_{s \in [0, 1/N]} |Z^{(\lambda, 1/3)}_{s+i/N} - Z^{(\lambda, 1/3)}_{i/N}| \Bigm| X^{(\lambda)} \Bigr]\, \ind{G(N,\kappa,\lambda)} \\
& \qquad \leq K \int_0^{\infty} \sqrt{\log N^{(\lambda)}(r,N,i)} \d r\, \ind{G(N,\kappa,\lambda)} \\
& \qquad \leq K \int_0^{\sqrt{\kappa}} \sqrt{\log \lceil \sqrt{\kappa} / r \rceil} \d r \leq K \sqrt{\kappa} \int_0^{1} \sqrt{\log \lceil 1/x \rceil} \d x
,\end{align*}
and the integral is some universal constant as well. Let us write $L \sqrt{\kappa}$ this upper bound.
Recall the exponent $a < 1/\bdeta$ which we have fixed. If we were able to take $a > 1/2$ then we could conclude by the Markov inequality. However in the case $\bdeta = 2$, in particular in the Brownian case, we need a slightly better bound. A very strong bound~\cite[Theorem 5.8]{BLM13} which follows from standard Gaussian concentration results is the following: for every $u>0$, we have
\begin{align*}
&\P\Bigl(G(N,\kappa,\lambda) \cap \Bigl\{\sup_{s \in [0, \frac{1}{N}]} |Z^{(\lambda, \frac{1}{3})}_{s+i/N} - Z^{(\lambda, \frac{1}{3})}_{i/N}| > u + \E\Bigl[\sup_{s \in [0, \frac{1}{N}]} |Z^{(\lambda, \frac{1}{3})}_{s+i/N} - Z^{(\lambda, \frac{1}{3})}_{i/N}| \Bigm| X^{(\lambda)}\Bigr]\Bigr\}\Bigr)
\\
&\leq \E\Bigl[\ind{G(N,\kappa,\lambda)} \exp\Bigl(- \frac{u^2}{2 \sup_{s \in [0, 1/N]} \E[(Z^{(\lambda, 1/3)}_{s+i/N} - Z^{(\lambda, 1/3)}_{i/N})^2]}\Bigr)\Bigr]
\\
&\leq \exp\Bigl(- \frac{u^2}{2 \kappa}\Bigr)
.\end{align*}
Now recall that the expectation in the first line is bounded above by $L \sqrt{\kappa}$ where $L$ is a universal constant, hence
\[\P\Bigl(G(N,\kappa,\lambda) \cap \Bigl\{\sup_{s \in [0, 1/N]} |Z^{(\lambda, 1/3)}_{s+i/N} - Z^{(\lambda, 1/3)}_{i/N}| > u + L \sqrt{\kappa}\Bigr\}\Bigr)
\leq \exp\Bigl(- \frac{u^2}{2 \kappa}\Bigr)
,\]
for every $u>0$. Combining all the previous displays, we infer with $u = (C/4-L) \sqrt{\kappa}$ that for every $\lambda$ large enough, it holds:
\[\Pr{\exists (s,t) \in [0,1]^2 \colon |t-s| \leq \frac{1}{N} \,\text{ and } \, |Z^{(\lambda, \frac{1}{3})}_s - Z^{(\lambda, \frac{1}{3})}_t| > \theta}
\leq 2\varepsilon + N \exp\Bigl(- \frac{(\frac{C}{4}-L)^2}{2}\Bigr).\]
Recall that $\theta = C \sqrt{\kappa}$ and $N = \lceil (M/\kappa)^{1/a} \rceil = \lceil (MC^2 / \theta^2)^{1/a} \rceil$ where $M$ was chosen only in terms of $\varepsilon$ in order to get the upper bound $2\varepsilon$. Since $L$ is a universal constant, then we can choose $C$ depending only on $\theta$ and $M$ (and thus $\varepsilon$) such that the right-hand side above is less than $3\varepsilon$. This fixes the value of $N$ in terms of $\theta$ and $\varepsilon$ only and the proof of tightness is complete.
\end{proof}

To complete the proof of Theorem~\ref{thm:convergence_Levy_looptree_label}, it remains to characterise the subsequential limits. We first consider the convergence under the unconditional law of L\'evy processes.

\begin{proof}[Proof of convergence for unconditioned processes]
Let us first consider the unconditioned L\'evy processes and let us prove that for any $t>0$ fixed, the looptree distance from $0$ to $t$ for $X^{(\lambda)}$ converges in distribution to that for $X$. Then we discuss the convergence of the distance between any $s$ and $t$ fixed, the argument for more general finite dimensional distributions is the same. 
We then briefly mention the modifications to control the labels.

\textsc{Step 1: distance from $0$ to $t$.}
The distance from $0$ to $t$ in the looptree coded by $X$ equals
\[\frac{1}{2} C_t + \sum_{s \prec t} \min(R^t_s, \Delta X_s-R^t_s).\]
Let us denote by $C^{(\lambda)}_t$ and $R^{(\lambda), t}_s$ respectively the continuous part and jump part that appear similarly in the looptree distance associated with $X^{(\lambda)}$.
We shall argue that, when $\delta>0$ is small, the following terms
\[\sum_{s \prec t} \min(R^{(\lambda), t}_s, \Delta X^{(\lambda)}_s-R^{(\lambda), t}_s) \ind{\Delta X^{(\lambda)}_s\leq \delta}
\qquad\text{and}\qquad
\frac{1}{2} \sum_{s \prec t} R^{(\lambda), t}_s \ind{\Delta X^{(\lambda)}_s\leq \delta}\]
are close to each other, and that the second one will contribute, together with $\frac{1}{2} C^{(\lambda)}_t$, to the continuous part $\frac{1}{2} C_t$ in the limit as $\lambda\to \infty$ and $\delta \to 0$.
We shall implicitly restrict ourselves to values of $\delta$ which are not atoms of the L\'evy measure $\pi$. This set of atoms is at most countable, so we can indeed find a sequence outside of it that tends to $0$. Henceforth, we thus assume that $X$ has no jump of size exactly $\delta$.

Recall that $C_t = X_{t-} - \sum_{s \prec t} R^t_s$, and the same holds for $X^{(\lambda)}$.
Standard properties of the Skorokhod $J_{1}$ topology imply that, since $X^{(\lambda)}$ converges in distribution to $X$, then $X^{(\lambda)}_{t-} \to X_{t-}$ and, jointly, the (finitely many) jumps larger than $\delta$ of $X$ all arise as limits of jumps larger than $\delta$ of $X^{(\lambda)}$ (without boundary effect since $X$ has no jump of size exactly $\delta$), and the times at which they take place converge, and so do the associated quantities $R^{(\lambda), t}_s$.
As a consequence,
\begin{eqnarray*}
C^{(\lambda)}_t + \sum_{s \prec t} R^{(\lambda), t}_s \ind{\Delta X^{(\lambda)}_s\leq \delta}
&=& X^{(\lambda)}_{t-} - \sum_{s \prec t} R^{(\lambda), t}_s \ind{\Delta X^{(\lambda)}_s>\delta}\\
&\displaystyle\cvloi[\lambda]& X_{t-} - \sum_{s \prec t} R^t_s \ind{\Delta X_s>\delta}
= C_{t} + \sum_{s \prec t} R^t_s \ind{\Delta X_s\leq \delta}
,
\end{eqnarray*}
and the limit further converges in probability to $C_t$ as $\delta\to0$.
Hence, to conclude about the convergence of the looptree distance from $0$ to $t$, it remains to compare this distance true distance, in $X^{(\lambda)}$, with the quantity we obtain when we replace the term $\sum_{s \prec t} \min(R^{(\lambda), t}_s, \Delta X^{(\lambda)}_s-R^{(\lambda), t}_s) \ind{\Delta X^{(\lambda)}_s\leq \delta}$ by $\frac{1}{2} \sum_{s \prec t} R^{(\lambda), t}_s \ind{\Delta X^{(\lambda)}_s\leq \delta}$.

Precisely, let us fix a bounded Lipschitz function $F$, with
$\sup_x |F(x)| \le K$ and $\sup_{x \ne y} |F(x)-F(y)|/(|x-y|) \le L$ for some positive constants $K$ and $L$,
and define
\begin{align*}
&D_F(\lambda, \delta)
\\
&= \Bigl|F\Bigl(\frac{1}{2} \Bigl(C^{(\lambda)}_t + \sum_{s \prec t} R^{(\lambda), t}_s \ind{\Delta X^{(\lambda)}_s\leq \delta}\Bigr) + \sum_{s \prec t} \min(R^{(\lambda), t}_s, \Delta X^{(\lambda)}_s-R^{(\lambda), t}_s) \ind{\Delta X^{(\lambda)}_s > \delta}\Bigr)
\\
&\quad- F\Bigl(\frac{1}{2} C^{(\lambda)}_t + \sum_{s \prec t} \min(R^{(\lambda), t}_s, \Delta X^{(\lambda)}_s-R^{(\lambda), t}_s)\Bigr)\Bigr|.
\end{align*}
Let us prove that $\E[D_F(\lambda, \delta)] \to 0$ as first $\lambda \to \infty$ and then $\delta \to 0$, then our claim follows from the Portmanteau theorem.

Let us adapt the argument in the proof of Theorem~1.2 in~\cite{CK14}.
Precisely, we can write $R^{(\lambda),t}_s = \Delta X^{(\lambda)}_s \times u^t_s$, where 
in the unconditioned process, the random variables $(u^t_s ; s \prec t)$ are i.i.d.~with the uniform law on $[0,1]$ and are independent from $(\Delta X^{(\lambda)}_s ; s \prec t)$. Consequently,
\[\E\biggl[\frac{\sum_{s \prec t} \Delta X^{(\lambda)}_s \ind{\Delta X^{(\lambda)}_s \leq \delta} u^t_s}{\sum_{s \prec t} \Delta X^{(\lambda)}_s \ind{\Delta X^{(\lambda)}_s \leq \delta}} \bigm| \Delta X^{(\lambda)}_s ; s \prec t \biggr] = \frac{1}{2}\]
and
\begin{align*}
\Var\biggl(\frac{\sum_{s \prec t} \Delta X^{(\lambda)}_s \ind{\Delta X^{(\lambda)}_s \leq \delta} u^t_s}{\sum_{s \prec t} \Delta X^{(\lambda)}_s \ind{\Delta X^{(\lambda)}_s \leq \delta}}\bigm| \Delta X^{(\lambda)}_s ; s \prec t \biggr) 
&= \frac{\sum_{s \prec t} (\Delta X^{(\lambda)}_s \ind{\Delta X^{(\lambda)}_s \leq \delta})^2}{12 (\sum_{s \prec t} \Delta X^{(\lambda)}_s \ind{\Delta X^{(\lambda)}_s \leq \delta})^2}\\
&\leq \frac{\delta}{12 \sum_{s \prec t} \Delta X^{(\lambda)}_s \ind{\Delta X^{(\lambda)}_s \leq \delta}},
\end{align*}
after bounding the sum of the squares by the sum without square times the largest element.
The same holds if we replace $u^t_s$ by $\min(u^t_s, 1-u^t_s)$, except that the expectation is now $1/4$ and the variance is $1/48$.

Fix $\lambda, \delta > 0$ and $\varepsilon \in (0, 1/4)$ and consider the following two events.
First
\[A_{\lambda, \delta, \varepsilon} = \Bigl\{\sum_{s \prec t} \Delta X^{(\lambda)}_s \ind{\Delta X^{(\lambda)}_s \leq \delta} \leq \varepsilon\Bigr\}
\]
controls the denominator above, so on the complement event, the variances are bounded above by $\delta/(12 \varepsilon)$ and $\delta/(48 \varepsilon)$ respectively. 
Second,
\[B_{\lambda, \delta, \varepsilon} = \biggl\{\biggl|\frac{\sum_{s \prec t} \min(R^{(\lambda), t}_s, \Delta X^{(\lambda)}_s-R^{(\lambda), t}_s) \ind{\Delta X^{(\lambda)}_s\leq \delta}}{\sum_{s \prec t} R^{(\lambda), t}_s \ind{\Delta X^{(\lambda)}_s\leq \delta}} - \frac{1}{2}\biggr| \le \frac{3\varepsilon}{1-2\varepsilon}\biggr\},\]
controls the ratios.
It is straightforward to check that if $|a-1/4| \le \varepsilon$ and $|b - 1/2| \le \varepsilon$, then 
$|a/b - 1/2| 
\le 
\max\{(1/4 + \varepsilon)/(1/2 - \varepsilon) - 1/2, 1/2 - (1/4 - \varepsilon)/(1/2 + \varepsilon)\} = 
3 \varepsilon / (1-2\varepsilon)$.
A union bound and the Chebychev inequality then easily show that, for some universal constant $V$ (for example $5/24$) coming from the variances, it holds:
\[\P(A_{\lambda, \delta, \varepsilon}^c \cap B_{\lambda, \delta, \varepsilon}^c)
\leq \frac{\delta V}{\varepsilon^3}.\]
Let us introduce a third and last event:
\[C_{\lambda, \delta, \varepsilon} = \Bigl\{\sum_{s \prec t} R^{(\lambda), t}_s \ind{\Delta X^{(\lambda)}_s\leq \delta} < \varepsilon^{-1/2}\Bigr\}
,\]
whose probability tends to $1$ as first $\lambda\to\infty$, then $\delta\to0$, and finally $\varepsilon\to0$ since this sum is upper bounded by $X^{(\lambda)}_{t-}$ which converges in distribution to $X_{t-}$ which is finite.

Recall the quantity $D_F(\lambda, \delta)$ and that we aim at showing that its expectation converges to $0$ as first $\lambda\to\infty$, then $\delta\to0$.
We shall partition the space into the four events: 
$A_{\lambda, \delta, \varepsilon} 
\cup \{A_{\lambda, \delta, \varepsilon}^c \cap B_{\lambda, \delta, \varepsilon} \cap C_{\lambda, \delta, \varepsilon}\} 
\cup \{A_{\lambda, \delta, \varepsilon}^c \cap B_{\lambda, \delta, \varepsilon}^c \cap C_{\lambda, \delta, \varepsilon}\} 
\cup \{A_{\lambda, \delta, \varepsilon}^c \cap C_{\lambda, \delta, \varepsilon}^c\} 
$.
Indeed, first, since $F$ is $L$-Lipschitz, then 
\[D_F(\lambda, \delta)
\le L \Bigl|\sum_{s \prec t} \min(R^{(\lambda), t}_s, \Delta X^{(\lambda)}_s-R^{(\lambda), t}_s) \ind{\Delta X^{(\lambda)}_s \leq \delta} - \frac{1}{2} \sum_{s \prec t} R^{(\lambda), t}_s \ind{\Delta X^{(\lambda)}_s \leq \delta}\Bigr|
.\]
Using the triangle inequality, this is bounded above by $3 L \varepsilon / 2$ on the event $A_{\lambda, \delta, \varepsilon}$.
If instead we factorise by the sum of the $R$'s, then on the event $B_{\lambda, \delta, \varepsilon}$, we have
\[D_F(\lambda, \delta)
\le \frac{3L \varepsilon}{1-2\varepsilon} \sum_{s \prec t} R^{(\lambda), t}_s \ind{\Delta X^{(\lambda)}_s \leq \delta}
,\]
which is now bounded above by $3L \sqrt{\varepsilon} / (1-2\varepsilon)$ on the event $C_{\lambda, \delta, \varepsilon}$.
For the last two cases, we simply use that $F$ is bounded by $K$, so $D_F(\lambda, \delta) \le 2K$ is uniformly bounded.
Therefore, we have:
\[\E[D_F(\lambda, \delta)]
\le \frac{3 L \varepsilon}{2} + \frac{3L \sqrt{\varepsilon}}{1-2\varepsilon} + 2K \P(A_{\lambda, \delta, \varepsilon}^c \cap B_{\lambda, \delta, \varepsilon}^c) + 2K \P(C_{\lambda, \delta, \varepsilon}^c).\]
We know that the first probability on the right is bounded above by some universal constant times $\delta / \varepsilon^3$ and that the second probability converges to $0$ as first $\lambda\to\infty$, then $\delta\to0$, and finally $\varepsilon\to0$. The same is then true for $\E[D_F(\lambda, \delta)]$ and the proof is complete.

\textsc{Step 2: distance from an ancestor to $t$.}
So far, we have proved the convergence of the looptree distance from $0$ to any given $t>0$. Let $s \in (0,t)$ and suppose first that $s \prec t$ in the looptree coded by $X$, namely that $X_{s-} \leq X_r$ for all $r \in [s,t]$. 
There are two cases according as whether or not $X$ jumps at time $s$.

If $\Delta X_s>0$, then recall from the proof of Lemma~\ref{lem:identifications_looptree_Levy} that $X$ cannot make a local minimum at the height $X_{s-}$. In particular, either $X_t>X_{s-}$ or $t = \inf\{r>s \colon X_r = X_{s-}\}$. In the second case, this means that we can find $t'<t$ arbitrarily close to $t$, with $X_{t'}>X_{t}=X_{s-}$ and $X_{t'}$ is arbitrarily close to $X_t$; this implies that the looptree distance between $t$ and $t'$ is arbitrarily small and $t'$ falls into the first case. Therefore we may assume when $\Delta X_s>0$ that $X_t>X_{s-}$ and actually $X_r > X_{s-}$ for all $r \in [s,t]$. 
In this case, as 
in the previous step, there are times $s^{(\lambda)} \prec t^{(\lambda)}$ 
such that $X^{(\lambda)}$ jumps at time $s^{(\lambda)}$, and such that $s$, $t$, $\Delta X_s$ and $R^t_s$ are the limits of the analogous quantities in $X^{(\lambda)}$ and the previous argument extends readily to prove the convergence of the looptree distance from $s^{(\lambda)}$ to $t^{(\lambda)}$ for $X^{(\lambda)}$ to that from $s$ to $t$ for $X$.

Suppose next that $\Delta X_s = 0$. We aim at finding similarly times $s^{(\lambda)} \prec t^{(\lambda)}$ that converge to $s$ and $t$ to which we can apply the previous argument.
Again there are two cases.
In the case where $X_r>X_s$ for all $r \in (s,t]$, there exists $\varepsilon>0$ small such that $X_r > X_s+2\varepsilon$ for $r \in (s+2\varepsilon,t]$. The Skorokhod convergence implies that for all $\lambda$ sufficiently large, there exists $r^{(\lambda)} \in (s-\varepsilon, s+\varepsilon)$ such that $X^{(\lambda)}_r > X^{(\lambda)}_{r^{(\lambda)}}+\varepsilon$ for all $r \in (r^{(\lambda)}+4\varepsilon,t]$. This further implies this existence of some time $s^{(\lambda)} \in (r^{(\lambda)}, r^{(\lambda)}+4\varepsilon) \subset (s-\varepsilon, s+5\varepsilon)$ with $s^{(\lambda)} \prec t$ and we conclude in this case.
A last case is when both $\Delta X_s = 0$ and $X_r=X_s$ for some $r \in (s,t)$. Then necessarily this time $r$ is a time of local minimum and since the latter are unique, then $s$ is not a time of local minimum so we can find $s'<s$ arbitrarily close to $s$ which falls into the previous case.

\textsc{Step 3: pairwise distances.}
Assume now that $s \not\prec t$ in the looptree coded by $X$. Then we can cut at their last common ancestor $s \wedge t$ and apply the same reasoning as above: if this last common ancestor is a jump time, then the quantities of interest associated with this macroscopic loop are limits of those for $X^{(\lambda)}$ and the two branches from $s \wedge t$ to $s$ on the one hand and from $s \wedge t$ to $t$ on the other hand converge by the previous paragraph.
The case when $s \wedge t$ is not a jump time is controlled similarly.
The same argument extends to control the pairwise distances between any finite collection of times, cutting at the corresponding branch-points.
Details are left to the reader.

\textsc{Step 4: labels.}
Let us end with a few words on the Gaussian labels on the looptree. Recall that the conditional covariance function given the L\'evy process is very close to be the looptree distance: one just replaces the infimum between the left and right length of the loops by their product. Thus the only modification in Step 2 is that one should replace $\min(R^{(\lambda), t}_s, \Delta X^{(\lambda)}_s-R^{(\lambda), t}_s)$ by $R^{(\lambda), t}_s (\Delta X^{(\lambda)}_s-R^{(\lambda), t}_s)$. If $U$ has the uniform law on $[0,1]$, we used previously that $\E[\min(U, 1-U)] = 1/4$, then this value is now replaced by $\E[U (1-U)] = 1/6$. Once divided by $\E[U]=1/2$, we obtain the constant $1/3$, which explains why we consider $Z^{1/3}$ where we considered $\Loop^{1/2}(X)$.
Step 2 and 3 are easily adapted.
\end{proof}

We finally turn to conditioned processes.

\begin{proof}[Proof of convergence for bridges and excursions]
We know from the two previous proofs that the looptree distances and labels associated with the bridges are tight and those associated with the unconditioned processes converge. 
Fix $\varepsilon \in (0,1)$ and notice that the looptree distance between all pairs $s,t \in [0,1-\varepsilon]$ is a measurable function of the trajectory of the process restricted to the time interval $[0,1-\varepsilon]$. The same holds for the covariance function of the labels on the time interval $[0,1-\varepsilon]$.
Then the convergence of these looptree distances and labels for the bridge follows from the absolute continuity relation as in the proof of Lemma~\ref{lem:cv_densite_excursion}. This characterises the subsequential limits and we conclude from tightness.

We next turn to the excursion versions. Recall that we recover the bridge $X^{\br}$ by cyclicly shifting the excursion $X^{\exc}$ at an independent uniform random time $U$. 
Notice that the looptree distance in the excursion between any two times $s, t \in [0,U]$ is the same as the looptree distance in the bridge between their images $1-U+s$ and $1-U+t$ after the cyclic shift. However they differ if one or both times lie after $U$.
Using the same time $U$ to relate the excursion and bridge versions of each $X^{(\lambda)}$, we infer from the case of bridges that the looptree distances restricted to $[0,U]$ in the excursions converge in distribution to those in $X^{\exc}$.

Assume next that $s < U < t$ and consider $q = \inf\{r > U \colon X^{\exc}_{r-} \leq \inf_{[r,t]} X^{\exc}\}$ the first ancestor of $t$ that lies after $U$ and $p = \sup\{r < U \colon X^{\exc}_{r-} \leq X^{\exc}_{q-}\}$ the last ancestor of $t$ before $U$. Then the looptree distance between $t$ and $q$ in the excursion is the same as that of their images $t-U$ and $q-U$ in the bridge, and all the other contributions to the distance between $s$ and $t$ depend only on the trajectory up to time $U$ as well as the value $X^{\exc}_U - X^{\exc}_{q-}$. 
Finally if both $s, t > U$, then it may be that their last common ancestor also satisfies $s \wedge t \geq U$, in which case the looptree distance in the excursion is that of their images in the bridge. If $s \wedge t \geq U$, then we easily adapt the previous reasoning.
In any case we infer the convergence in distribution of the looptree distances in the excursion from the joint convergence of the bridges and the associated looptree distances.

The same reasoning shows the convergence of the covariance function of the label processes, which combined with the tightness argument implies the convergence in distribution.
\end{proof}

\subsection{Convergence of maps}
\label{ssec:convergence_cartes_Levy}

Let us suppose that for every $\lambda>0$, the pair $(X^{(\lambda)}, Z^{(\lambda, 1/3)})$ arises as the scaling limit of discrete paths $(W^{(\lambda, n)}, Z^{(\lambda, n)})$ in the sense of Theorem~\ref{thm:convergence_looptrees_labels_Levy}. Suppose that these paths encode a bipartite planar map $\Map^{(\lambda, n)} = (\Map^{(\lambda, n)}, d^{(\lambda, n)}, p^{(\lambda, n)})$. Then Theorem~\ref{thm:convergence_cartes_Boltzmann_Levy} implies the convergence of the rescaled map along a subsequence to a limit space $\Map^{(\lambda)} = (\Map^{(\lambda)}, d^{(\lambda)}, p^{(\lambda)})$.
Now assume that for every $t \geq 0$, the transition density $p^{(\lambda)}_t$ converges uniformly as $\lambda \to \infty$ to $p_t$, so both Lemma~\ref{lem:cv_densite_excursion} and Theorem~\ref{thm:convergence_Levy_looptree_label} are satisfied, and let us study the convergence of $\Map^{(\lambda)}$ as $\lambda\to\infty$, again along a subsequence.

\begin{thm}\label{thm:convergence_Levy_cartes}
Suppose that the previous assumptions are satisfied.
Then from every sequence of integers tending to infinity, one can extract a subsequence along which the convergence in distribution
\[d^{(\lambda)} \cvloi[\lambda] D,\]
holds for the uniform topology on $[0,1]^2$, jointly with Theorem~\ref{thm:convergence_Levy_looptree_label}.
The limit $D$ is a random continuous pseudo-distance which satisfies the identity in law:
\[(D(U,t))_{t \in [0,1]} \eqloi (Z^{\exc}_t - \min Z^{\exc})_{t \in [0,1]},\]
where $U$ is sampled uniformly at random on $[0,1]$ and independently of the rest.
It also satisfies almost surely the property that for every $s,t \in [0,1]$, 
\[\text{if}\enskip d_{\Loop(X)}(s,t) = 0, \enskip\text{then}\enskip D(s,t) = 0.\]
Furthermore, along any such convergent subsequence, we have
\[\Map^{(\lambda)} \cvloi{} \Map = [0,1]/\{D=0\}\]
in the Gromov--Hausdorff--Prokhorov topology. 
Finally, if the excursion of $X$ reduces to a scaled Brownian excursion or stable excursion, then the maps converge in distribution without extraction to a multiple of the Brownian sphere or a stable map respectively.
\end{thm}

\begin{proof}
The argument is very close to that of Theorem~\ref{thm:convergence_cartes_Boltzmann_Levy}.
Indeed, from this theorem, for every $\lambda$, we have the identity in law $d^{(\lambda)}(U,\cdot) = Z^{(\lambda,1/3)} - \min Z^{(\lambda,1/3)}$ where $U$ is sampled uniformly at random on $[0,1]$ and independently of the rest. Since $Z^{(\lambda,1/3)}$ converges in distribution to $Z^{1/3}$ by Theorem~\ref{thm:convergence_Levy_looptree_label}, then tightness of $d^{(\lambda)}$ follows just as in the discrete setting, and so does the identity in law $D(U,\cdot) = Z^{1/3} - \min Z^{1/3}$ for any subsequential limit $D$.
Lemma~\ref{lem:distance_map_on_looptree} shows that for every $\lambda$, the pseudo-distance $d^{(\lambda)}$ can be seen as a pseudo-distance on the looptree $\Loop(X^{(\lambda)})$, in the sense that $d_{\Loop(X^{(\lambda)})}=0$ implies $d^{(\lambda)}=0$.
The same argument used to prove this lemma, using the approximation $X^{(\lambda)} \to X$ instead of discrete paths, shows further that this passes to any subsequential limit, namely that $d_{\Loop(X)}=0$ implies $D=0$.
Then the convergence in the Brownian or stable case follows just as for discrete maps in Theorem~\ref{thm:convergence_cartes_Boltzmann_Levy}.
\end{proof}

\phantomsection
\addcontentsline{toc}{section}{References}

{\small
\bibliographystyle{alpha}

}

\end{document}